\documentclass[11pt]{amsart}
\usepackage{tabularx,booktabs}

\usepackage{amsmath,tikz}
\usepackage{amsfonts}

\usepackage{amscd}
\usepackage{amsthm}
\usepackage{amssymb} \usepackage{latexsym}
\usepackage{eufrak}
\usepackage{euscript}
\usepackage{epsfig}
\usepackage{graphics}
\usepackage{array}
\usepackage{enumerate}
\usepackage{dsfont}
\usepackage{color}
\usepackage{wasysym}
\usepackage{hyperref}
\usepackage{pdfsync}
\usepackage[linesnumbered,boxed,vlined,ruled]{algorithm2e}

\def \d {\partial}
\DeclareMathOperator{\Pttn}{Pttn}

\newcommand{\bel}[1]{\begin{equation}\label{#1}}

\newcommand{\be}{\begin{equation}}

\newcommand{\qe}{\end{equation}}
\newcommand{\R}{{\mathbb R}}

\newcommand{\wt}{\widetilde}
\newcommand{\CC}{\mathcal{C}_{\geq 0}}
\newcommand{\MM}{\mathcal{M}_{\geq 0}}

\newcommand{\Hmm}[1]{\leavevmode{\marginpar{\tiny%
$\hbox to 0mm{\hspace*{-0.5mm}$\leftarrow$\hss}%
\vcenter{\vrule depth 0.1mm height 0.1mm width \the\marginparwidth}%
\hbox to
0mm{\hss$\rightarrow$\hspace*{-0.5mm}}$\\\relax\raggedright #1}}}

\newcommand{\curvature}[1]{
 \ifnum #1=1
 5.74238\cdots\times 10^{-5}\else{%
 \ifnum #1=2
 1.83701\cdots\times 10^{-5}\else{%
 \ifnum #1=3
 7.85456\cdots\times 10^{-6}\else{%
 \ifnum #1=4
 3.49781\cdots \times10^{-6}\else{%
 \ifnum #1=5
 1.46541\cdots\times 10^{-6}\else{%
 \ifnum #1=6
 1.38186\cdots\times 10^{-8}\else{%
 \ifnum #1=7
 1.12992\cdots\times 10^{-8}\else{
 \ifnum #1=8
 4.70209\cdots\times 10^{-9} \else{
 \ifnum #1=9
 1.30776\cdots\times 10^{-9} \else{
 \ifnum #1=10
4.96239\cdots\times 10^{-11}
\fi}
\fi}
\fi}
\fi}
\fi}
\fi}
\fi}
\fi}
\fi}
\fi}

\newcommand\generalcurv[1]{
\ifnum #1>54
\curvature{10} \else{
\ifnum #1>44
\curvature{9} \else{
\ifnum #1>40
\curvature{8}  \else{
\ifnum #1>34
\curvature{7}\else{
\ifnum #1>20
\curvature{6} \else{
\ifnum #1>18
\curvature{5} \else{
\ifnum #1>16
\curvature{4} \else{
\ifnum #1>15
\curvature{3} \else{
\ifnum #1>12
\curvature{2} \else
\curvature{1}
\fi}\fi}\fi}\fi}\fi}\fi}\fi}\fi}\fi}

\newcommand\generalp[1]{
\ifnum #1>54
(29,55,55) \else{
\ifnum #1>44
(35,38,43) \else{
\ifnum #1>40
(9,36,41)  \else{
\ifnum #1>34
(15,18,19) \else{
\ifnum #1>20
(16,16,21) \else{
\ifnum #1>18
(11,16,16) \else{
\ifnum #1>16
(11,13,15) \else{
\ifnum #1>15
(12,12,16) \else{
\ifnum #1>12
(8,10,13) \else
(6,7,11)
\fi}\fi}\fi}\fi}\fi}\fi}\fi}\fi}\fi}

\newcommand\generalq[1]{
\ifnum #1>54
(32,43,55) \else{
\ifnum #1>44
(31,44,45) \else{
\ifnum #1>40
(10,23,26) \else{
\ifnum #1>34
(11,26,35) \else{
\ifnum #1>20
(14,20,20) \else{
\ifnum #1>18
(10,17,19) \else{
\ifnum #1>16
(11,12,17) \else{
\ifnum #1>15
(11,14,15) \else{
\ifnum #1>12
(7,13,13) \else
(6,8,9)
\fi}\fi}\fi}\fi}\fi}\fi}\fi}\fi}\fi}

\newcommand{\pt}{\mathrm{Pttn}}

\newcommand{\NNG}{\mathcal{PC}_{\geq0}}

\newtheorem{theorem}{Theorem}[section]
\newtheorem{lemma}[theorem]{Lemma}
\newtheorem{corollary}[theorem]{Corollary}
\newtheorem{definition}[theorem]{Definition}

\newtheorem{remark}[theorem]{Remark}
\newtheorem{prop}[theorem]{Proposition}
\newtheorem{problem}[theorem]{Problem}
\newtheorem{example}[theorem]{Example}

\begin{document}

\title[A curvature notion for planar graphs stable under planar duality]{A curvature notion for planar graphs stable under planar duality}

\author{Yohji Akama}
\email{yoji.akama.e8@tohoku.ac.jp}
\address{Yohji Akama: Mathematical Institute, Graduate School of Science, Tohoku University,
Sendai, 980-0845, Japan}

\author{Bobo Hua}
\email{bobohua@fudan.edu.cn}
\address{Bobo Hua: School of Mathematical Sciences, LMNS, Fudan University, Shanghai 200433, China; Shanghai Center for Mathematical Sciences, Fudan University, Shanghai 200433, China}

\author{Yanhui Su}
\email{suyh@fzu.edu.cn}
\address{Yanhui Su: College of Mathematics and Computer Science, Fuzhou University, Fuzhou 350116, China}

\author{Lili Wang}
\email{lili\_wang@fudan.edu.cn}
\address{Lili Wang: School of Mathematical Sciences, Fudan University, Shanghai 200433, China}

\begin{abstract}

Woess \cite{Woess98} introduced a curvature notion on the set of edges of a planar graph, called $\Psi$-curvature in our paper, which is stable under the planar duality. We study geometric and combinatorial properties for the class of infinite planar graphs with non-negative $\Psi$-curvature. By using the discharging method, we prove that for such an infinite graph the number of vertices (resp. faces) of degree $k,$ except $k=3,4$ or $6,$ is finite.
As a main result, we prove that for an infinite planar graph with non-negative $\Psi$-curvature the sum of the number of vertices of degree at least $8$ and the number of faces of degree at least $8$ is at most one.





\end{abstract}
\maketitle\tableofcontents

\section{Introduction}\label{sec:intro}

The combinatorial curvature for planar graphs was introduced by \cite{MR0279280,MR0410602,MR919829,Ishida90}, which is a discrete analog of the Gaussian curvature of surfaces. It has been intensively studied by many authors, see e.g.
\cite{MR1600371,Woess98,MR1864922,BP01,MR1894115,MR1923955,
MR2038013,MR2096789,RBK05,MR2243299,MR2299456,
MR2410938,MR2466966,MR2470818,MR2558886,MR2818734,MR2826967,MR3624614,Gh17}.

Let $(V,E)$ be a locally finite, undirected, connected graph with the set of vertices $V$ and the set of edges $E.$
We always assume $(V,E)$ is simple, i.e. it has no selfloops no multiple edges. Two vertices $x,y$ are called adjacent if $\{x,y\}\in E,$ denoted by $x\sim y.$ The graph $(V,E)$ is called \emph{planar} if it can be topologically embedded into the sphere $\mathbb S^2$ or the plane $\mathbb R^2$. We write $G=(V,E,F)$ for the combinatorial structure, or the cell complex, induced by the embedding where $F$ is the set of faces, i.e. connected components of the complement of the embedding image of the graph $(V,E)$ in the target.
Two elements in $V,E,F$ are called \emph{incident} if the closures of their images have non-empty intersection. We call the closure $\bar{\sigma}$ of a face $\sigma$ a closed face.
A planar graph $G$ is called a (planar) \emph{tessellation} if the following conditions hold, see e.g. \cite{MR2826967}:
\begin{enumerate}[(i)]
\item Any closed face is a closed disk whose boundary consists of finitely many edges of the graph.
\item Any edge is contained in exactly two different closed faces.
\item The intersection of two closed faces is an empty set, a vertex or the closure of an edge.
\end{enumerate}
In this paper, we call tessellations planar graphs for simplicity and assume that for any vertex $x$ and any face $\sigma$
$$|x|\geq 3,\ |\sigma|\geq 3, $$ where $|\cdot|$ denotes the degree of a vertex or the degree of a face, i.e. the number of edges incident to the face.
For a planar graph $G$, the \emph{combinatorial curvature} at the
vertex is defined as
\begin{equation}\label{def:comb}\Phi(x)=1-\frac{|x|}{2}+\sum_{\sigma\in
F \colon x\in \overline{\sigma}}\frac{1}{|\sigma|},\quad  x\in V,\end{equation} where the summation is taken over all faces $\sigma$ incident to $x.$

 To digest the definition, we endow the ambient space of $G$ with a canonical piecewise flat metric and call it the (regular) \emph{Euclidean polyhedral surface}, denoted by $S(G)$:
 replace each face by a regular Euclidean polygon of same facial degree and of side length one, glue them together along their common edges, and define the metric on the ambient space via gluing methods, see \cite[Chapter~3]{MR1835418}. It is well-known that the generalized Gaussian curvature on an Euclidean polyhedral surface, as a measure, concentrates on the vertices. One can show that the combinatorial curvature at a vertex is in fact the mass of the generalized Gaussian curvature at that vertex up to the normalization $2\pi,$ see e.g. \cite{MR2127379,MR3318509}.


We denote by $\NNG$ the class of planar graphs with non-negative combinatorial curvature. This class of planar graphs has been studied by \cite{MR1864922, MR2299456, MR2410938, MR3318509, HJ15}.
As is well-known, $G\in\NNG$ if and only if $S(G)$ has non-negative curvature in the sense of Alexandrov, i.e. it is a (possibly non-smooth) convex surface, see \cite{MR3318509, MR1835418}. For a planar graph $G,$ we denote by $$\Phi(G):=\sum_{x\in V}\Phi(x)$$ the total curvature of $G.$ For a finite planar graph $G,$ i.e. embedded into $\mathbb{S}^2,$ the (discrete) Gauss-Bonnet theorem states that $\Phi(G)=2.$ For an infinite planar graph $G\in \NNG,$ DeVos and Mohar \cite{MR2299456} proved a discrete analog of Cohn-Vossen theorem in the Riemannian geometry.
\begin{theorem}[\cite{MR2299456}]\label{thm:DeVosMohar} For any infinite $G\in \NNG,$
\begin{equation}\label{CVthm}\Phi(G)\leq 1.\end{equation}
\end{theorem}

For a planar graph $G$ with an embedding, one can define the dual graph $G^*$: the vertices of $G^*$ are corresponding to faces of $G,$ the faces of $G^*$ are corresponding to vertices of $G,$ and two vertices in $G^*$ are adjacent if and only if the corresponding faces in $G$ share a common edge. One can show that $G$ is a tessellation if and only if so is $G^*,$ see e.g. \cite{MR2243299}. For $G\in \NNG,$ the dual graph $G^*$ may not have non-negative combinatorial curvature, e.g. the boundary of the truncated cube, one of the Archimedean solids. This means that the dual operation doesn't preserve the class $\NNG.$

Baues and Peyerimhoff \cite{MR2243299} introduced the so-called corner curvature on planar graphs. A corner of $G=(V,E,F)$ is a pair $(x,\sigma)$ such that $x\in V,$ $\sigma\in F$ and $x$ is incident to $\sigma.$ The corner curvature of a corner $(x,\sigma)$ is defined as
$$C(x,\sigma):=\frac{1}{|x|}+\frac{1}{|\sigma|}-\frac12.$$ We define the class of planar graphs with non-negative corner curvature as $$\CC:=\{G: C(x,\sigma)\geq 0, \forall\ \mathrm{corner}\ (x,\sigma)\}.$$ One can show that $\CC\subset \NNG,$ and $G\in \CC$ if and only if $G^*\in \CC,$ see e.g. \cite{MR2243299}. Although the class $\CC$ has nice properties, it is somehow restrictive.


Woess \cite[pp. 387]{Woess98} introduced an interesting curvature notion on edges. 
\begin{definition}\label{Mcurvature} For any edge $e,$ let $x_1,x_2$ be end-vertices of $e$ and $\sigma_1,\sigma_2$ be two faces whose closures contain $e,$ see Figure~\ref{Fig:1}.
We define the $\Psi$-curvature of the edge $e$ as
$$\Psi(e)=\frac{1}{|x_1|}+\frac{1}{|x_2|}+\frac{1}{|\sigma_1|}+\frac{1}{|\sigma_2|}-1.$$
\end{definition}
\begin{remark}Woess called it $\phi$ curvature, which is same as the above notion up to a sign.
\end{remark}
We denote by $$\MM:=\{G: \Psi(e)\geq 0, \forall\ e\in E\}$$ the class of planar graphs with non-negative $\Psi$-curvature. By the definition, one easily verifies that $G\in \MM$ if and only if $G^*\in\MM.$ In fact, for any $e\in E$, there is a unique edge $e^*$ in $G^*$ such that $\Psi(e)=\Psi(e^*)$. Moreover, we observe that the $\Psi$-curvature is related to the corner curvature as follows. 
\begin{prop}\label{prop:relationCM} For any edge $e,$
$$\Psi(e)=\frac{1}{2}\left[C(x_1,\sigma_1)+C(x_1,\sigma_2)+C(x_2,\sigma_1)+C(x_2,\sigma_2)\right],$$ where $x_1,x_2$ are end-vertices of $e$ and $\sigma_1,\sigma_2$ are two faces whose closures contain $e.$
\end{prop} This yields that $\CC\subset \MM.$

\begin{figure}[htbp]
 \begin{center}
   \begin{tikzpicture}
    \node at (0,0){\includegraphics[width=0.6\linewidth]{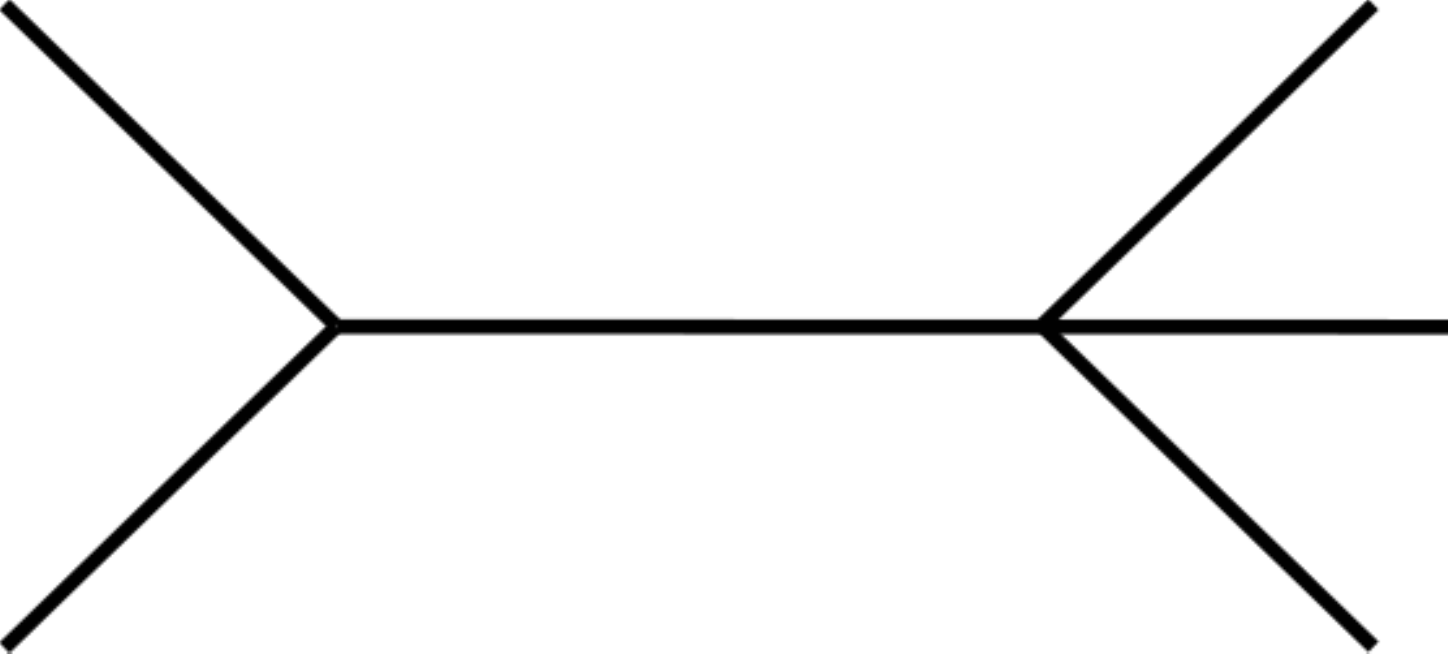}};
    \node at (0, -.3){\Large $e$};
    \node at (-2.6,   0){\Large $x_1$};
    \node at (1.5,   -.3){\Large $x_2$};
    \node at (0,  1.1){\Large $\sigma_1$};
    \node at (0,  -1.5){\Large $\sigma_2$};
   \end{tikzpicture}
  \caption{}\label{Fig:1}
 \end{center}
\end{figure}

To digest the definition of $\Psi$-curvature, we introduce the definition of medial graphs, introduced by Steiniz \cite{Steinitz22}. Let $G=(V,E,F)$ be a planar graph. The medial graph $\hat G=(\hat V, \hat E, \hat F)$ associated to $G$ is defined as follows. For each edge $e\in E$, we denote by $m(e)$ the midpoint of $e$ in $S(G)$. The set of vertices of $\hat G$ is given by $\hat V=\{m(e): e\in E\}$. Any two vertices $m(e_1)$ and $m(e_2)$ are adjacent, i.e. $\{m(e_1), m(e_2)\}\in \hat E$, if and only if $e_1$ and $e_2$ share a common vertex, and are incident to a common face. In this case, we draw a segment from $m(e_1)$ to $m(e_2)$ in the Euclidean polyhedral surface $S(G)$ representing the edge. The faces of $\hat G$ are induced by the embedding of the vertices and edges of $\hat G$ in $S(G)$, homeomorphic to $\mathbb S^2$ or $\mathbb R^2$. So that we have the following bijections
\[\hat V\approx E, \  \hat F\approx V\cup F.\]
Note that the medial graph $\hat{G}$ is always $4$-regular, i.e. $|m(e)|=4,$ for any $m(e)\in \hat{V}$. We will show that
if $G$ is a tessellation, then so is $\hat{G},$ see Proposition~\ref{medial-graph-tes}.

By the definitions, we observe that the $\Psi$-curvature of $G$ is given by the combinatorial curvature of the medial graph $\hat G$. So that many problems on the class $\MM$ can be reformulated to those on $4$-regular planar graphs in $\NNG$. The proof is evident, so that we omit it here.
\begin{prop} For any edge $e,$ $$\Psi(e)=\Phi(m(e)),$$ where $\Phi(m(e))$ is the combinatorial curvature at the vertex $m(e)$ in $\hat{G}.$
\end{prop}

Let $G$ be a planar graph. For $k\geq 3$, we denote by $V_k(G)$ the number of vertices of degree $k$ in $G$, and by $F_k(G)$ the number of faces of degree $k$ in $G.$ 
In this paper, we study the following problem.

\begin{problem}\label{num-face-vert-nonneg}
For any infinite $G\in \MM$, what are the upper bounds for $F_k(G)$ and $V_k(G)?$
\end{problem}
By the examples of regular planar tilings, it is possible that $V_k(G)=\infty$ or $F_k(G)=\infty$, for $k=3,4,6$. For a planar graph $G$, any face of the medial graph $\hat G$ corresponds to a vertex or a face in $G$ same degree. Hence we have
\begin{equation}\label{eq:key1}
V_k(G)+F_k(G)=F_k(\hat G)
\end{equation}
for any $k\geq 3$. By setting $\Gamma=\hat{G}$, Problem \ref{num-face-vert-nonneg} is reformulated to the estimate of $F_k(\Gamma)$ for a $4$-regular planar graph $\Gamma$ in $\NNG$. We first prove the following result, see Section~\ref{sec:3}.
\begin{prop}\label{max-face-degree} For a $4$-regular infinite planar graph $\Gamma\in \NNG,$ the degree of any face is at most $12.$
\end{prop} This yields that for any $G\in \MM,$ $F_k(G)=V_k(G)=0$ for any $k\geq 13.$

For any vertex $x$ with $|x|=N$ in a planar graph, the vertex pattern of $x$ is defined as  $$\pt(x):=(|\sigma_1|,|\sigma_2|,\cdots,|\sigma_{N}|),$$ where $\{\sigma_i\}_{i=1}^N$ are the faces incident to $x,$ ordered by $|\sigma_1|\leq|\sigma_2|\leq\cdots\leq|\sigma_{N}|.$
By \cite{MR2410938}, for any $\Gamma\in \NNG$ there are only finitely many vertices of positive curvature. So that except finite vertices, the combinatorial curvature is vanishing. If $\Gamma$ is $4$-regular, their vertex patterns with vanishing curvature are given by $(3,3,6,6)$, $(4,4,4,4)$, $(3,4,4,6)$ or $(3,3,4,12).$
This yields that for any $k\neq 3,4,6,12$, $F_k(\Gamma)$ is finite, which implies that $V_k(G)+F_k(G)$ is finite for any $G\in \MM.$ In fact, by some elementary arguments, we get the following estimates.
\begin{prop}\label{num-est-5-11}
 For a $4$-regular infinite planar graph $\Gamma\in \NNG,$
$$
F_5(\Gamma)\leq 21,\ F_7(\Gamma)\leq 15.$$
\end{prop}

For the critical case $k=12,$ it is very possible that there are infinitely many $12$-gons in a $4$-regular $\Gamma\in \NNG$. As a main result of the paper, we prove that the number of $12$-gons is finite, and in fact we give the tight upper bound estimate for the number of faces of degree at least $8$ as follows.
\begin{theorem}\label{thm:main2} For a $4$-regular infinite planar graph $\Gamma\in \NNG,$
$$\sum_{k=8}^{12}F_{k}(\Gamma)\leq 1.$$
\end{theorem}
\begin{remark}
Figure \ref{sharp-graph} indicates that the upper bound for the case $k=12$ is sharp. In fact, replacing $\sigma$ in Figure \ref{sharp-graph} by $k$-gon for $8\leq k\leq 11$, we obtain other examples for the sharpness when $8\leq k\leq 11$.
\end{remark}

We collect our estimates in Table~\ref{tabl2}.
\begin{table}
\refstepcounter{table}\label{tabl2}
\begin{tabular}{|lc|lr|}
\hline
$k$-gons && $F_k(\Gamma)$&\\
    \hline
$k=5$&&$\leq 21$&\\
$k=7$&&$\leq 15$&\\
$8\leq k\leq 12$& &$\leq 1$&\\
$k\geq 13$& &$=0$&\\
\hline
\multicolumn{4}{c}{}\\
\end{tabular}

\textbf{\tablename~\thetable.}
\end{table}

\begin{figure}[htbp]
 \begin{center}
   \begin{tikzpicture}
    \node at (0,0){\includegraphics[width=0.84\linewidth]{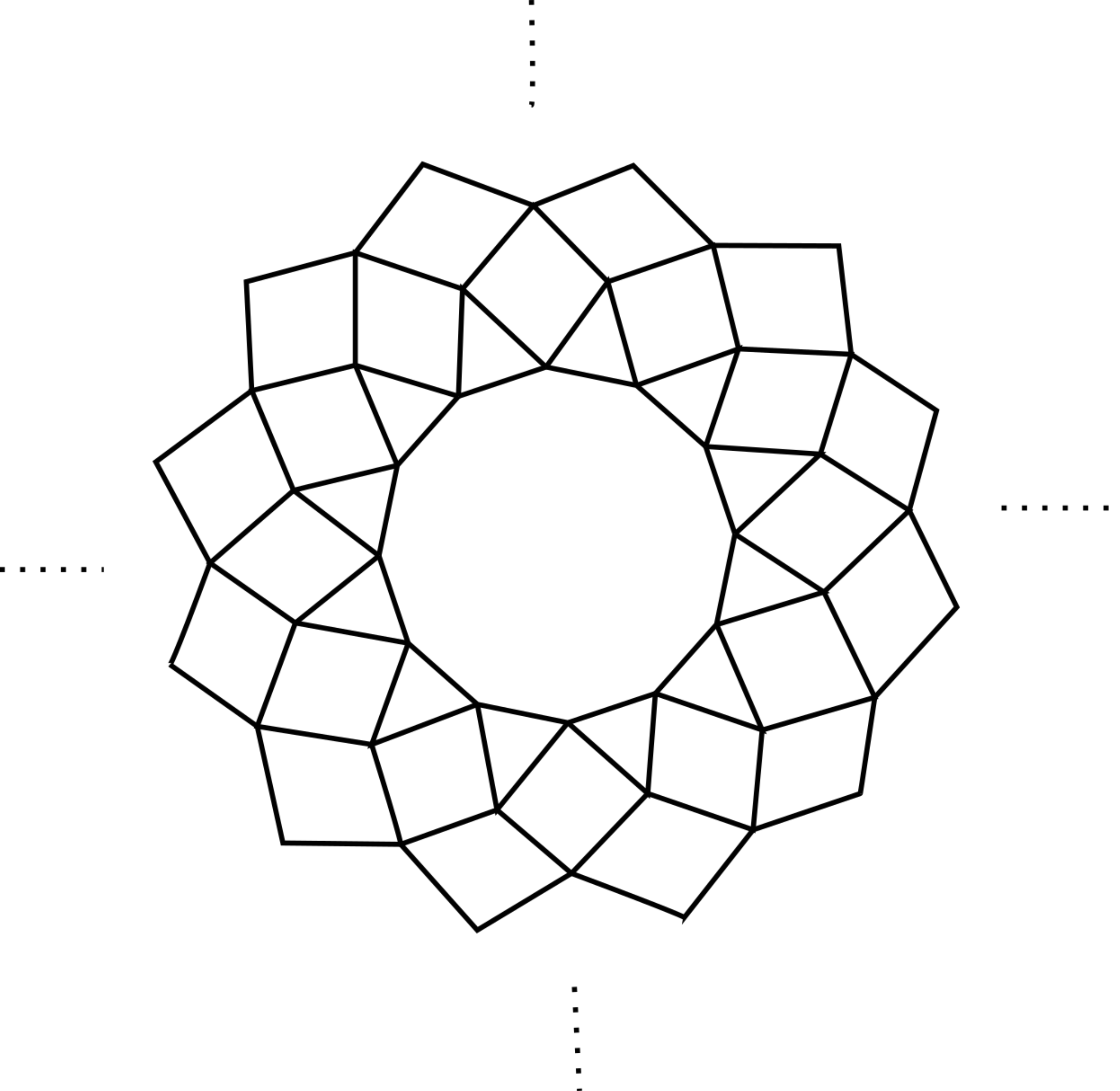}};
    \node at (0,0){\Large $\sigma$};
   \end{tikzpicture}
  \caption{The fact $\sigma$ is a $12$-gon. The vertex patterns of vertices incident to $\sigma$ are $(3,3,4,12)$, the vertex patterns of vertices which are adjacent to vertices in the boundary of $\sigma$ are $(3,4,4,4)$, and the other vertex patterns are $(4,4,4,4)$.}\label{sharp-graph}
 \end{center}
\end{figure}

The theorem is proved via the discharging method. Discharging methods were used by many authors in graph theory for different purposes, see e.g. \cite{MR0543792,MR1441258,MR1897390,MR2299456,MR3624614,Gh17, HS19}. The curvature at vertices of a planar graph can be regarded as the charge concentrated on vertices. The discharging method is to redistribute the charge on vertices via transferring the charge from ``goog" vertices to ``bad" vertices. For any face $\sigma,$ the $1$-neighborhood of $\sigma,$ denoted by $U_1(\sigma)$, is the union of vertices incident to $\sigma$ and the vertices adjacent to these vertices. To estimate the number of $k$-gons with $8\leq k \leq 12$, we first argue that the $1$-neighborhood of $k_1$-gon and the $1$-neighborhood of $k_2$-gon are disjoint for $8\leq k_1,k_2 \leq 12$, see Theorem~\ref{nexist-both-adj}.  For the discharging rule, for any $k$-gon $\sigma$ with $8\leq k\leq 12$ we redistribute the curvature of nearby vertices, within the $1$-neighborhood of $\sigma,$ to those vertices incident to $\sigma$, such that the total redistributed curvature of the vertices incident to $\sigma$ is uniformly bounded below. Then one can estimate the cardinality of $k$-gons with $8\leq k \leq 12$ by using the total curvature restriction, i.e. Theorem~\ref{thm:DeVosMohar}.

By the above theorem and \eqref{eq:key1}, we give the answer of Problem~\ref{num-face-vert-nonneg} for $8\leq k \leq 12.$
\begin{corollary}For any infinite $G\in \MM$, $\sum\limits_{k=8}^{12} \left(F_k(G)+V_k(G)\right)\leq 1.$
\end{corollary}
This result is sharp since one can construct an infinite graph whose medial graph is given by $\Gamma$ in Figure \ref{sharp-graph}.

The paper is organized as follows: In \S \ref{Preliminary}, we recall some basic facts about planar graphs and the curvature notions.  In \S \ref{sec:3}, we study basic properties of $4$-regular infinite planar graphs with non-negative combinatorial curvature, and prove Proposition~\ref{max-face-degree} and Propositon~\ref{num-est-5-11}. In \S \ref{sec:4}, we investigate the local structure of large faces in $4$-regular infinite planar graphs with non-negative combinatorial curvature, and prove that $1$-neighborhoods of large faces are disjoint in Theorem~\ref{nexist-both-adj}. In \S \ref{section 5}, we use the discharge method to prove the main result, Theorem \ref{thm:main2}.

\section{Preliminaries}\label{Preliminary}
Let $G=(V,E,F)$ be a planar graph. For any face $\sigma,$ we denote by $$\partial\sigma:=\{x\in V: x\prec \sigma\}$$ the set of boundary vertices of $\sigma.$
For any two elements in $V,E,F,$ we call that one is contained in the other if  the former is contained in the closure of the latter, denoted by $x\prec e, e\prec \sigma, x\prec \sigma$ etc., where $x\in V, e\in E, \sigma\in F.$ Two edges $e_1$ and $e_2$ (resp. two faces $\sigma_1$ and $\sigma_2$) are called lower-adjacent, denoted by $e_1\smile e_2$ (resp. $\sigma_1\smile \sigma_2$), if their closures contain a common vertex (resp. edge). We write $e=\sigma\cap\tau$ if $e\prec\sigma$ and $e\prec \tau$. For a face $\sigma$, we say that faces $\tau$ and $\omega$ are $\sigma$-neighbours if $\tau\smile\sigma$, $\omega\smile\sigma,$ and the intersection $\bar{\tau}$ and $\bar{\omega}$ is a vertex in $\sigma$. In this case, we also say that $\tau$ is $\sigma$-adjacent to $\omega.$ For convenience, sometimes we don't distinguish the closure of a face $\bar{\sigma}$ and a face $\sigma$ itself.

%
%
%
%


For a planar graph $G=(V,E,F),$ we introduce the definition of the medial graph $\hat{G}=(\hat{V}, \hat{E}, \hat{F})$ of $G.$ Let $\varphi$ be the embedding of $(V,E)$ in $S(G)$, where we equip it with the metric in the target for as in the introduction simplicity. For any edge $e\in E$, we denote by $m(e)$ the midpoint of $\varphi(e)$, which corresponds to a vertex in $\hat V$. For any $e_1, e_2\in E$, satisfying $e_1\smile e_2,$ $e_1\prec\sigma$ and $e_2\prec \sigma$ for some $\sigma\in F$, we draw a segment $l_{e_1e_2}$ from $m(e_1)$ to $m(e_2)$ in $S(G)$ presenting an edge $\{m(e_1), m(e_2)\}$ in $\hat E$. The face set $\hat F$ is induced by the above embedding $(\hat V, \hat E)$ in $S(G)$. We describe the face set $\hat F$ as follows. For any $x\in V$ with $|x|=N$, let $\{e_i\}_{i=1}^N$ be the set of edges incident to $x$, whose embedding images are ordered clockwise in $S(G)$. We denote by $\Sigma_x$ the face of $\hat F$ bounded by the edge set $\{l_{e_1e_2}, l_{e_2e_3}, \cdots, l_{e_{N-1}e_N}, l_{e_Ne_1}\}$. For any $\sigma\in F$ with $|\sigma|=Q$, let $\{s_i\}_{i=1}^Q$ be the set of edges incident to $\sigma$, whose embedding images are ordered clockwise in $S(G)$. We denote by $\Sigma_\sigma$ the face of $\hat F$ bounded by the edge set $\{l_{s_1s_2}, l_{s_2s_3}, \cdots, l_{s_{Q-1}s_Q}, l_{s_Qs_1}\}$. Then $\hat F=\{\Sigma_x\colon x\in V\} \cup \{\Sigma_\sigma \colon \sigma \in F\}$. One easily sees that for any $x\in V$, $\sigma\in F$, $|\Sigma_x|=|x|$, $|\Sigma_\sigma|=|\sigma|.$

\begin{prop}\label{medial-graph-tes}
 If $G$ is a planar tessellation, then so is $\hat{G}.$
\end{prop}
\begin{proof}
It suffices to check that $\hat G$ satisfies the planar tessellation as stated in \S \ref{sec:intro}. (i) This is obvious.

(ii) For any edge $\{m(e_1), m(e_2)\}\in \hat E$, there is a unique $x\in V$ (resp. $\sigma\in F$) such that $x$ (resp. $\sigma$) is incident to both $e_1$ and $e_2$. Then the embedding image $l_{e_1e_2}$ of $\{m(e_1),m(e_2)\}$ is contained in exactly $\bar{\Sigma}_x$ and $\bar{\Sigma}_\sigma$.

(iii) If $\bar{\Sigma}_x\cap\bar{\Sigma}_y\ne\emptyset$ for $x, y\in V$, then there is a unique edge $e\in E$ such that $x,y\prec e$. Then $\bar{\Sigma}_x\cap\bar{\Sigma}_y=\{m(e)\}$. If $\bar{\Sigma}_\sigma\cap\bar{\Sigma}_{\sigma'}\ne\emptyset$ for $\sigma,\sigma'\in F$, then there is a unique edge $e\in E$ such that $e\prec \sigma,\sigma'$. Then the intersection consists of a vertex $m(e)$. If $\bar{\Sigma}_x\cap\bar{\Sigma}_{\sigma}\ne\emptyset$ for $x\in V$ and $\sigma\in F$, then there are edges $e_1$ and $e_2$ such that $x$ (resp. $\sigma$) is incident to both $e_1$ and $e_2$. Hence, $\bar{\Sigma}_x\cap\bar{\Sigma}_\sigma=l_{e_1e_2}$ which represents the edge $\{m(e_1),m(e_2)\}$.

Hence the proof is completed.
\end{proof}

For a $4$-regular infinite planar graph with non-negative combinatorial curvature,
the list of vertex patterns of positive curvature is given in Table~\ref{tabl1}, and the vertex patterns with vanishing
curvature are given by
\begin{align}\begin{split}\label{vanishing-cur-pattern}
(3,3,6,6),  (4,4,4,4), (3,4,4,6), (3,3,4,12).
\end{split}\end{align}

\begin{table}
\refstepcounter{table}\label{tabl1}
\begin{tabular}{|lc|lr|}
\hline
Patterns of $x$ &&$\Phi(x)$&\\
    \hline
$(3,3,3,k)$&$3\leq k$&$1/k$&\\
$(3,3,4,k)$&$4\leq k\leq 11$&$1/k-1/12$&\\
$(3,3,5,k)$&$5\leq k\leq 7$&$1/k-2/15$&\\
$(3,4,4,k)$&$4\leq k\leq 5$&$1/k-1/6$&\\
\hline
\multicolumn{4}{c}{}\\
\end{tabular}

\textbf{\tablename~\thetable.} The patterns of a vertex with positive curvature.
\end{table}

Noting that $\CC\subset \MM,$ we give an example to show that the converse is not true.
\begin{example}\label{exam:p1} Let $G$ be a periodic square tiling of $\R^2$ as in Figure~\ref{perior-tiling}. Then the medial graph $\hat G$ is the uniform (Archimedean) tiling of type $(3,4,4,6).$ Hence $G\in\MM.$ One easily sees that $G\not\in \NNG,$ since it has a vertex of pattern $(4,4,4,4,4,4)$ whose combinatorial curvature is negative. Note that $\CC\subset \NNG,$
 so that $G\not\in \CC.$ \begin{figure}[htbp]

 \begin{center}
   \begin{tikzpicture}
    \node at (0,0){\includegraphics[width=0.6\linewidth]{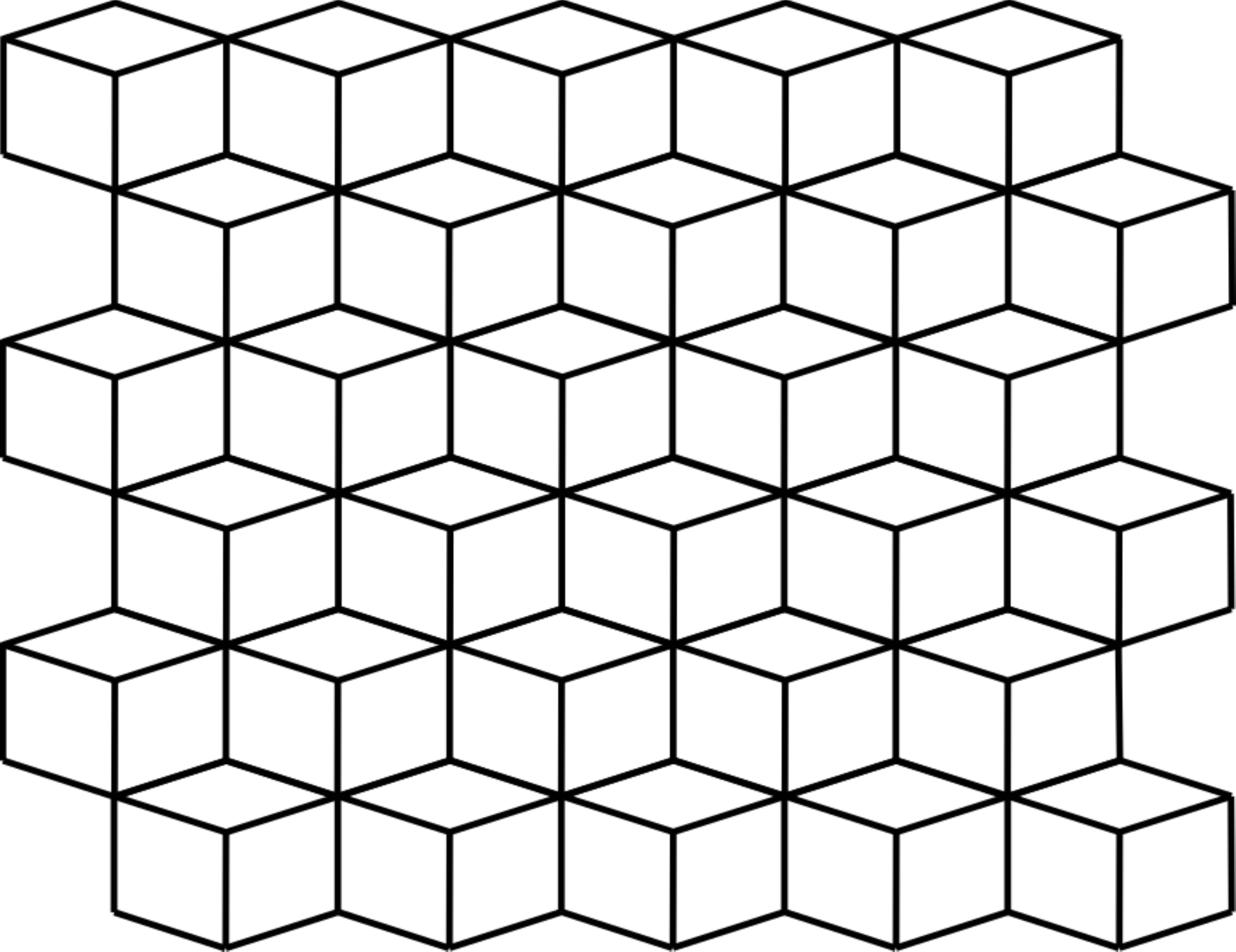}};
   \end{tikzpicture}
  \caption{}\label{perior-tiling}
 \end{center}
\end{figure}
\end{example}

Next we show that $\NNG\not\subset \MM$ and $\MM \not\subset \NNG.$ The second assertion follows from Example~\ref{exam:p1}.
For the first assertion, consider the uniform (Archimedean) tiling $G$ of type $(3, 12, 12).$ Then $G\in \NNG.$ Note that in the medial graph $\hat{G}$ there is a vertex of pattern $(3,3,12,12)$ whose combinatorial curvature is negative, so that $G\notin \MM$.

\section{Some properties of $4$-regular infinite planar graphs with non-negative combinatorial curvature}\label{sec:3}
In this section, we study some basic properties of infinite $4$-regular planar graphs with non-negative combinatorial curvature.

First, we derive an upper bound for the degree of faces in an infinite $4$-regular planar graph with non-negative combinatorial curvature.
\begin{theorem}\label{thm:13} Let $\Gamma$ be a $4$-regular planar graph with non-negative combinatorial curvature. Suppose that $\sup_{\sigma\in F}|\sigma|\geq 13,$ then $\Gamma$ is a finite graph, i.e. embedded into $\mathbb S^2,$ and $\Gamma$ is an antiprism.
\end{theorem}
\begin{proof} Let $\sigma$ be a face of degree at least $13.$ For any vertex $x\in \partial \sigma,$ $\Phi(x)\geq 0$, then by Table \ref{tabl1} and (\ref{vanishing-cur-pattern}), $\pt(x)=(3,3,3,|\sigma|)$ and  \begin{equation}\label{eq2:eq1}\sum_{x\in \partial\sigma} \Phi(x)=1.\end{equation}
Let $\{\tau_i\}_{i=1}^{|\sigma|}$ be the triangles lower-adjacent to $\sigma,$ ordered clockwise along $\sigma.$ For any $1\leq i\leq |\sigma|,$ let $x_i\in \partial\tau_i\setminus\partial \sigma.$ We claim that  $x_i\neq x_j$ for $i\neq j$. Note that all vertices in $\partial \sigma$ are of pattern $(3,3,3,|\sigma|),$ and there are triangles $\{\omega_i\}_{i=1}^{|\sigma|}$ with $\omega_i$ is lower-adjacent to $\tau_i$ and $\tau_{i+1}$ and incident to $\sigma,$ where the index is understood in the sense of modulo $|\sigma|.$ Hence for any $i,$ $x_i$ is incident to three triangles $\omega_{i-1},\tau_i,\omega_{i}.$ Hence $x_i$ is not identified with $x_j$ for $i\neq j,$ which follows the tessellation properties $(i)-(iii)$ and $|x_i|=4.$ This proves the claim.

We divide it into cases.

\emph{Case~1.} $\Gamma$ is an infinite graph, i.e. embedded into $\R^2.$ By Theorem \ref{thm:DeVosMohar},
$$\sum_{x\in V}\Phi(x)\leq 1.$$ By \eqref{eq2:eq1}, this yields that for any $x\not\in \partial\sigma,$ $\Phi(x)=0.$
 Note that for each $x_i$ defined above, it is incident to three triangles. So that $\Phi(x_i)>0$ which yields a contradiction. 

\emph{Case~2.} $\Gamma$ is a finite graph, i.e. embedded into $\mathbb S^2$. By the Gauss-Bonnet theorem, $$\sum_{x\in V}\Phi(x)=2.$$ Since $\{x_i\}$ satisfies that $x_i\neq x_j$ for $i\neq j$ and each $x_i$ incident to three triangles. Let $\sigma_1$ be a face incident to $x$, which different from the three triangles obtained before. One can show that all $x_i$ are the boundary of same face $\sigma_1$ and $|\sigma_1|=|\sigma|.$ This proves that $\Gamma$ is an antiprism.
\end{proof}

This yields that the maximum facial degree of an infinite $4$-regular planar graph with non-negative combinatorial curvature is at most $12.$
\begin{proof}[Proof of Proposition~\ref{max-face-degree}]
Since $\Gamma$ is an infinite $4$-regular planar graph with non-negative combinatorial curvature,
by Theorem~\ref{thm:13}, $|\sigma|\leq 12,$ for any face $\sigma.$ This proves the proposition.
\end{proof}

Next, we give upper bounds for the number of faces of degree $k,$ for $k=5,7.$
\begin{proof}[Proof of Proposition~\ref{num-est-5-11}]
Let $\sigma$ be a $k$-gon in $\Gamma$. Let $\sigma_1$ and $\sigma_2$ be two $k$-gons in $\Gamma$ with facial degree $k$.
By Table \ref{tabl1} and (\ref{vanishing-cur-pattern}), if $k\geq 7$, then
\begin{align}\begin{split}\label{two-k-gon-nonadj}
\d\sigma_1\cap\d\sigma_2=\varnothing.
\end{split}\end{align}
We divide it into cases.


\textbf{Case 1.} $k=7$. For any $z\in \d\sigma$, $\Pttn(z)=(3,3,3,7)$, $(3,3,4,7)$ or $(3,3,5,7)$, and $\Phi(z)\geq \frac{1}{105}$. This implies that $\sum_{z\in\partial \sigma}\Phi(z)\geq \frac{1}{15}.$
By (\ref{two-k-gon-nonadj}) and Theorem \ref{thm:DeVosMohar}, we have
\[
\frac{1}{15}F_7(\Gamma)\leq \sum_{\sigma\in F: |\sigma|=7}\sum_{z\in\partial \sigma}\Phi(z)
\leq  1,
\]
which gives
\[
F_7(\Gamma)\leq 15.
\]

\textbf{Case 2.} $k=5$. For any $z\in V$, by Table \ref{tabl1} and (\ref{vanishing-cur-pattern}), we have $$\#\{\tau: z\prec\tau, |\tau|=5\}\leq 2.$$ For any $z\in\partial\sigma$, if $\#\{\tau: z\prec\tau, |\tau|=5\}=1$, then $\Phi(z)\geq \frac{1}{105}$. If $\#\{\sigma: z\prec\sigma, |\sigma|=5\}=2$, then $\Phi(z)=\frac{1}{15}$ and $\Phi(z)$ contributes the curvature $\frac{1}{30}$ to each pentagon.
Hence, by Theorem \ref{thm:DeVosMohar}, we have
\[
\frac{1}{21}F_5(\Gamma)\leq \sum_{\sigma\in F: |\sigma|=5}\sum_{z\in\partial \sigma}\Phi(z)
\leq  2\sum_{z\in V}\Phi(z)\leq 1,
\]
which yields
\[
F_5(\Gamma)\leq 21.
\]

This proves the proposition.
\end{proof}

\begin{figure}[htbp]
 \begin{center}
 \includegraphics[width=0.4\linewidth]{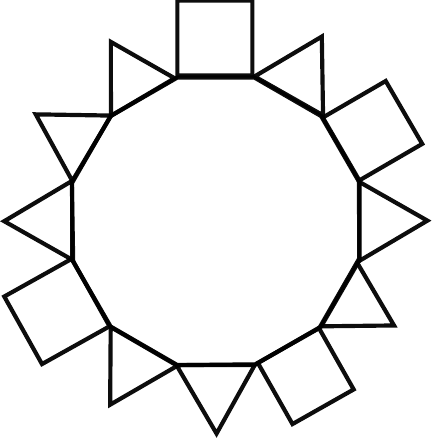}
 \caption{}\label{Fig:2}
 \end{center}
\end{figure}

Let $\Gamma$ be a $4$-regular infinite planar graph satisfying $\Gamma\in \NNG$. For a large face $\sigma$ in $\Gamma$ with facial degree no less than $8$, the non-negative combinatorial curvature implies that any vertex in $\partial \sigma$ is of pattern $(3,3,3,k)$ or $(3,3,4,k)$, see Table \ref{tabl1} and (\ref{vanishing-cur-pattern}). Hence, we obtain that the faces surrounding $\sigma$ are triangles or squares in the following property. Since the proof is obvious, we omit it here.

\begin{prop}\label{prop:nonadj}Let $\Gamma$ be a $4$-regular infinite planar graph satisfying $\Gamma\in \NNG$ and $\sigma$ be a face in $\Gamma$ with facial degree $k$ and $k\geq 8$. Then any face lower-adjacent to $\sigma$ is either a triangle or a square. Moreover, for any two squares lower-adjacent to $\sigma,$ they are not $\sigma$-adjacent.
\end{prop}

For the faces lower-adjacent to a large face, we obtain that they have no intersection except vertices in the large face.
\begin{prop}\label{nonadj1}
Let $\Gamma$ be a $4$-regular infinite planar graph satisfying $\Gamma\in \NNG$ and $\sigma$ be a face in $\Gamma$ with facial degree $8\leq k\leq 12$.   Then for any $\tau,\omega$ lower-adjacent to $\sigma,$ \begin{equation}\label{eq4:eq2}(\partial \tau\setminus \partial \sigma)\cap (\partial \omega\setminus \partial \sigma)=\emptyset.\end{equation}
\end{prop}
\begin{remark} Let $\{\tau_i\}_{i=1}^{k}$ be faces lower-adjacent to a face $\sigma$ with facial degree $k$ and $8\leq k \leq 12$. This proposition implies that the vertices in $\{\partial \tau_i\setminus \partial \sigma\}_{i=1}^{k}$ has no identification. This means that in general we have the picture as in Figure~\ref{Fig:2}.
\end{remark}

\begin{proof} By Proposition~\ref{prop:nonadj}, $\tau$ (resp. $\omega$) is either a triangle or a square. By the tessellation properties, one easily shows that $\tau$ and $\omega$ satisfy \eqref{eq4:eq2} if they are $\sigma$-neighbours. We may assume that $\tau$ and $\omega$ are not $\sigma$-neighbours. Suppose that the statement \eqref{eq4:eq2} is not true.

We divide it into cases.

\emph{Case~1.} $|\tau|=|\omega|=3.$ Let $x=(\partial \tau\setminus \partial \sigma)\cap (\partial \omega\setminus \partial \sigma).$ Let $e_1,e_2$ (resp. $e_3,e_4$) be two edges contained in $\tau$ (resp. in $\omega$), but not in $\sigma.$ For $1\leq i\leq 4,$ let $y_i$ be a vertex satisfying $y_i\prec e_i$ and $y_i\prec \sigma.$ Since $\tau$ is not $\sigma$-adjacent to $\omega$, either $y_1$ is not adjacent to $y_3$ or $y_2$ is not adjacent to $y_4.$ Without loss of generality, we assume that $y_1$ is not adjacent to $y_3,$ see Figure~\ref{Fig:6-4-1}.
Since $|x|=4,$ there are two other faces, $\tau'$ and $\omega'$, containing $x,$ so that each of them contains two edges from $\{e_i\}_{i=1}^4.$ Hence by the tessellation properties there is one of $\tau'$ and $\omega',$ say $\tau'$, containing two edges $e_1$ and $e_3$ or two edges $e_1$ and $e_4.$ For either case, it contradicts the tessellation properties, since $\tau'$ and $\sigma$ intersect at two non-adjacent vertices $\{y_1,y_3\}$ or $\{y_1,y_4\}$.

\begin{figure}[htbp]
 \begin{center}
   \begin{tikzpicture}
    \node at (0,0){\includegraphics[width=0.4\linewidth]{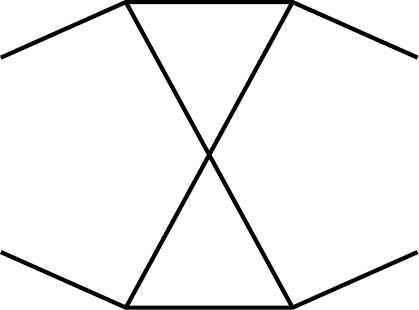}};
    \node at (-0.4, 0){\Large $x$};
    \node at (-.7,   .7){\Large $e_1$};
    \node at (.7,   .7){\Large $e_2$};
    \node at (0,  1){\Large $\tau$};
    \node at (0,  -1){\Large $\omega$};
     \node at (-.7,   -.7){\Large $e_3$};
    \node at (.7,   -.7){\Large $e_4$};
    \node at (-1,   2.1){\Large $y_1$};
    \node at (1,   2.1){\Large $y_2$};
    \node at (-1,   -2.1){\Large $y_3$};
    \node at (1,   -2.1){\Large $y_4$};
    \node at (-2.2,  -1){\Large $\sigma$};
   \end{tikzpicture}
  \caption{}\label{Fig:6-4-1}
 \end{center}
\end{figure}

\emph{Case~2.} One of $\tau$ and $\omega$ is a triangle and the other is not, say $|\tau|=3$ and $|\omega|=4.$
Let $e_1,e_2$ (resp. $e_3,e_4,e_5$) be the edges contained in $\tau$ (resp. in $\omega$), but not in $\sigma.$ For $1\leq i\leq 3,$ let $y_i$ be a vertex satisfying $y_i\prec e_i$ and $y_i\prec \sigma.$ We divide it into two subcases.

\emph{Case~2.1.} $y_1\not \sim y_3,$ see Figure~\ref{Fig:6-4-2}. \begin{figure}[htbp]
 \begin{center}
   \begin{tikzpicture}
    \node at (0,0){\includegraphics[width=0.4\linewidth]{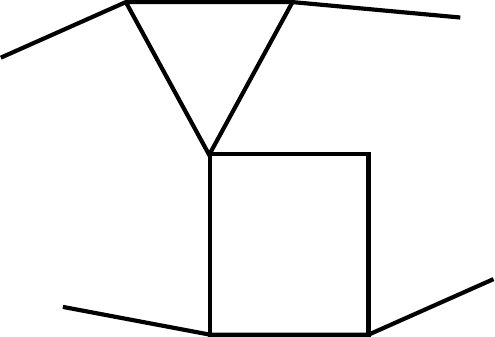}};
    \node at (-0.6, 0){\Large $x$};
    \node at (-1.1,   .9){\Large $e_1$};
    \node at (.4,   .9){\Large $e_2$};
    \node at (-.4,  1){\Large $\tau$};
    \node at (.5,  -.8){\Large $\omega$};
     \node at (-.7,   -.7){\Large $e_3$};
      \node at (1.6,   -.7){\Large $e_5$};
    \node at (.5,   0.4){\Large $e_4$};
    \node at (-1.2,   2.1){\Large $y_1$};
    \node at (.7,   2.1){\Large $y_2$};
    \node at (-0.5,   -2.1){\Large $y_3$};
        \node at (-1.6,  -1.2){\Large $\sigma$};
   \end{tikzpicture}
  \caption{}\label{Fig:6-4-2}
 \end{center}
\end{figure} By $|x|=4,$ there is one face $\tau'$ containing $x$ which contains the edges $e_1$ and $e_3$ or $e_2$ and $e_3.$ In both cases, $\tau'$ and $\sigma$ intersect at two non-adjacent vertices $\{y_1,y_3\}$ or $\{y_2,y_3\},$ which contradicts the tessellation properties.

\emph{Case~2.2.} $y_1\sim y_3,$ see Figure~\ref{Fig:4-6-3}.\begin{figure}[htbp]
 \begin{center}
   \begin{tikzpicture}
    \node at (0,0){\includegraphics[width=0.4\linewidth]{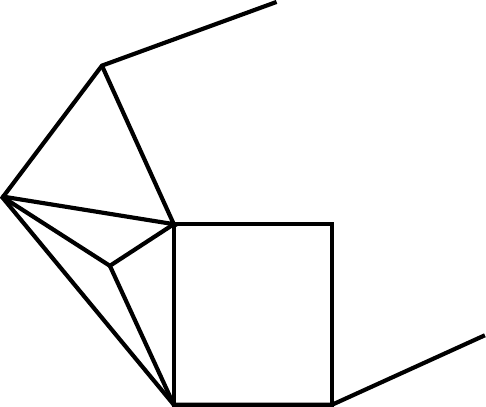}};
    \node at (-0.5, 0){\Large $x$};
    \node at (-1.5,   .1){\Large $e_1$};
    \node at (-.8,   .7){\Large $e_2$};
    \node at (-1.6,  .5){\Large $\tau$};
    \node at (0.2,  -1.2){\Large $\omega$};
     \node at (-.4,   -1.2){\Large $e_3$};
    \node at (.1,  0 ){\Large $e_4$};
      \node at (1.2,  -1.2 ){\Large $e_5$};
    \node at (-2.7,   .1){\Large $y_1$};
    \node at (-1.6,   1.8){\Large $y_2$};
    \node at (-1,   -2.2){\Large $y_3$};
        \node at (-1.5,   -1){\Large $\eta$};
        \node at (-1.4,   -.4){\Large $z$};
            \node at (2.2,  -1.3){\Large $\sigma$};

   \end{tikzpicture}
  \caption{}\label{Fig:4-6-3}
 \end{center}
\end{figure}  Let $\eta$ be a face lower-adjacent to $\sigma$ which contains the edge $\{y_1y_3\}.$ Note that $\eta$ and $\omega$ are $\sigma$-adjacent. By Proposition~\ref{prop:nonadj} and $|\omega|=4,$ $|\eta|=3.$
Set $z=\partial \eta\setminus \partial \sigma.$ It is easy to see that $z\neq x.$ Considering the vertex pattern of $y_3,$ we get $z\sim x.$ This contradicts $|x|=4.$

\emph{Case~3.}  $|\tau|=|\omega|=4.$ We divide it into subcases.

\emph{Case~3.1.} $\sharp (\partial \tau \cap \partial \omega)=2.$ Let $\{x,y\}=\partial \tau \cap \partial \omega.$ By the orientability of $\R^2,$ we have the case as in Figure~\ref{Fig:4-6-4}. \begin{figure}[htbp]
 \begin{center}
   \begin{tikzpicture}
    \node at (0,0){\includegraphics[width=0.4\linewidth]{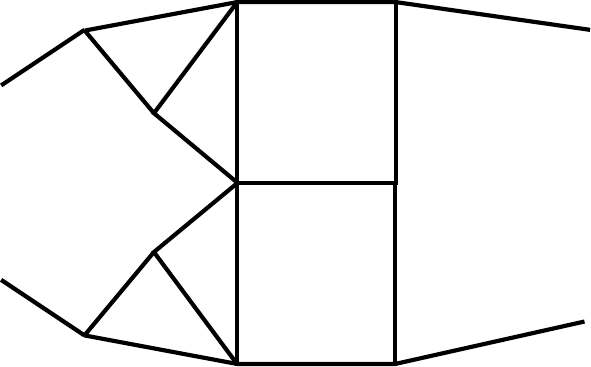}};
    \node at (-.8, 0){\Large $x$};
        \node at (1.2, 0){\Large $y$};
    \node at (.2,  .7){\Large $\tau$};
        \node at (-1.5,  .5){\Large $z_1$};
         \node at (-1.5,  -.5){\Large $z_2$};
    \node at (0.2,  -.8){\Large $\omega$};
        \node at (-1.2,  1.1){\Large $\tau_1$};
        \node at (-1.2,  -1.1){\Large $\tau_2$};
            \node at (-2.3,  -.7){\Large $\sigma$};
   \end{tikzpicture}
  \caption{}\label{Fig:4-6-4}
 \end{center}
\end{figure} By a similar argument as in the proof of Case 2.3, one can show that there are more than two edges on the boundary of $\sigma$ which lie on the left hand side of $\tau$ and $\omega.$  Let $\tau_1\smile \sigma$ (resp. $\tau_2\smile \sigma$) be the face $\sigma$-adjacent to $\tau$ (resp. $\omega$) as in Figure~\ref{Fig:4-6-4}. By Proposition~\ref{prop:nonadj}, $|\tau_i|=3,$ $i=1,2.$ Set $z_i=\partial \tau_i\setminus \partial \sigma$ for $i=1,2.$ By Case~1, $z_1,$ $z_2,$ $x$ are different. By considering the vertex patterns of $\partial \tau_1\cap \partial \tau$ (resp. $\partial \tau_2\cap \partial \omega$), we get that $z_1\sim x$ ($z_2\sim x$). This contradicts $|x|=4.$

\emph{Case~3.2.} $\sharp (\partial \tau \cap \partial \omega)=1.$ Set $x=\partial \tau \cap \partial \omega.$
Let $\tau_1\smile \sigma$ (resp. $\tau_2\smile \sigma$) be the face $\sigma$-adjacent to $\tau$ (resp. $\omega$) as in Figure~\ref{Fig:4-6-5}. \begin{figure}[htbp]
 \begin{center}
   \begin{tikzpicture}
    \node at (0,0){\includegraphics[width=0.4\linewidth]{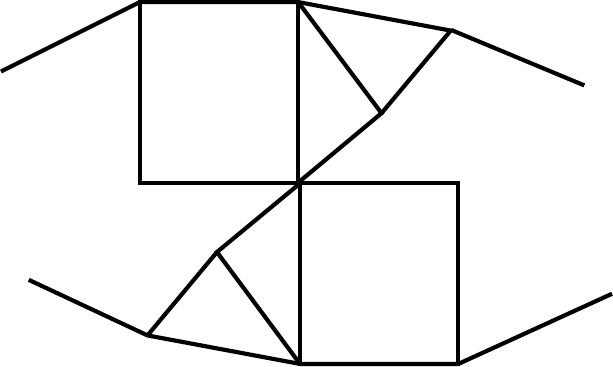}};
    \node at (.2, -.3){\Large $x$};
     \node at (-.8,  .7){\Large $\tau$};
        \node at (1,  .5){\Large $z_1$};
         \node at (-1,  -.5){\Large $z_2$};
    \node at (0.5,  -.8){\Large $\omega$};
        \node at (.6,  1.1){\Large $\tau_1$};
        \node at (-.7,  -1.1){\Large $\tau_2$};
            \node at (-2.1,  -.6){\Large $\sigma$};
   \end{tikzpicture}
  \caption{}\label{Fig:4-6-5}
 \end{center}
\end{figure}  The same argument as in Case~3.1 yields that $\tau_1$ and $\tau_2$ are triangles. Set $z_i=\partial \tau_i\setminus \partial \sigma$ for $i=1,2.$ We have $z_1\sim x$ ($z_2\sim x$) which contradicts $|x|=4.$

By combining all cases above, we prove the theorem.
\end{proof}

\section{The local structure of large faces in $4$-regular infinite planar graphs with non-negative combinatorial curvature}\label{sec:4}
In this section, we study the local structure of large faces in infinite $4$-regular planar graphs with non-negative combinatorial curvature. The main result of this section is the following theorem, which states that the $1$-neighborhoods of two large faces are disjoint. This will be crucial for the discharging methods in Section~\ref{section 5}.

We denote by
\[
U_1(\sigma)=\d\sigma\cup \{x\in V: x\sim z, z\in\d\sigma\},
\]
the $1$-neighborhood of a face $\sigma$ in $\Gamma$,  and by
\[
\d U_1(\sigma)=U_1(\sigma)\setminus \d\sigma,
\]
the boundary vertices of $U_1(\sigma)$.
\begin{theorem}\label{nexist-both-adj}
Let $\Gamma$ be a $4$-regular infinite planar graph satisfying $\Gamma\in \NNG$. For $i=1,2$, let $\sigma_i$ be a face with facial degree $k_i\geq 8$. Then
\[
U_1(\sigma_1)\cap U_1(\sigma_2)=\varnothing.
\]
\end{theorem}
\begin{remark}\label{three-cases}
Clearly,  $U_1(\sigma_i)=\d U_1(\sigma_i)\cup\d\sigma_i$ for $i=1,2$. So we divide the proof into three cases:
\begin{enumerate}[(i)]
 \item $\partial\sigma_1\cap \partial\sigma_2=\varnothing$;
 \item $\d U_1(\sigma_1)\cap \d U_1(\sigma_2)=\varnothing$;
 \item $\d U_1(\sigma_i)\cap \d\sigma_j=\varnothing$ for $1\leq i\neq j\leq 2$,
\end{enumerate}
see Proposition \ref{prop:disj1}, Theorem \ref{prop:disj2} and Proposition \ref{vert-in-sigma-conn-anosigma} respectively in the following.
\end{remark}

\begin{prop}\label{prop:disj1} Let $\Gamma$ be a $4$-regular infinite planar graph satisfying $\Gamma\in \NNG$. For $i=1,2$, let $\sigma_i$ be a face with facial degree $k_i\geq 8$. Then $\partial \sigma_1\cap \partial \sigma_2=\emptyset.$
\end{prop}
\begin{proof} Suppose that it is not true. For any vertex $x\in \partial \sigma_1\cap \partial \sigma_2,$ we have $\Phi(x)<0.$ This is a contradiction.
\end{proof}

Before proving Theorem \ref{prop:disj2}, we need the following two lemmas. Meantime, Proposition \ref{non-square-lower-triangle} is essential and we put it after Theorem \ref{prop:disj2}  since its proof is independent and tedious.
\begin{lemma}\label{square-to-trianglespair}
Let $\Gamma$ be a $4$-regular infinite planar graph satisfying $\Gamma\in \NNG$ and $\sigma_i$ be faces of $\Gamma$ with facial degree $k_i\geq 8$ for $i=1,2$. If there is a square lower-adjacent to $\sigma_1$ and $\sigma_2$, then there exist two triangles $\tau$ and $\omega$ such that $\tau$(resp. $\omega$) is $\sigma_1$-adjacent(resp. $\sigma_2$-adjacent) to the square, and $\partial\tau\setminus \partial\sigma_1=\partial\omega\setminus \partial\sigma_2$.
\end{lemma}
\begin{figure}[htbp]
\begin{center}
   \begin{tikzpicture}
    \node at (0,0){\includegraphics[width=0.3\linewidth]{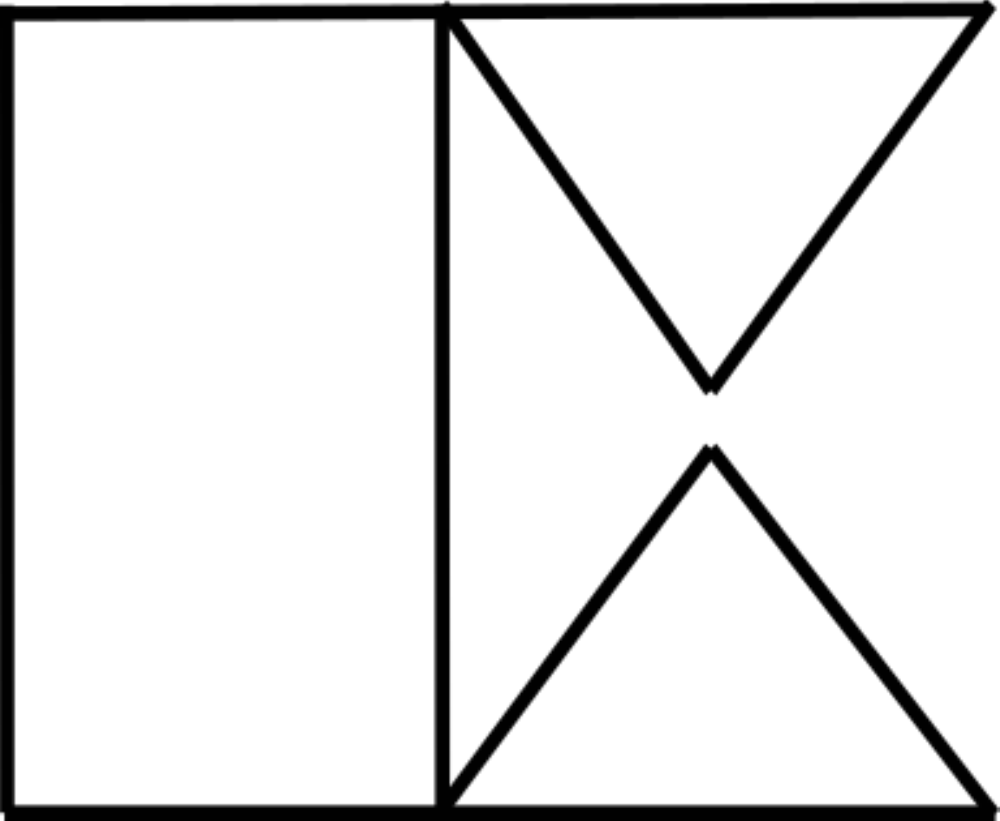}};
    \node at (1.1, .1){\small $p$};
    \node at ( .8, -.5){\small $q$};
    \node at ( -.2, 1.8){\small $p_1$};
    \node at ( -.2, -1.8){\small $q_1$};
    \node at (.8,  -1){\small $\omega$};
    \node at (.8,  1.0){\small $\tau$};
    \node at (-1,   .1){\small $\omega_1$};
    \node at (-1.4,  -2.1){\small $\sigma_2$};
    \node at (-1.4,   2.1){\small $\sigma_1$};
    \end{tikzpicture}
  \caption{}\label{square-up-down-triangles}
 \end{center}
\end{figure}

\begin{proof} Let $\omega_1$ be a square lower-adjacent to $\sigma_1$ and $\sigma_2$, as in Figure \ref{square-up-down-triangles}.
Since $\Pttn(x)=(3,3,3,k_i)$ or $(3,3,4,k_i)$ for $x\prec \sigma_i$, $i=1,2$,
then $|\omega_1|=4$ implies that there are two triangles, $\tau$ and $\omega$, $\sigma_1$-adjacent to $\omega_1$ and $\sigma_2$-adjacent to $\omega_1$, respectively. Let $p=\partial\tau\setminus \partial\sigma_1$, $q=\partial\omega\setminus\partial\sigma_2$, $p_1=\partial\tau\cap\partial\omega_1$ and $q_1=\partial\omega\cap\partial\omega_1$.
Since $\Pttn(p_1)=(3,3,4,k_1)$ and $|p_1|=4$, $\{p_1,q_1\}$ and $\{p_1,p\}$ are contained in a triangle. By $\Pttn(q_1)=(3,3,4,k_2)$ and $|q_1|=4$,
$\{p_1,q_1\}$ and $\{q_1,q\}$ are contained in a triangle. Then $p=q$.
This proves the lemma.
\end{proof}

\begin{figure}[htbp]
 \begin{center}
   \begin{tikzpicture}
    \node at (0,0){\includegraphics[width=0.38\linewidth]{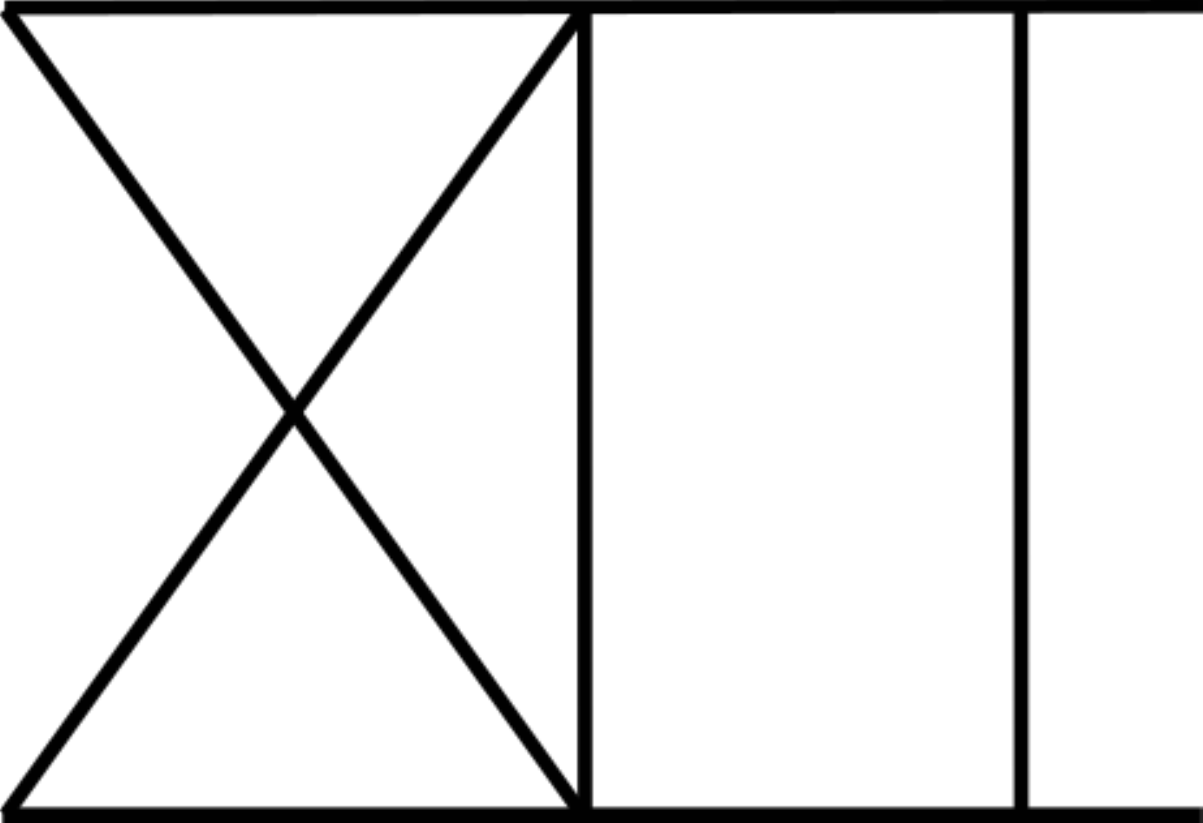}};
    \node at (.9, -.1){\small $\omega_1$};
    \node at (-.4, -.1){\small $\omega'$};
     \node at (-.45, .6){\small $e_1$};
   \node at (- .0, 1.8){\small $y_1$};
   \node at (1.6, 1.8){\small $p_1$};
     \node at (-2, .7){\small $e_2$};
    \node at (-2.2, 1.8){\small $y_2$};
     \node at (-.45, -.8){\small $e_4$};
    \node at (0, -1.8){\small $y_4$};
    \node at (1.6, -1.8){\small $q_1$};
     \node at (-2.1, -.8){\small $e_3$};
    \node at ( -2.2, -1.8){\small $y_3$};

    \node at (-1,  0){\small $x$};
    \node at (-1,  -1){\small $\omega$};
    \node at (-1,  1.0){\small $\tau$};

    \node at (-1.8,  - .1){\small $\tau'$};

    \node at (-1.3,  -2.1){\small $\sigma_2$};
    \node at (-1.2,   2.1){\small $\sigma_1$};

   \end{tikzpicture}
  \caption{}\label{up-down-triangles-3333}
 \end{center}
\end{figure}
\begin{lemma}\label{all-faces-type}
Assume that there are two triangles $\tau$ and $\omega$ such that $\tau\smile\sigma_1$  and $\omega\smile \sigma_2$.
If $\partial\tau\setminus\partial\sigma_1=\partial\omega\setminus\partial\sigma_2$,
then one of the following holds:
\begin{enumerate}
\item There exists a square $\omega_1$ such that $\omega_1$ is $\sigma_1$-adjacent to $\tau$ and $\omega_1$ is $\sigma_2$-adjacent to $\omega$.

\item There exist two triangles, $\tau_1$ and $\omega_1$, satisfying that $\tau_1$ is $\sigma_1$-adjacent to $\tau$ and $\omega_1$ is $\sigma_2$-adjacent to $\omega$ for $i=1,2$, such that $(\partial\tau_1\setminus\partial\sigma_1)\cap(\partial\omega_1\setminus\partial\sigma_2)\neq\varnothing$.
\end{enumerate}
\end{lemma}
\begin{proof}
Let $x=\partial\tau\setminus\partial\sigma_1=\partial\omega\setminus\partial\sigma_2$.  Let $e_1,e_2$ (resp. $e_3,e_4$) be two edges contained in $\tau$ (resp. in $\omega$) ordered clockwise along $\tau$ (resp. $\omega$), but not in $\sigma.$ Since $|x|=4,$ there exists a face $\tau'$(resp. $\omega'$) containing $e_2,e_3$ (resp. $e_1,e_4$). For $1\leq i\leq 4,$ let $y_i$ be a vertex satisfying $y_i\prec e_i$, $y_1,y_2\prec \sigma_1$ and $y_3, y_4\prec \sigma_2$.

 Suppose that $|\omega'|=3$, as shown in Figure \ref{up-down-triangles-3333}.  Then $y_1\sim y_4$ and the edge $\{y_1,y_4\}$ is contained in $\omega'$. By $|y_1|=4$, there exists an edge $\{y_1,p_1\}\prec\sigma_1$ such that $\{y_1,p_1\}\nprec\tau$, and both $\{y_1,p_1\}$ and $\{y_1,y_4\}$ are contained in a face. Similarly, by $|y_4|=4$, there  exists an edge $\{y_4, q_1\}\prec\sigma_2$ such that $\{y_4,q_1\}\nprec\omega$, and both $\{y_4,q_1\}$ and $\{y_1,y_4\}$ are contained in a face. Hence, $\{y_1,p_1\}$, $\{y_1,y_4\}$ and $\{y_4,q_1\}$ are contained in a face, denoted by $\omega_1$. Since $\Pttn(y_1)=(3,3,3,k_1)$ or $(3,3,4,k_1)$,  $|\omega_1|=3,4$. We claim that $|\omega_1|=4$. Otherwise, by $p_1=q_1\prec \sigma_2$, together with $|\sigma_1|,|\sigma_2|\geq 8$, we obtain $\Phi(p_1)<0$. This contradicts  $\Phi(p_1)\geq 0$ and proves the claim. Hence, $|\omega_1|=4$ and $\omega_1$ is $\sigma_1$-adjacent to $\omega$ and $\sigma_2$-adjacent to $\omega$.

Suppose that $|\omega'|=4$, as shown in Figure \ref{up-down-triangles-33-4-33}. $|x|=4$ yields that both $e_1$ and $e_4$ are contained in $\omega'$. Then $\Pttn(y_1)=(3,3,4,k_1)$ and $\Pttn(y_4)=(3,3,4,k_2)$. Hence, there are two triangles, $\tau_1$ and $\omega_1$, $\sigma_1$-adjacent to $\tau$ and $\sigma_2$-adjacent to $\omega$ respectively, as shown in Figure \ref{up-down-triangles-33-4-33}.
Let $p=\partial\tau_1\setminus \partial\sigma_1$, $q=\partial\omega_1\setminus\partial\sigma_2$. Then for $|y_1|=4$ we obtain that $\{y_1, p\}$ and $e_1$ are contained in a face, $\{y_4, q\}$ and $e_4$ are contained in a face for $|y_4|=4$. Hence, $\{y_1, p\},e_1,e_4,\{y_4, q\}$ are contained in $\omega'$, and $|\omega'|=4$ implies that $p=q$. This proves the lemma.

\begin{figure}[htbp]
 \begin{center}
   \begin{tikzpicture}
    \node at (0,0){\includegraphics[width=0.39\linewidth]{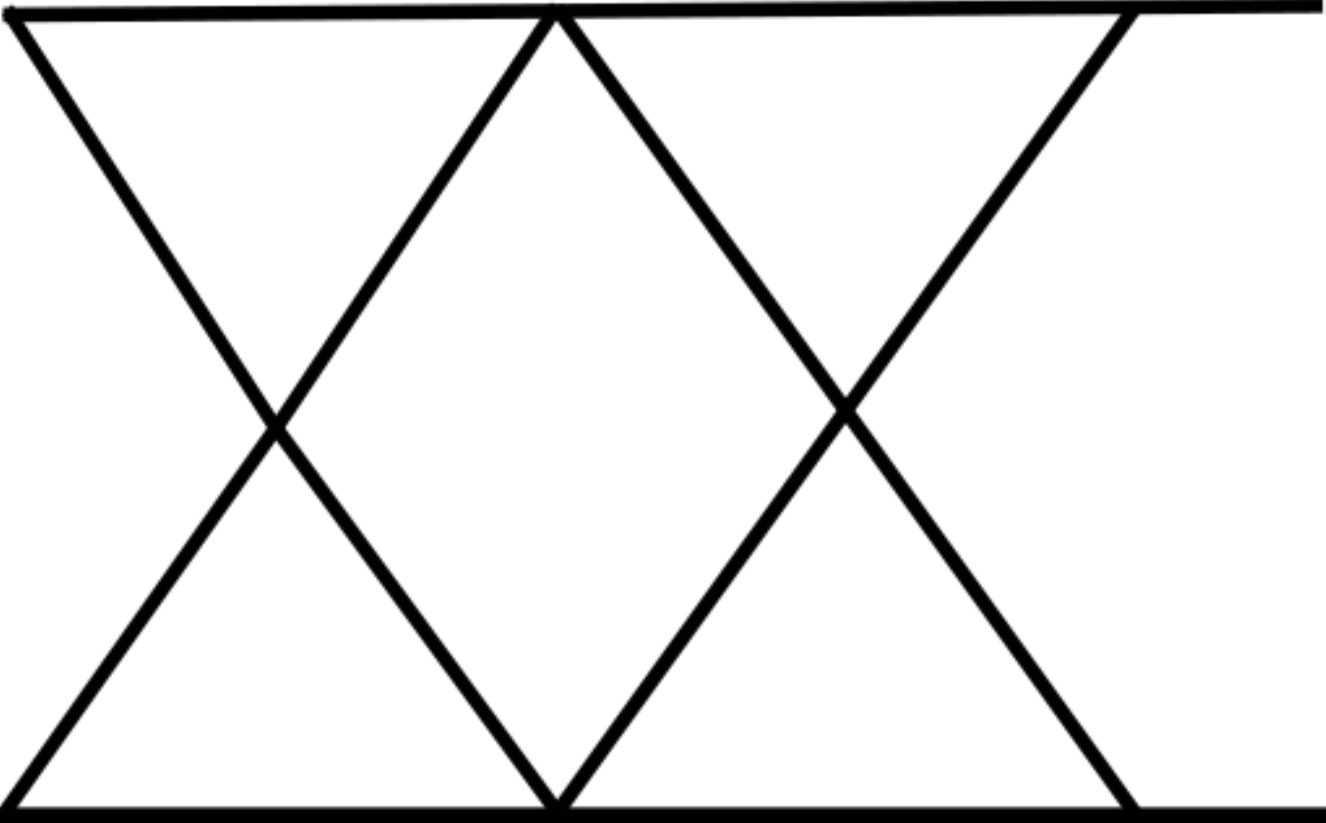}};

   \node at (-2.2, -.0){\small $\tau'$};
   \node at (.6, -.4){\small $q$};
   \node at (.75, .3){\small $p$};
   \node at (-1.4,  -.9){\small $\omega$};

     \node at (.6,  -.9){\small $\omega_1$};
   \node at (-2.2,  1.7){\small $y_2$};
   \node at (-0.4, 1.7){\small $y_1$};
   \node at (-1.4,  .9){\small $\tau$};
   \node at (.6,  .9){\small $\tau_1$};
   \node at (-2.2,  -1.8){\small $y_3$};
   \node at (-0.3,  -1.8){\small $y_4$};
   \node at (-1.2, -.1){\small $x$};
   \node at ( -.3, .0){\small $\omega'$};
   \node at (-1.4,  -2.1){\small $\sigma_2$};
   \node at (-1.4,   2.2){\small $\sigma_1$};
   \node at (-.70, .6){\small $e_1$};
   \node at (-2.2, .7){\small $e_2$};
   \node at (-2.2, -.8){\small $e_3$};
   \node at (-.65, -.8){\small $e_4$};
   \end{tikzpicture}
  \caption{}\label{up-down-triangles-33-4-33}
 \end{center}
\end{figure}

\end{proof}

\begin{theorem}\label{prop:disj2} Let $\Gamma$ be a $4$-regular infinite planar graph satisfying $\Gamma\in \NNG$. For $i=1,2$, let $\sigma_i$ be a face with facial degree $k_i\geq 8$. Then
\begin{align}\begin{split}\label{1-neighborhood-non-adj}
\d U_1(\sigma_1) \cap \d U_1(\sigma_2)=\varnothing.
\end{split}\end{align}
\end{theorem}

\begin{proof}
Note that
for any vertex $z\prec \sigma_i$,
\begin{align}\begin{split}\label{vexpttn-geq8}
\Pttn(z)=(3,3,3,k_i)\ \textrm{or} \ (3,3,4,k_i), i=1,2.
\end{split}\end{align}
We will prove (\ref{1-neighborhood-non-adj}) by contradiction. Suppose that
\[\d U_1(\sigma_1) \cap \d U_1(\sigma_2)\neq \varnothing.\]
For any vertex $z\in \d U_1(\sigma_1) \cap \d U_1(\sigma_2)$, we have $z\notin \d\sigma_1\cup\d \sigma_2$ and there exist two vertices, $z_1(\prec\sigma_1)$ and $z_2(\prec \sigma_2)$, such that $z\sim z_1$ and $z\sim z_2$. Hence, there are two faces, denoted by $\tau$ and $\omega$, which are lower-adjacent to $\sigma_1$ and $\sigma_2$ respectively, such that $z\in (\partial\tau\setminus\partial\sigma_1)\cap(\partial\omega\setminus\partial\sigma_2)$.
By Proposition \ref{prop:nonadj}, we obtain $|\tau|=3,4$  and  $|\omega|=3,4$. We divide it into cases.


\begin{figure}[htbp]
 \begin{center}
   \begin{tikzpicture}
    \node at (0,0){\includegraphics[width=0.5\linewidth]{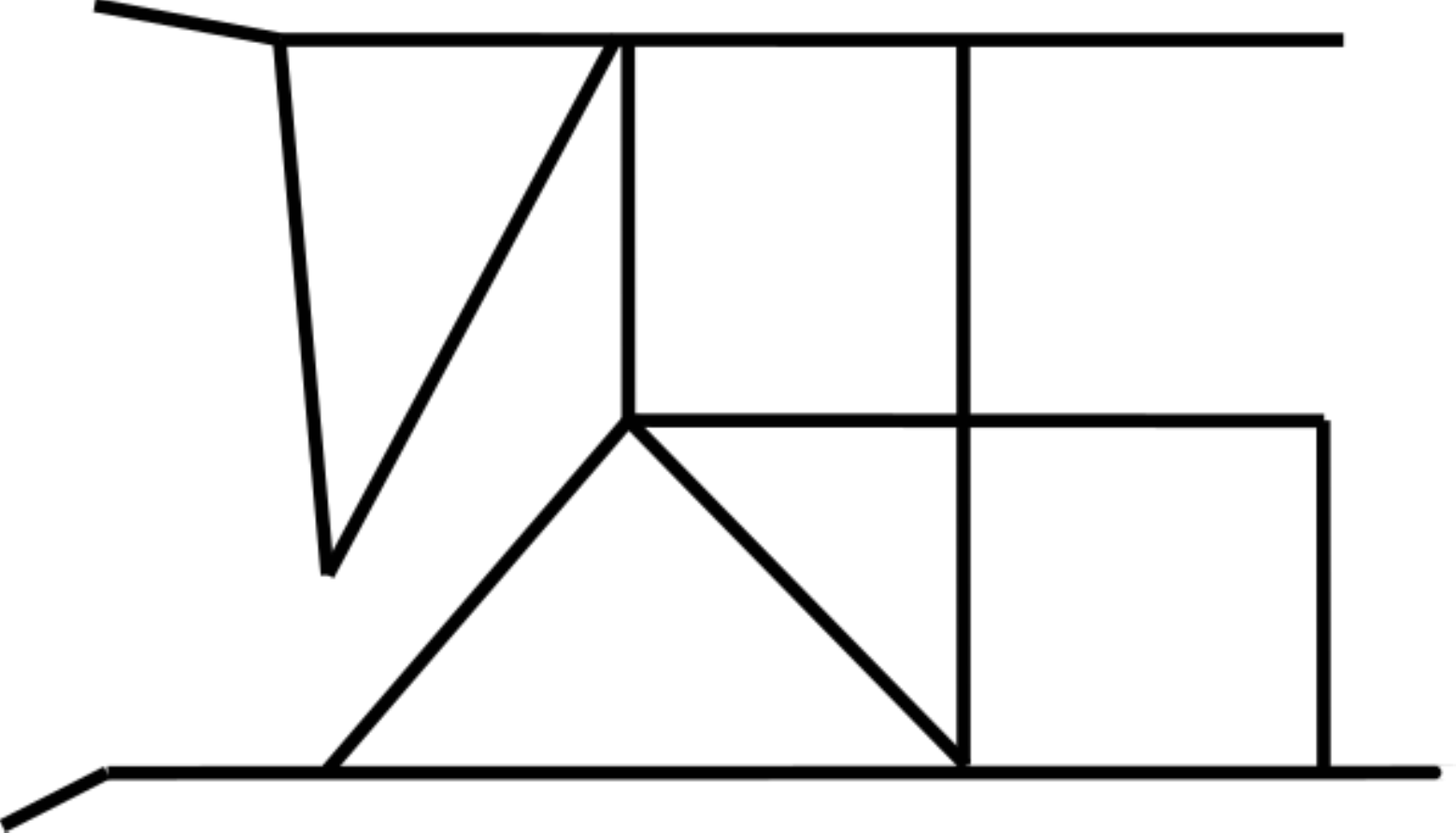}};
    \node at (1.3,  -.2){\Large $x$};
    \node at (1.3,  .8){\Large $e_1$};
    \node at (.3,  .2){\Large $e_2$};
    \node at (-.2,  .8){\Large $e_3$};
    \node at (1.8,  .2){\Large $e_5$};
    \node at (2.9,  -.8){\Large $e_6$};
    \node at (.7,  -.8){\Large $e_4$};
    \node at (2.9,  .3){\Large $y_5$};
    \node at (1.0,  -2.0){\Large $y_4$};
    \node at (1.0,   2.0){\Large $y_1$};
    \node at (1.8,  -.8){\Large $\omega$};
    \node at (.2,  .9){\Large $\tau$};
    \node at (-.7,  -1.1){\Large $\tau_1$};
    \node at (-.4,  -.4){\Large $x_1$};
    \node at (-.65,   .15){\Large $y_2$};
    \node at (-.7,   2){\Large $y_3$};
    \node at (-1.8,  -2.0){\Large $z$};
    \node at (-2.1,  -.7){\Large $x_2$};
    \node at (-1.6,  .7){\Large $\tau_2$};
    \node at (-2.7,  -2.5){\Large $\sigma_2$};
    \node at (-2.2,  2.5){\Large $\sigma_1$};
    \end{tikzpicture}
  \caption{}\label{two-square-adj1}
 \end{center}
\end{figure}

\textbf{Case 1.} $|\tau|=|\omega|=4$.
We divide it into subcases.

\textbf{Case 1.1.}  $\#\{(\partial\tau\setminus\partial\sigma_1)\cap(\partial\omega\setminus\partial\sigma_2)\}=1.$ Set $x=\partial \tau \cap \partial \omega.$ As in Figure~\ref{two-square-adj1}, let $e_1,e_2,e_3$ (resp. $e_4,e_5,e_6$) be the edges ordered clockwise along $\sigma_1$ (resp. $\sigma_2$) which are contained in $\tau$ (resp. in $\omega$), but not in $\sigma_1$(resp. $\sigma_2$). We write $e_i=\{x,y_i\}$ for $i=1,2,4,5$ and $y_1\prec\sigma_1, y_4\prec\sigma_2$, $y_2=e_2\cap e_3$, $y_5=e_5\cap e_6$ and $y_3=e_3\cap \sigma_1$.
Clearly, $y_2\nprec\sigma_2$ and $y_5\nprec\sigma_1$. Otherwise, $|y_1|=|y_4|=3$, which is a contradiction.
Let $\tau_1$ be a face $\sigma_1$-adjacent to $\omega$ and $\partial\tau_1\cap\partial\omega=y_4$.  Set $x_1=\partial\tau_1\setminus\partial\sigma_2$.
By $|x|=4$, $e_2, e_4$ are contained in a face. Moreover, $|y_4|=4$ and $\Pttn(y_4)=(3,3,4,k_2)$ which yields $e_4, \{x_1,y_4\}$ are contained in a triangle. Then there exists a triangle containing $e_2=\{x,y_2\}$, $e_4=\{x,y_4\}$,  which implies $y_2=x_1$. Let $z\prec \sigma_2$ such that $\partial\tau_1=\{y_2,z,y_4\}$. The fact $|y_2|=4$ implies that $e_3,\{y_2,z\}$ are contained in a face.

Similarly, let $\tau_2$ be a face $\sigma_1$-adjacent to $\tau$ and $\partial\tau_2\cap\partial\tau=y_3$. Let $x_2=\partial\tau_2\setminus \partial\sigma_1$. By $|y_3|=4$ and $\Pttn(y_3)=(3,3,4,k_1)$, we obtain $e_3$ and $\{y_3,x_2\}$ are contained in a triangle. Hence, there exists a triangle containing $e_3=\{y_2,y_3\}$, $\{y_2,z\}$ and $\{y_3,x_2\}$, which implies $z=x_2$. This contradicts $|z|=4$.
\begin{figure}[htbp]
 \begin{center}
   \begin{tikzpicture}
    \node at (0,0){\includegraphics[width=0.5\linewidth]{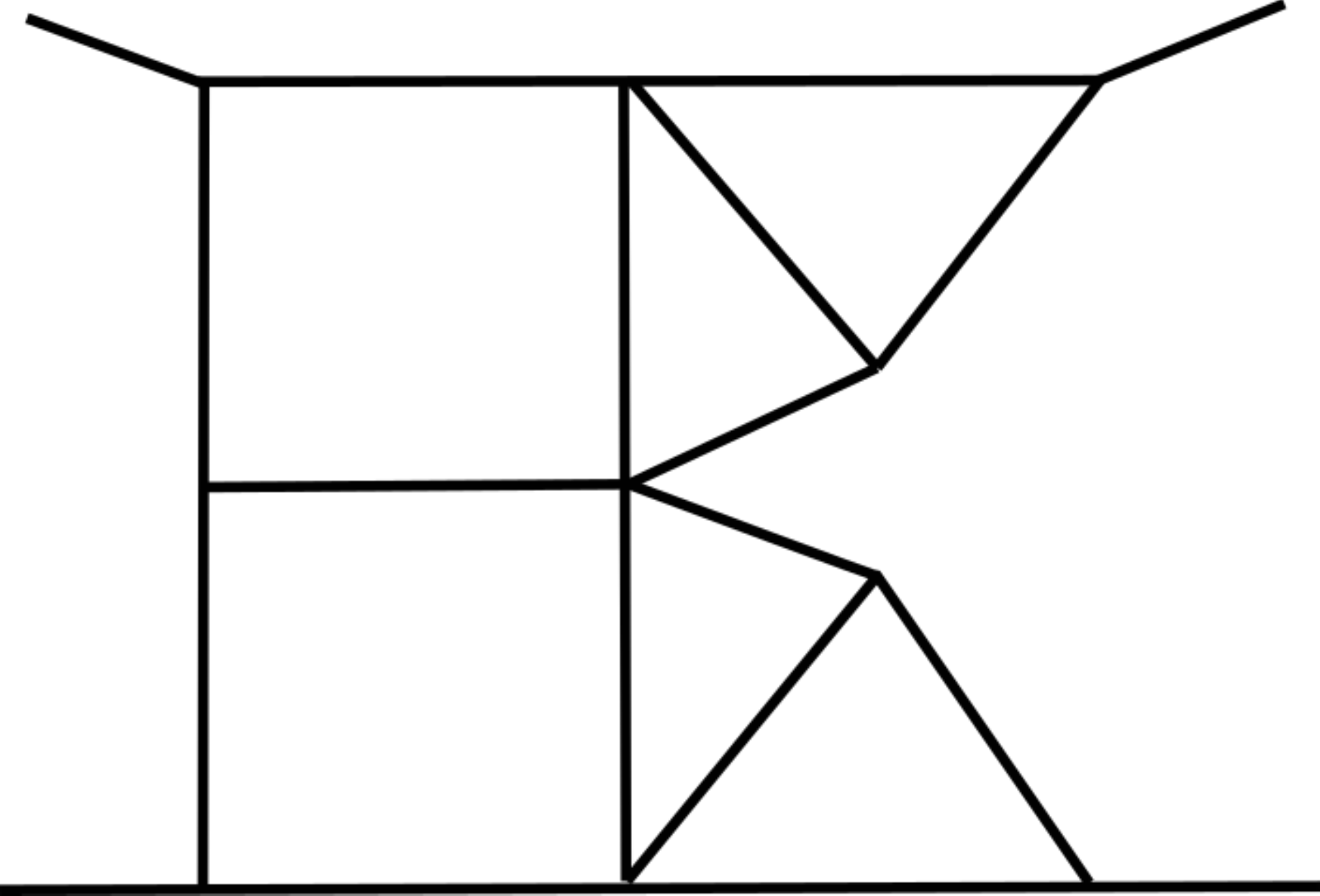}};
    \node at (.3, -.15){\Large $x$};
    \node at (-2.4, -.2){\Large $y$};
    \node at (-1,  -1.2){\Large $\omega$};
    \node at (-1,  0.8){\Large $\tau$};
    \node at (1,  1.2){\Large $\tau_1$};
    \node at (1.3,  0.2){\Large $z_1$};
    \node at (1.3,  -.4){\Large $z_2$};
    \node at (1,  -1.4){\Large $\tau_2$};
    \node at (-1.6,  -2.5){\Large $\sigma_2$};
    \node at (-1.6,  2.3){\Large $\sigma_1$};
    \end{tikzpicture}
    \caption{}\label{two-squares-two-commvertex}
 \end{center}
\end{figure}

\textbf{Case~1.2.} $\#\{(\partial\tau\setminus\partial\sigma_1)\cap(\partial\omega\setminus\partial\sigma_2)\}=2.$ Let $\{x,y\}=\partial \tau \cap \partial \omega.$ By the orientability of $\R^2,$ we have the case as in Figure~\ref{two-squares-two-commvertex}. Let $\tau_1$ (resp. $\tau_2$) be a face $\sigma_1$-adjacent (resp. $\sigma_2$-adjacent) to $\tau$ (resp. $\omega$) as in Figure~\ref{two-squares-two-commvertex}. By (\ref{vexpttn-geq8}), $|\tau_i|=3,$ $i=1,2.$ Set $z_i=\partial \tau_i\setminus \partial \sigma_i$ for $i=1,2.$
Note that $z_1\neq x$ and $z_2\neq x$. By considering the vertex patterns of $\partial \tau_1\cap \partial \tau$ (resp. $\partial \tau_2\cap \partial \omega$), we get that $z_1\sim x$ ($z_2\sim x$). If $z_1\neq z_2$, then $|x|=5$, this contradicts $|x|=4.$ If $z_1=z_2$, then $|z_1|\geq 5$, which contradicts that the graph is $4$-regular.

\textbf{Case 2.} One of $\tau$ and $\omega$ is a triangle and the other is not, say $|\tau|=4$ and $|\omega|=3.$

\begin{figure}[htbp]
 \begin{center}
   \begin{tikzpicture}
    \node at (0,0){\includegraphics[width=0.4\linewidth]{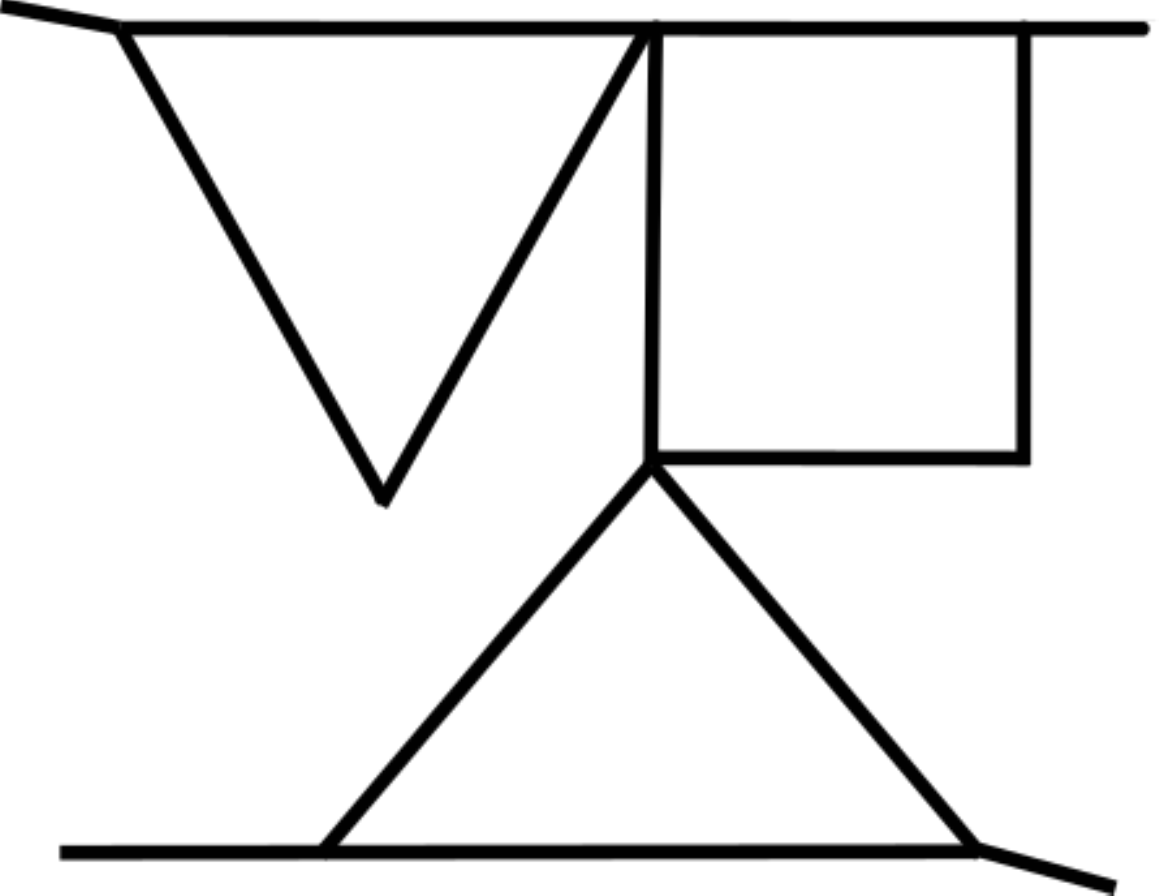}};
    \node at (.0, -.1){\Large $x$};
    \node at (2.1, -.1){\Large $y$};
    \node at (.4,  -1.2){\Large $\omega$};
    \node at (1.1, .9){\Large $\tau$};
    \node at (-.7,   1.1){\Large $\tau_1$};
    \node at (-.9,    -.5){\Large $x_1$};
    \node at (-2.1,  -2.2){\Large $\sigma_2$};
    \node at (-2.1,   2.2){\Large $\sigma_1$};
   \end{tikzpicture}
  \caption{}\label{triangle-square-two-adj}
 \end{center}
\end{figure}

\textbf{Case 2.1.} $\#\{(\partial\tau\setminus\partial\sigma_1)\cap(\partial\omega\setminus\partial\sigma_2)\}=1$. Let $x=\partial\tau\cap\partial\omega$, $x_1=\partial\tau_1\setminus\partial\sigma_1$, as shown in Figure \ref{triangle-square-two-adj}. With the same argument as $\tau$, $\tau_1$, $\tau_2$ in Case 1.2, we obtain that the vertex $x_1\prec(\sigma_2\cap\omega)$, which implies that $|x_1|=5$, a contradiction.


\textbf{Case 2.2.} $\#\{(\partial\tau\setminus\partial\sigma_1)\cap(\partial\omega\setminus\partial\sigma_2)\}=2$. Let $\{x,y\}=\partial\tau\cap\partial\omega$ and $x\nprec\sigma_1$ and $x\nprec\sigma_2$.
Assume that $\partial\tau\cap\partial\sigma_1=\{z_1,z_2\}$ and $x\sim z_1$. By Proposition \ref{non-square-lower-triangle}, we obtain that this case is impossible.


\begin{figure}[htbp]
\begin{center}
   \begin{tikzpicture}
    \node at (0,0){\includegraphics[width=0.3\linewidth]{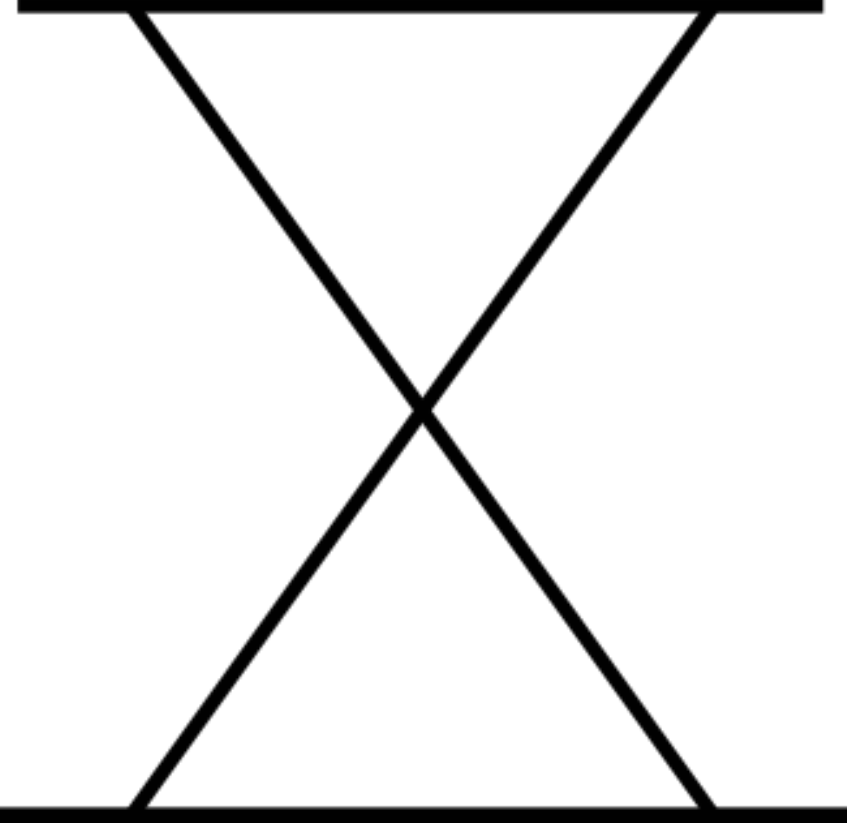}};
    \node at (.3, -.0){\Large $x$};
    \node at (1.2, -.1){\Large $\omega'$};
    \node at (1.0, .8){\Large $e_1$};
    \node at ( -.9, .8){\Large $e_2$};
    \node at (.9, -.8){\Large $e_4$};
    \node at (-.9, -.8){\Large $e_3$};
    \node at (0,  -1){\Large $\omega$};
    \node at (0,  1.0){\Large $\tau$};
    \node at (-.8,  - .1){\Large $\tau'$};
    \node at (-1.4,  -2.3){\Large $\sigma_2$};
    \node at (-1.4,   2.3){\Large $\sigma_1$};
   \end{tikzpicture}
  \caption{}\label{up-down-trian-only}
 \end{center}
\end{figure}

\textbf{Case 3.} $|\tau|=|\omega|=3$. Obviously, $\#\{(\partial\tau\setminus\partial\sigma_1)\cap(\partial\omega\setminus\partial\sigma_2)\}=1$.
Let $x=(\partial\tau\setminus\partial\sigma_1)\cap(\partial\omega\setminus\partial\sigma_2)$.  Let $e_1,e_2$ (resp. $e_3,e_4$) be two edges contained in $\tau$ (resp. in $\omega$) ordered clockwise along $\tau$ (resp. $\omega$), but not in $\sigma.$ Since $|x|=4,$ there exists a face $\tau'$(resp. $\omega'$) containing $e_2,e_3$ (resp. $e_1,e_4$), as shown in Figure \ref{up-down-trian-only}. We divide it into subcases.

\textbf{Case 3.1.} $|\omega'|=3$. By $|\tau'|=3$ or $4$, $\Pttn(x)=(3,3,3,3)$ or $(3,3,3,4)$, and
\[
\Phi(x)\geq\frac{1}{4}.
\]

\textbf{Case 3.2.} $|\omega'|=4$. By $|\tau'|=3$ or $4$, $\Pttn(x)=(3,3,3,4)$ or $(3,3,4,4)$, and
\[
\Phi(x)\geq \frac{1}{6}.
\]

Without loss of generality, we assume that $k_1\leq k_2$. 
Let $\{\tau_i\}_{i=1}^{k_1-1}$ be triangles lower-adjacent to $\sigma_1$ ordered clockwise along $\sigma_1$ such that $\tau_i$ are $\sigma_1$-adjacent to $\tau_{i+1}$ for $1\leq i\leq k_1-1$, $\tau$ is $\sigma_1$-adjacent to $\tau_1$, and $\tau_{k_1-1}$ is $\sigma_1$-adjacent to $\tau$. By Proposition \ref{prop:nonadj}, $|\tau_i|=3$ or $4$ for $1\leq i\leq k_1-1$.
If $|\tau_1|=3$, by Lemma \ref{all-faces-type}, we obtain that there is a triangle, denoted by $\tau_1'$, such that $\tau_1'$ is $\sigma_2$-adjacent to $\omega$ and $\d\tau_1\setminus\d\sigma_1= \d\tau_1'\setminus\d \sigma_2$. This is Case 3.2, i.e. $\Phi(x)\geq \frac{1}{6}$.
If $|\tau_1|=4$, Lemma \ref{all-faces-type} yields that $\tau_1$ is lower-adjacent to both $\sigma_1$ and $\sigma_2$, then by Case 3.1, $\Phi(x)\geq \frac{1}{4}$. Moreover, Lemma \ref{square-to-trianglespair} yields $|\tau_2|=3$ and there is a triangle, denoted by $\tau_2'$, such that $\tau_2'$ is $\sigma_2$-adjacent to $\omega$ and $\d\tau_2\setminus\d\sigma_1= \d\tau_2'\setminus\d \sigma_2$. Repeating the process, we obtain for any triangle $\tau_j\in\{\tau_i\}_{i=1}^{k_1-1}$, $\Phi(\d\tau_j\setminus\d\sigma_1)\geq \frac{1}{6}$ if $|\tau_{j+1}|=3$ and $\Phi(\d\tau_j\setminus\d\sigma_1)\geq \frac{1}{4}$ if $|\tau_{j+1}|=4$.

Set $S=\tau\cup\{\tau_i\}_{i=1}^{k_1-1}$. Suppose there are $l$ squares in $S$, then there are $l$ triangles whose estimates are same as $\tau$ in Case 3.1 and the remaining $k_1-2l$ triangles whose estimates are same as $\tau$ in Case 3.2. Note that for any vertex $z\prec\sigma_1$, $\Phi(z)=\frac{1}{k_1}-\frac{1}{12}$. We can estimate the sum of the curvature of vertices in $U_1(\sigma_1)$ as follows.
\begin{align}\begin{split}
\sum_{z\in U_1(\sigma_1)}\Phi(z)
&=\sum_{z\in \d U_1(\sigma_1)}\Phi(z)+\sum_{z\in \d \sigma_1}\Phi(z)\\
&\geq l\times\frac{1}{4}+(k_1-2l)\times\frac{1}{6}+k_1\left(\frac{1}{k_1}-\frac{1}{12}\right)\\
&=1+\frac{k_1-l}{12}>1,
\end{split}\end{align}
where we used $l\leq [\frac{k_1}{2}]$ by the fact that two squares can not be $\sigma_1$-adjacent to each other, see Proposition \ref{prop:nonadj}.
This contradicts Theorem \ref{thm:DeVosMohar}.
\end{proof}

Before stating the next proposition, we introduce a modified curvature notion $\Phi'$. For any fixed subset $B\subset V$ satisfying $B\setminus (\d\sigma_1\cup\d\sigma_2)\neq \varnothing$. We define
\begin{align}\begin{split}\label{mod-curvature1}
\Phi'(z)=\Phi(z)+\frac{1}{N}\sum_{z'\in B\setminus\left(\partial\sigma_1\cup \partial\sigma_2\right)}\Phi(z'), \ \ \forall z\in B\cap\left(\partial\sigma_1\cup \partial\sigma_2\right),
\end{split}\end{align}
where
\[
 N:=\#\{z|z \in B\cap\left(\partial\sigma_1\cup \partial\sigma_2\right)\}.
\]
Note that $\Phi'$ depends on the choice of $B$. By using this modified curvature, we distribute the curvature of vertices in $B\setminus (\partial\sigma_1\cup\partial\sigma_2)$ to the vertices in $B\cap(\partial\sigma_1\cup\partial\sigma_2)$ evenly.

\begin{prop}\label{non-square-lower-triangle}
Let $\sigma_1$ be a $k_1$-gon and $\sigma_2$ be a $k_2$-gon with $k_1,k_2\geq 8$. The following configuration is impossible: There are a square lower-adjacent to $\sigma_1$ and a triangle is lower-adjacent to $\sigma_2$ such that the square is lower-adjacent to the triangle and the triangle has a vertex neither in $\sigma_1$ nor in $\sigma_2$.
\end{prop}

\begin{proof}
We will prove it by contradiction. Assume that there exist a square $\tau$, lower-adjacent to $\sigma_1$, and a triangle $\omega$, lower-adjacent to $\sigma_2$, such that $\tau$ lower-adjacent to $\omega$ and $p=(\partial\omega\setminus\partial\sigma_2)\nprec\sigma_1$.
Let $\{x,x'\}=\tau\cap\sigma_1$, $\{y,y_1\}=\omega\cap\sigma_2$, and $\{p,y\}=\tau\cap\omega$, as shown in Figure \ref{triangle-square-oneedge-com-begin}.

\begin{figure}[htbp]
 \begin{center}
   \begin{tikzpicture}
    \node at (0,0){\includegraphics[width=0.45\linewidth]{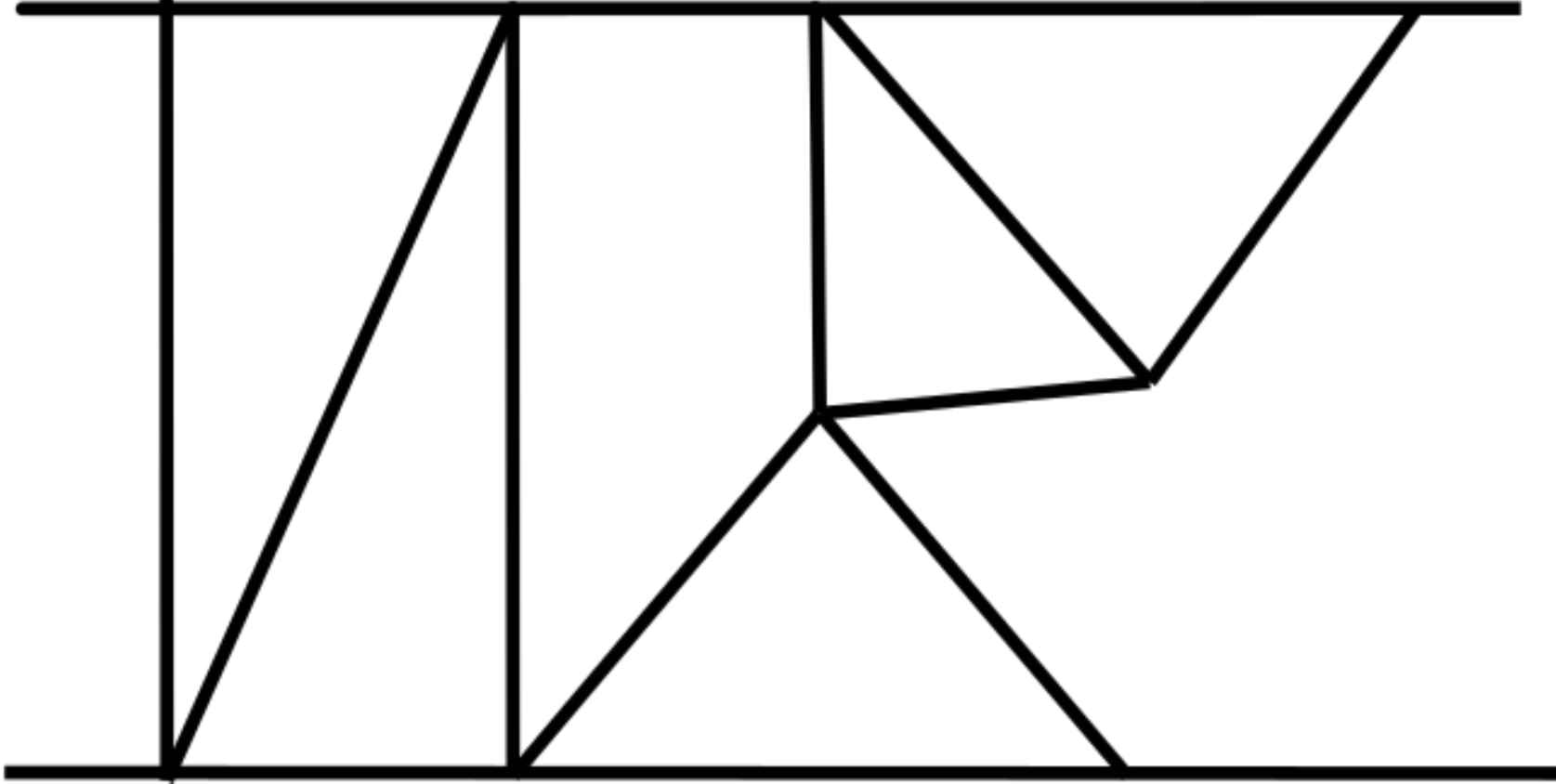}};
    \node at (2.4, 1.6){\small $x_1$};
    \node at (1.3, .9){\small $\tau_1$};
    \node at ( .5, .6){\small $\tau'$};
    \node at ( -.5, .3){\small $\tau$};
    \node at (-.1,  -.1){\small $p$};
    \node at (1.5, -.2){\small $p_1$};
    \node at (1.2, -1.6){\small $y_1$};
    \node at (-1.0, -1.6){\small $y$};
    \node at (.2,  -.8){\small $\omega$};
    \node at (-1.4,   -.3){\small $\tau_0$};
    \node at (-1.0,  -2.0){\small $\sigma_2$};
    \node at (-1.0,   2.1){\small $\sigma_1$};
   \node at (-1,   1.7){\small $x'$};
    \node at (.2,   1.6){\small $x$};
   \end{tikzpicture}
  \caption{}\label{triangle-square-oneedge-com-begin}
 \end{center}
\end{figure}

Since $k_1,k_2\geq 8$, we obtain for any $z\prec \sigma_i$, $i=1$ or $i=2$. \begin{align}\begin{split}\label{pattern}
\Pttn(z)=(3,3,3,k_i) \ \textrm{or} \ (3,3,4,k_i).
\end{split}\end{align}
Without loss of generality, we assume $p\sim x$, as shown in Figure \ref{triangle-square-oneedge-com-begin}. 
By $|\tau|=4$, $x\prec\tau$ and (\ref{pattern}), we obtain $\Pttn(x)=(3,3,4,k_1)$. Then there is a triangle $\tau_1$, $\sigma_1$-adjacent to $\tau$. Set $p_1=\partial\tau_1\setminus\partial\sigma_1$, $\{x,x_1\}=\tau_1\cap\sigma_1$ and $x\prec\tau$. Since $\Pttn(x)=(3,3,4,k_1)$ and $|x|=4$, $\{p,x\}$ and $\{x,p_1\}$ are contained in a triangle, denoted by $\tau'$. $|\tau'|=3$ implies that $p\sim p_1$. Clearly, $p_1\nprec \sigma_2$. Otherwise, $p_1=y_1$ or $\{p_1,y_1\}\prec\sigma_2$, this yields $|y_1|=5$ or $|y_1|=3$, which contradicts $|y_1|=4$. Hence,
\begin{align}\begin{split}\label{tau-tau'-tau_1}
|\tau|=4,\ |\tau'|=|\tau_1|=|\omega|=3.
\end{split}\end{align}

Without loss of generality, we assume that $k_1\leq k_2$. Next we will prove that there are at least $2(k_1-1)$ vertices in $\partial\sigma_1\cup\partial\sigma_2$, and the modified curvature $\Phi'$ of each vertex of them is bounded below by $\frac{1}{12}$. We divide it into two steps.

\textbf{Step 1.}  For some subset $B$ to be chosen, we will prove that for any $z\in B\cap(\partial\sigma_1\cup\partial\sigma_2)$, the modified curvature
\begin{align}\begin{split}\label{lower-bdd-curv-of-vertinsigma}
\Phi'(z)\geq \frac{1}{12}.
\end{split}\end{align}
By (\ref{mod-curvature1}) and $\Phi(z)\geq 0$, it suffices to show
\begin{align}\begin{split}\label{mid-ver-cur}
\sum_{z'\in B\setminus\left(\partial\sigma_1\cup \partial\sigma_2\right)}\Phi(z')\geq \frac{N}{12}.
\end{split}\end{align}

Let $\omega'$ be a face containing the two edges $\{p,y_1\}$ and $\{p,p_1\}$.
We divide it into cases.
\begin{figure}[htbp]
 \begin{center}
   \begin{tikzpicture}
    \node at (0,0){\includegraphics[width=0.55\linewidth]{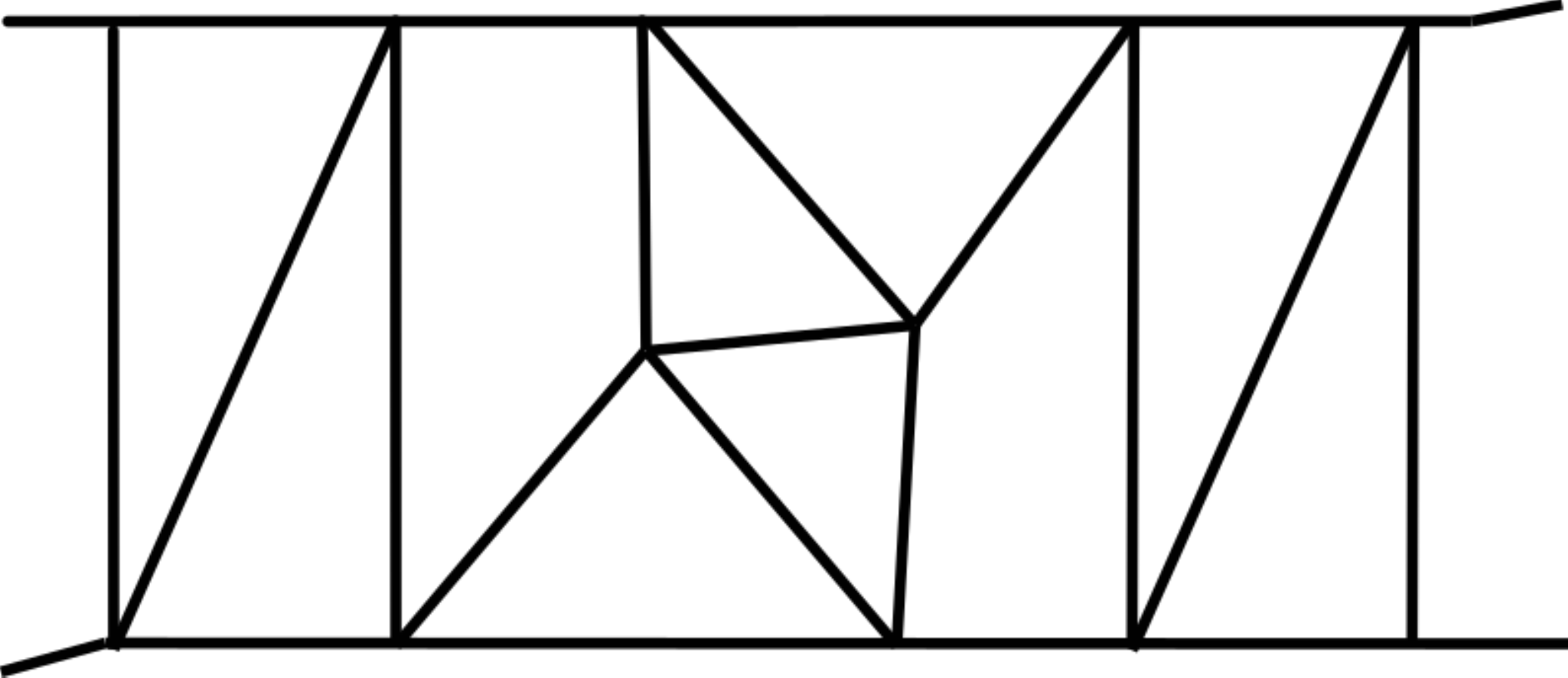}};
    \node at (1.5, 1.6){\small $x_1$};
    \node at (.6, .9){\small $\tau_1$};
    \node at (-.2, .6){\small $\tau'$};
    \node at (.6, .4){\small $p_1$};
    \node at (0.6, -1.6){\small $y_1$};
    \node at (1.6, -1.6){\small $y_1'$};
    \node at (-1.7, -1.6){\small $y$};
    \node at (-.5,  -.8){\small $\omega$};
    \node at (.3,  -.4){\small $\omega'$};
    \node at (1.1,  -.3){\small $\omega_1$};
    \node at (-.3,  -.2){\small $p$};
    \node at (-1.2,   .1){\small $\tau$};
    \node at (-1.0,  -2.0){\small $\sigma_2$};
    \node at (-1.0,   2.1){\small $\sigma_1$};
   \node at (-1.7,   1.7){\small $x'$};
    \node at (-.6,   1.6){\small $x$};
   \end{tikzpicture}
  \caption{}\label{triangle-square-oneedge-com}
 \end{center}
\end{figure}

\textbf{Case 1.} $|\omega'|=3$, as shown in Figure \ref{triangle-square-oneedge-com}. Then $p_1\sim y_1$.

\textbf{Claim:} There exists a square lower-adjacent to $\sigma_2$, denoted by $\omega_1$ such that $\omega_1$ is $\sigma_2$-adjacent to $\omega$, $\omega_1\cap\omega'=\{p_1,y_1\}$, and $\omega_1\cap\tau_1=\{p_1,x_1\}$.

Let $x_1$ be a vertex such that $\{x,x_1\}=\partial\tau_1\cap\partial\sigma_1$. $|p_1|=4$ yields that $\{p_1,x_1\}$ and $\{p_1,y_1\}$ are contained in a face.
$|y_1|=4$ implies that there exists $y_1'$ satisfying $y_1'\prec\sigma_2$ and $y_1'\neq y$ such that $\{p_1,y_1\}$, $\{y_1,y_1'\}$ are contained in a face. Hence, $\{p_1,x_1\}$, $\{p_1,y_1\}$ and $\{y_1,y_1'\}$ contained in a face, denoted by $\omega_1$. By $\Pttn(y_1)=(3,3,3,k_2)$ or $(3,3,4,k_2)$, $|\omega_1|=3$ or $4$. Clearly, $|\omega_1|=4$. Otherwise, $|\omega_1|=3$ yields $x_1=y_1'\prec\sigma_2$. Then $\Phi(x_1)=(3,3,k_1,k_2)$. This yields $\Phi(x_1)<0$  for $k_1,k_2\geq 8$, which contradicts $\Phi(x_1)\geq 0$. Since $\{p_1,y_1\}\prec\omega'$, $\{p_1,x_1\}\prec\tau_1$ and $\{ y_1,y_1'\}\prec\sigma_2$, together with $\Gamma$ is a tessellation, $\{p_1,y_1\}=\omega'\cap \omega_1$, $\{p_1,x_1\}=\tau_1\cap\omega_1$ and $\{y_1,y_1'\}=\sigma_2\cap\omega_1$. This proves the claim.

Choose
\[
B=\partial\omega\cup\partial\tau\cup \partial\omega_1\cup\partial\tau_1.
\]
Since $B\setminus\left(\partial\sigma_1\cup \partial\sigma_2\right)=\{p,p_1\}$, $|\omega_1|=4$, $|\tau'|=3$, and (\ref{tau-tau'-tau_1}), we obtain
\[
\Pttn(p)=\Pttn(p_1)=(3,3,3,4),
\]
and
\begin{align}\begin{split}\label{p-p_1-1/4}
\ \Phi(p)=\Phi(p_1)=\frac{1}{4}.
\end{split}\end{align}
It is easy to check that $N=6$ and
\[\sum_{z'\in B\setminus\left(\partial\sigma_1\cup \partial\sigma_2\right)}\Phi(z')=\Phi(p)+\Phi(p_1)=\frac{1}{2}=\frac{N}{12}\]
by (\ref{p-p_1-1/4}). This proves (\ref{mid-ver-cur}).

\begin{figure}[htbp]
 \begin{center}
   \begin{tikzpicture}
    \node at (0,0){\includegraphics[width=0.65\linewidth]{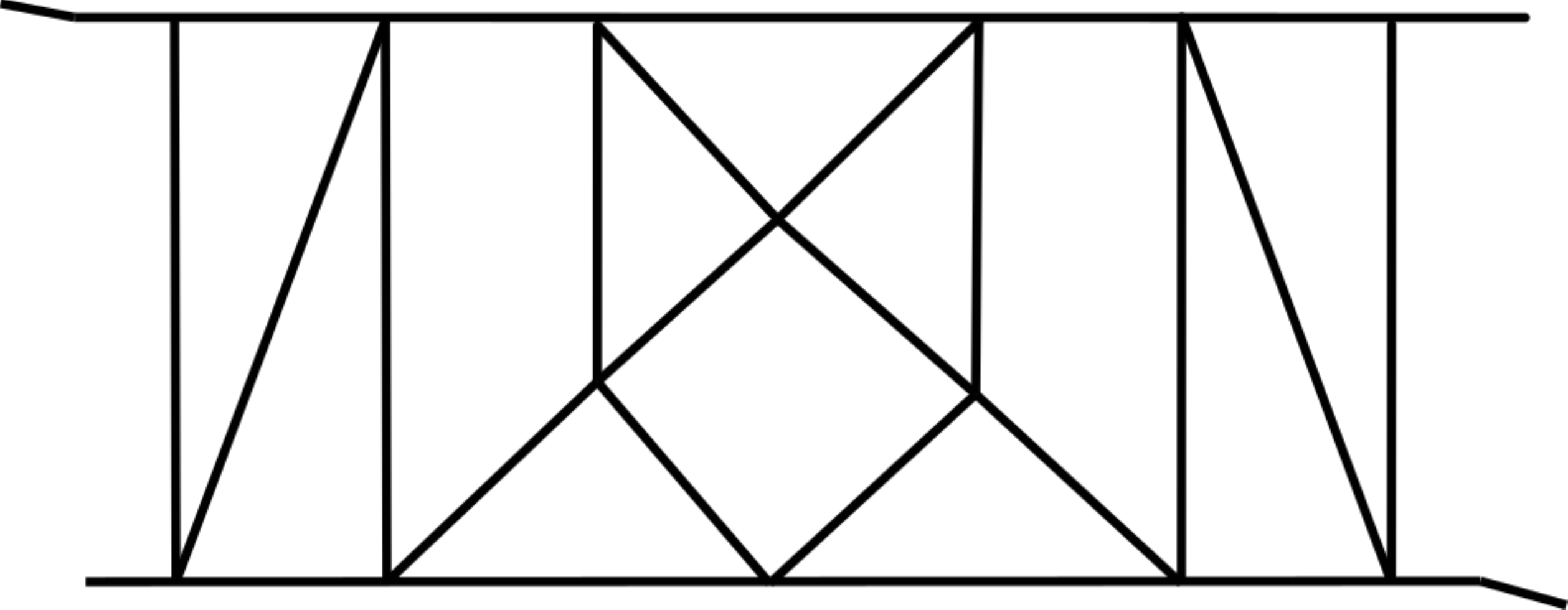}};
    \node at (1.2, 1.7){\small $x_1$};
    \node at (2.2, 1.8){\small $x_1'$};
    \node at (.0,1.0){\small $\tau_1$};
    \node at (.0, .15){\small $p_1$};
    \node at (.65, .5){\small $\tau'_1$};
    \node at (1.5, .5){\small $\tau_2$};
    \node at (.15, -.4){\small $\omega'$};
    \node at (-.0, -1.7){\small $y_1$};
    \node at (1.3, -.3){\small $q_1$};
    \node at (2, -1.7){\small $y_2$};
    \node at (-2.1, -1.7){\small $y$};
    \node at (-.9,  -1.0){\small $\omega$};
    \node at (1,  -1){\small $\omega_1$};
    \node at (-.7,  -.4){\small $p$};
     \node at (-.7,  .4){\small $\tau'$};
    \node at (-1.6,   .1){\small $\tau$};
    \node at (-2.5,  -2 ){\small $\sigma_2$};
    \node at (-2.5,   2 ){\small $\sigma_1$};
   \node at (-2,   1.75){\small $x'$};
    \node at (-.9,   1.7){\small $x$};
   \end{tikzpicture}
  \caption{}\label{triangle-square-oneedge-com2}
 \end{center}
\end{figure}

Let $\tau_1'$ be a face, lower-adjacent to $\tau_1$, such that  $\{p_1,x_1\}=\tau_1\cap\tau_1'$.

\textbf{Case 2.} $|\omega'|=4$, $|\tau_1'|=3$.

\textbf{Claim 1:} There exists a triangle $\omega_1$, $\sigma_2$-adjacent to $\omega$, such that $q_1=\partial\omega_1\setminus\partial\sigma_2$, $\{y_1,y_2\}=\omega_1\cap\sigma_2$, $y_1\prec\omega$, $q_1\nprec\sigma_1$ and  $q_1\sim p_1$, as shown in Figure \ref{triangle-square-oneedge-com2}.

By $|\omega'|=4$, $y_1\prec\sigma_2$ and (\ref{pattern}), $\Pttn(y_1)=(3,3,4,k_2)$.
Then there exists a triangle $\omega_1$, $\sigma_2$-adjacent to $\omega$, such that $q_1=\partial\omega_1\setminus\partial\sigma_2$, $\{y_1,y_2\}=\omega_1\cap\sigma_2$ and $y_1\prec\omega$, as shown in Figure \ref{triangle-square-oneedge-com2}. Clearly, $q_1\nprec\sigma_1$. Otherwise,
$q_1=x_1$ or $\{x_1,q_1\}\prec \sigma_1$, which yields $|x_1|=5$ or $3$. This contradicts $|x_1|=4$. $|y_1|=4$ yields that $\{p,y_1\}$ and $\{y_1,q_1\}$ are contained in a face. Using $\{p,y_1\}=\omega\cap\omega'$, $\{p,p_1\}\nprec\omega$ and $\{q_1,y_1\}\nprec\omega$, we obtain that $\{p,p_1\}$, $\{p,y_1\}$ and $\{y_1,q_1\}$ are contained in $\omega'$. $|\omega'|=4$ implies that $p_1\sim q_1$, which proves the claim. Hence,
\begin{align}\begin{split}\label{omega_1}
|\omega_1|=3.
\end{split}\end{align}

\textbf{Claim 2:} There exists a square, denoted by $\tau_2$, lower-adjacent to $\sigma_1$ such that $\{x_1,x_1'\}= \tau_2\cap\sigma_1$, $\{q_1,x_1\}=\tau_1'\cap\tau_2$ and $\{q_1,y_2\}=\tau_2\cap\omega_1$, as shown in Figure \ref{triangle-square-oneedge-com2}.

By $|p_1|=4$, $\{p_1,x_1\}$ and $\{p_1,q_1\}$ are contained in a face. Since $\{p_1,x_1\}=\tau_1\cap\tau_1'$ and $\{p_1,q_1\}\prec\tau_1'$. $|\tau_1'|=3$ implies that $x_1\sim q_1$. $|q_1|=4$ yields that $\{q_1,x_1\}$ and $\{q_1,y_2\}$ are contained in a face. For $|x_1|=4$, there exists $x_1'\prec\sigma_1$ and $x_1'\neq x$ such that $\{q_1,x_1\}$ and $\{x_1,x_1'\}$ are contained in a face. Hence, $\{q_1,x_1\}$, $\{q_1,y_2\}$ and $\{x_1,x_1'\}$ are contained in a face, denoted by $\tau_2$. By (\ref{pattern}), $\Pttn(x_1)=(3,3,3,k_1)$ or $(3,3,4,k_1)$, which implies that $|\tau_2|=3$ or $4$. Clearly, $|\tau_2|=4$. Otherwise, $|\tau_2|=3$ yields $y_2=x_1'\prec\sigma_1$. Then $\Phi(y_2)=(3,3,k_1,k_2)$. So that $\Phi(y_2)<0$ for $k_1,k_2\geq 8$, which contradicts $\Phi(y_2)\geq 0$.  Since $\{q_1,x_1\}\prec\tau_1'$, $\{q_1,y_2\}\prec\omega_1$ and $\{x_1,x_1'\}\prec\sigma_1$, together with $\Gamma$ is a tessellation, $\{q_1,x_1\}=\tau_1'\cap \tau_2$, $\{q_1,y_2\}=\tau_2\cap\omega_1$ and $\{x_1,x_1'\}=\sigma_1\cap\tau_2$, which proves the claim.
Hence,
\begin{align}\begin{split}\label{tau_1'-tau_2}
|\tau_2|=4.
\end{split}\end{align}

Set
\[
B=\partial\tau\cup\partial\tau_1\cup\partial\tau
_2\cup\partial\omega\cup\partial\omega_1.
\]
Using (\ref{tau-tau'-tau_1}), (\ref{omega_1}) and (\ref{tau_1'-tau_2}), together with $|\omega'|=4$ and $|\tau_1'|=3$, we obtain
\begin{align}\begin{split}\label{p-p_1}
\Pttn(p)=\Pttn(q_1)=(3,3,4,4),\ \Pttn(p_1)=(3,3,3,4).
\end{split}\end{align}
It is easy to check that $N=7$ and
\[
\sum\limits_{z'\in B\setminus(\partial\sigma_1\cup\partial\sigma_2)}\Phi(z')=\Phi(p)+\Phi(q_1)+\Phi(p_1)=\frac{7}{12}=\frac{N}{12} \]
by (\ref{p-p_1}).

\begin{figure}[htbp]
 \begin{center}
   \begin{tikzpicture}
    \node at (0,0){\includegraphics[width=0.7\linewidth]{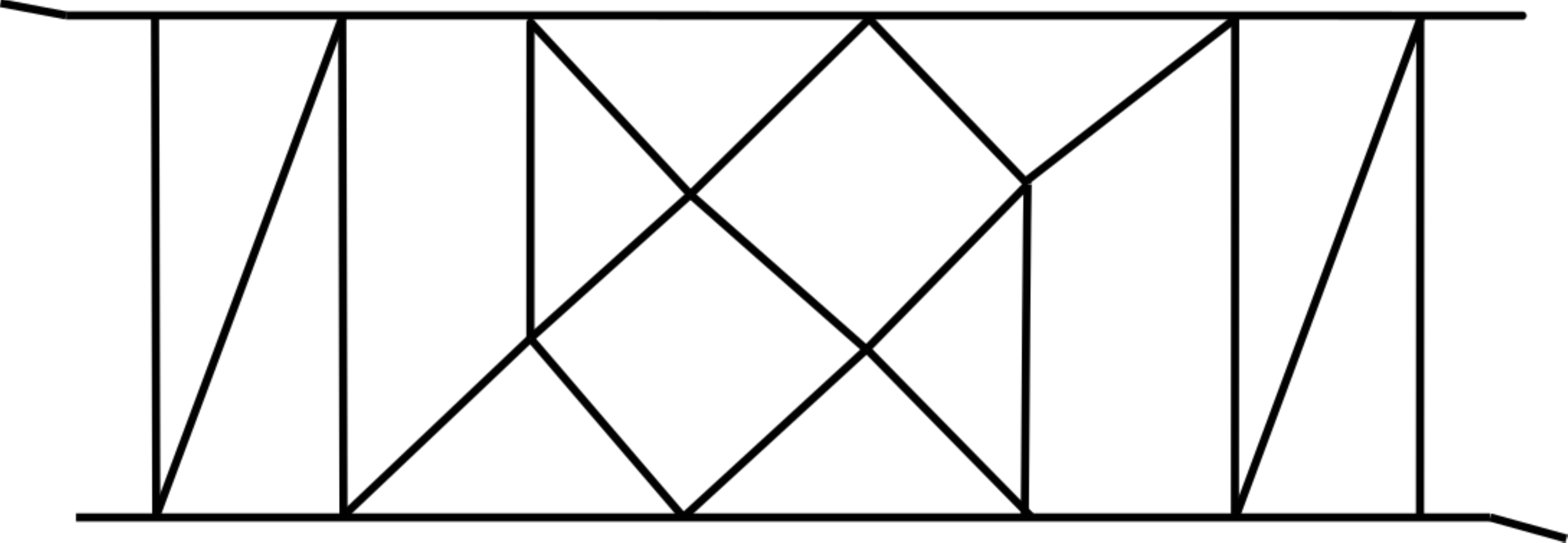}};
    \node at (.6, 1.7){\small $x_1$};
    \node at (2.4, 1.7){\small $x_2$};
    \node at (-.4,1.15){\small $\tau_1$};
    \node at (1.6,1.15){\small $\tau_2$};
    \node at (.6,.6){\small $\tau'_1$};
    \node at (-.1, .5){\small $p_1$};
    \node at (-.6, -1.7){\small $y_1$};
    \node at (.5, -.15){\small $q_1$};
    \node at (1.65,  .4){\small $p_2$};
    \node at (1.3, -1.7){\small $y_2$};
    \node at (2.4, -1.7){\small $y_2'$};
    \node at (-2.4, -1.7){\small $y$};
    \node at (-1.5,  -1.1){\small $\omega$};
    \node at (.5,  -1.1){\small $\omega_1$};
    \node at (-.5,  -.45){\small $\omega'$};
    \node at (1.05,  -.5){\small $\omega_1'$};
    \node at (1.9,  -.3){\small $\omega_2$};
    \node at (-1.2,  -.4){\small $p$};
    \node at (-1.9,   .4){\small $\tau$};
    \node at (-1.2,   .4){\small $\tau'$};
    \node at (-1.0,  -2.2){\small $\sigma_2$};
    \node at (-1.0,   2.2){\small $\sigma_1$};
   \node at (-1.4,   1.7){\small $x$};
    \node at (-2.4,   1.8){\small $x'$};
   \end{tikzpicture}
  \caption{}\label{triangle-square-oneedge-com3}
 \end{center}
\end{figure}

Let $\omega_1'$ be a face, lower-adjacent to $\omega_1$, such that  $\{q_1,y_2\}=\omega_1\cap\omega_1'$.

\textbf{Case 3.} $|\omega'|=4$, $|\tau_1'|=4$, $|\omega_1'|=3$.

For $|\tau_1'|=4$, with the similar argument as in Claim 1 of Case 2, we obtain that there exists a triangle $\tau_2$ such that $\tau_2$ is $\sigma_1$-adjacent to $\tau_1$, $\{x_1,x_2\}=\tau_2\cap\partial\sigma_1$, $p_2=\partial\tau_2\setminus \partial\sigma_1$, $x_1\prec \tau_1$, $p_2\nprec\sigma_1$, and $q_1\sim p_2$, as shown in Figure \ref{triangle-square-oneedge-com3}. Then
\begin{align}\begin{split}\label{tau_2}
|\tau_2|=3.
\end{split}\end{align}


For $|\omega_1'|=3$, with the same argument as in Claim 2 of Case 2, we obtain that there exists a square $\omega_2$ such that
$\omega_2$ is $\sigma_2$-adjacent to $\omega_1$, $\omega_2\cap\omega_1'=\{p_2,y_2\}$, $\omega_2\cap\tau_2=\{p_2,x_2\}$, and $\omega_2\cap\sigma_2=\{y_2,y_2'\}$, as shown in Figure \ref{triangle-square-oneedge-com3}. Hence,
\begin{align}\begin{split}\label{omega_2-omega_1'}
|\omega_2|=4.
\end{split}\end{align}
Choose
\[
B=\partial\tau\cup\partial\omega\cup\left(\bigcup_{i=1}^2(\partial\tau_i\cup\partial
\omega_i)\right).
\]
Using (\ref{tau-tau'-tau_1}), (\ref{omega_1}), (\ref{tau_2}) and (\ref{omega_2-omega_1'}), together with
\[
|\omega'|=|\tau_1'|=4, |\omega_1'|=3,
\]
we obtain
\begin{align}\begin{split}\label{p-p_1-q-q_1}
\Pttn(p)=\Pttn(p_1)=\Pttn(q_1)=\Pttn(p_2)=(3,3,4,4).
\end{split}\end{align}
It is easy to check that $N=8$ and
\[
\sum\limits_{z'\in B\setminus(\partial\sigma_1\cup\partial\sigma_2)}\Phi(z')=\Phi(p)+\Phi(q_1)+\Phi(p_1)+\Phi(p_2)=\frac{2}{3}=\frac{N}{12} \]
by (\ref{p-p_1-q-q_1}).


\begin{figure}[htbp]
 \begin{center}
   \begin{tikzpicture}
    \node at (0,0){\includegraphics[width=0.7\linewidth]{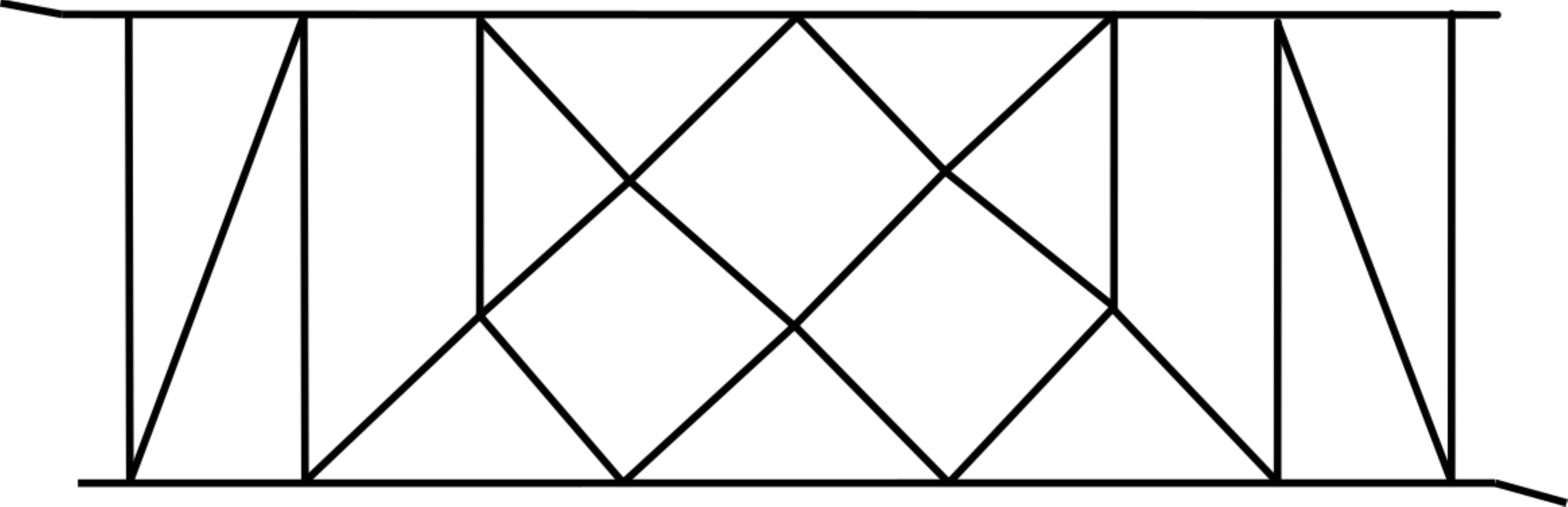}};
    \node at (.2, 1.6){\small $x_1$};
    \node at (1.9, 1.6){\small $x_2$};
    \node at (2.7, 1.7){\small $x_2'$};
    \node at (-.8,1){\small $\tau_1$};
    \node at (1.05,1 ){\small $\tau_2$};

    \node at (.2,.6){\small $\tau'_1$};
    \node at (1.5,.6){\small $\tau'_2$};
    \node at (2.3,.6){\small $\tau_3$};
    \node at (-.4, .4){\small $p_1$};
    \node at (-.9, -1.6){\small $y_1$};
    \node at (.4, -.4){\small $q_1$};
    \node at (.9,  .1){\small $p_2$};
    \node at (1.1, -1.6){\small $y_2$};
    \node at (2.7, -1.6){\small $y_3$};
    \node at (-2.7, -1.6){\small $y$};
    \node at (-1.7,  -.9){\small $\omega$};
    \node at (.15,  -.9){\small $\omega_1$};
    \node at (-.8,  -.4){\small $\omega'$};
    \node at (1.05,  -.5){\small $\omega_1'$};
    \node at (2.1,  -.9){\small $\omega_2$};
    \node at (2.1,  -.2){\small $q_2$};
    \node at (-1.4,  -.4){\small $p$};
    \node at (-2.1,   .4){\small $\tau$};
    \node at (-1.4,   .5){\small $\tau'$};
    \node at (-1.0,  -2.2){\small $\sigma_2$};
    \node at (-1.0,   2.2){\small $\sigma_1$};
   \node at (-1.6,   1.6){\small $x$};
    \node at (-2.6,   1.7){\small $x'$};
   \end{tikzpicture}
  \caption{}\label{triangle-square-oneedge-com4}
 \end{center}
\end{figure}

Let $\tau_2'$ be a face, lower-adjacent to $\tau$, such that $\{p_2,x_2\}=\tau_2\cap\tau_2'$.

\textbf{Case 4.} $|\omega'|=|\tau_1'|=4$, $|\omega_1'|=4$, $|\tau_2'|=3$.

For $|\omega_1'|=4$, with the similar argument as in Claim 1 of Case 1, we obtain that there exists a triangle $\omega_2$ such that $\omega_2$ is $\sigma_2$-adjacent to $\omega_1$, $q_2=\partial\omega_2\setminus\partial\sigma_2$, $\{y_2,y_3\}=\omega_2\cap\sigma_2$, $y_2\prec\omega_1$, $q_2\nprec\sigma_1$, and $p_2\sim q_2$, as shown in Figure \ref{triangle-square-oneedge-com4}.
Then
\begin{align}\begin{split}\label{omega_2}
|\omega_2|=3.
\end{split}\end{align}

For $|\tau_2'|=3$, with the similar argument as in Claim 2 of Case 2, we obtain that there exists a square $\tau_3$ such that $\tau_3$ is $\sigma_1$-adjacent to $\tau_2$, $\{x_2,q_2\}=\tau_2'\cap\tau_3$, $\{q_2,y_3\}=\omega_2\cap\tau_3$ and $\{x_2,x_2'\}=\tau_3\cap\sigma_1$, as shown in Figure \ref{triangle-square-oneedge-com4}.
Hence,
\begin{align}\begin{split}\label{tau_2'-tau_3}
&|\tau_3|=4.
\end{split}\end{align}

Choose
\[
B=\partial\tau\cup\partial\omega\cup\left( \bigcup_{i=1}^2(\partial\tau_i\cup\partial
\omega_i)\right)\cup\partial\tau_3.
\]
Using (\ref{tau-tau'-tau_1}), (\ref{omega_1}), (\ref{tau_2}),  (\ref{omega_2})  and (\ref{tau_2'-tau_3}), together with
\[
|\omega'|=|\tau_1'|=|\omega_1'|=4, |\tau_2'|=3,
\]
we obtain
\begin{align}\begin{split}\label{p-p_1-p_2-q_2}
&\Pttn(p)=\Pttn(p_1)=(3,3,4,4),\Pttn(q_1)=(3,4,4,4),\\
&\Pttn(p_2)=\Pttn(q_2)=(3,3,4,4).
\end{split}\end{align}
It is easy to check that $N=9$ and
\[
\sum\limits_{z'\in B\setminus(\partial\sigma_1\cup\partial\sigma_2)}\Phi(z')=\Phi(p)+\sum_{i=1}^2(\Phi(q_i)+\Phi(p_i))=\frac{3}{4}=\frac{N}{12}
\]
by (\ref{p-p_1-p_2-q_2}).


\begin{figure}[htbp]
 \begin{center}
   \begin{tikzpicture}
    \node at (0,0){\includegraphics[width=0.7\linewidth]{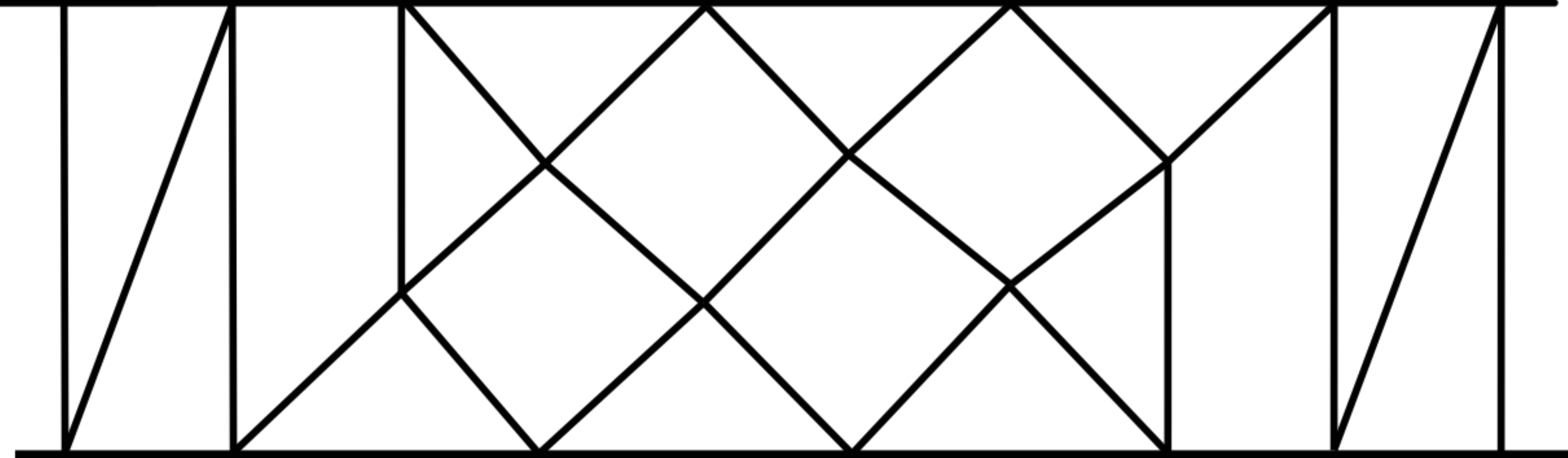}};
    \node at (-2.1,   1.5){\small $x$};
    \node at (-3.1,   1.6){\small $x'$};
    \node at (-.4, 1.5){\small $x_1$};
    \node at (1.2, 1.5){\small $x_2$};
    \node at (3, 1.5){\small $x_3$};
    \node at (-3.1, -1.6){\small $y$};
    \node at (-1.3, -1.6){\small $y_1$};
     \node at (.5, -1.6){\small $y_2$};
    \node at (2.3, -1.6){\small $y_3$};
    \node at (3.1, -1.6){\small $y_3'$};
    \node at (-1.3,1){\small $\tau_1$};
    \node at (.5,1 ){\small $\tau_2$};
    \node at (2.2,1 ){\small $\tau_3$};
    \node at (-2.7,   .4){\small $\tau$};
    \node at (-1.8,   .4){\small $\tau'$};
    \node at (-1.0, .4){\small $p_1$};
    \node at (-.4,.4){\small $\tau'_1$};
    \node at (1.3,.5){\small $\tau'_2$};
    \node at (.8,  .45){\small $p_2$};
    \node at (2.7,-.2){\small $\omega_3$};
    \node at (-1.9,  -.4){\small $p$};
    \node at (-1.3,  -.4){\small $\omega'$};
    \node at (-.1, -.4){\small $q_1$};
    \node at (.55,  -.4){\small $\omega_1'$};
    \node at (-2.2,  -.9){\small $\omega$};
    \node at (-.4,  -1){\small $\omega_1$};
    \node at (1.3,  -1){\small $\omega_2$};
    \node at (1.3,    -.05){\small $q_2$};
    \node at (2.4,.2){\small $p_3$};
     \node at (1.85,  -.4){\small $\omega_2'$};
    \node at (-1.0,  -2.2){\small $\sigma_2$};
    \node at (-1.0,   2.0){\small $\sigma_1$};

   \end{tikzpicture}
  \caption{}\label{triangle-square-oneedge-com6}
 \end{center}
\end{figure}

Let $\omega_2'$ be a face such that $\omega_2'$ is lower-adjacent to $\omega_2$ and $\{q_2,y_3\}=\omega_2\cap\omega_2'$.

\textbf{Case 5.} $|\omega'|=|\tau_1'|=|\omega_1'|=4$, $|\tau_2'|=4$, $|\omega_2'|=3$.

For $|\tau_2'|=4$, with the similar argument as in Claim 1 of Case 2, we obtain that there exists a triangle $\tau_3$ such that $\tau_3$ is $\sigma_1$-adjacent to $\tau_2$, $p_3=\partial\tau_3\setminus\partial\sigma_1$, $\{x_2,x_3\}=\tau_3\cap\sigma_1$, $x_2\prec\tau_2$, $p_3\nprec\sigma_2$ and  $p_3\sim q_2$, as shown in Figure \ref{triangle-square-oneedge-com6}. We have
\begin{align}\begin{split}\label{tau_3}
|\tau_3|=3.
\end{split}\end{align}

For $|\omega_2'|=3$, with the similar argument as in Claim 2 of Case 1, we obtain that there exists a square $\omega_3$ such that $\omega_3$ is $\sigma_2$-adjacent to $\omega_2$, $\{p_3,x_3\}=\omega_3\cap\tau_3$, $\{p_3,y_3\}=\omega_3\cap\omega_2'$ and $\{y_2,y_3'\}=\omega_3\cap\sigma_2$.
Hence,
\begin{align}\begin{split}\label{omega_1'-tau_3-34}
|\omega_3|=4.
\end{split}\end{align}

Choose
\[
B=\partial\tau\cup\partial\omega \bigcup_{i=1}^3(\partial\tau_i\cup\partial
\omega_i).
\]
Using (\ref{tau-tau'-tau_1}), (\ref{omega_1}), (\ref{tau_2}), (\ref{omega_2}), (\ref{tau_3}) and (\ref{omega_1'-tau_3-34}), together with
\[
|\omega'|=|\tau_1'|=|\omega_1'|=|\tau_2'|=4, |\omega_2'|=3,
\]
we obtain
\begin{align}\begin{split}\label{p-p_1-p_2-q_2-q_1}
&\Pttn(p)=\Pttn(p_1)=\Pttn(p_3)=\Pttn(q_2)=(3,3,4,4),\\
&\Pttn(q_1)=\Pttn(p_2)=(3,4,4,4).
\end{split}\end{align}
It is easy to check that $N=10$ and
\[
\sum\limits_{z'\in B\setminus(\partial\sigma_1\cup\partial\sigma_2)}\Phi(z')=\Phi(p)+\sum_{i=1}^2(\Phi(q_i)+\Phi(p_i))+\Phi(p_3)=\frac{N}{12} \]
by (\ref{p-p_1-p_2-q_2-q_1}).


\begin{figure}[htbp]
 \begin{center}
   \begin{tikzpicture}
    \node at (0,0){\includegraphics[width=0.8\linewidth]{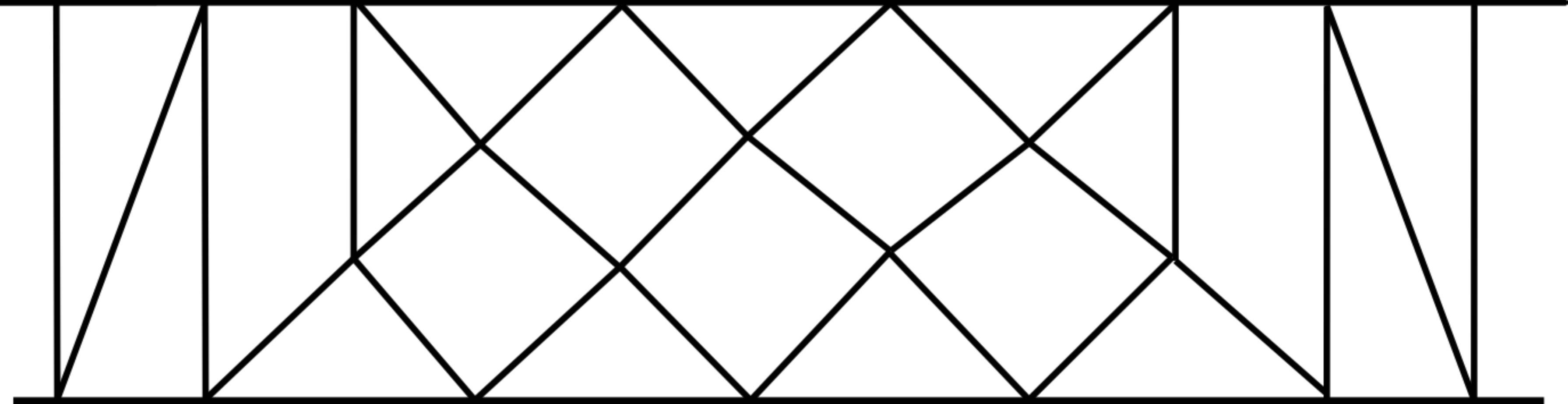}};
\node at (-2.7,   1.5){\small $x$};
    \node at (-3.7,   1.6){\small $x'$};
    \node at (-1.1, 1.5){\small $x_1$};
    \node at (.6, 1.5){\small $x_2$};
    \node at (2.5, 1.5){\small $x_3$};
    \node at (3.5, 1.6){\small $x_3'$};
    \node at (-3.7, -1.5){\small $y$};
    \node at (-2.0, -1.5){\small $y_1$};
     \node at (-.1, -1.5){\small $y_2$};
    \node at (1.7, -1.5){\small $y_3$};
    \node at (3.5, -1.5){\small $y_4$};
    \node at (-1.9,1){\small $\tau_1$};
    \node at (-.2,1 ){\small $\tau_2$};
    \node at (1.6,1 ){\small $\tau_3$};
    \node at (-3.2,   .4){\small $\tau$};
    \node at (-2.4,   .4){\small $\tau'$};
    \node at (-1.5, .4){\small $p_1$};
    \node at (-1,.4){\small $\tau'_1$};
    \node at (.8,.5){\small $\tau'_2$};
    \node at (.2,  .45){\small $p_2$};
    \node at (-2.4,  -.4){\small $p$};
    \node at (-1.8,  -.4){\small $\omega'$};
    \node at (-.7, -.4){\small $q_1$};
    \node at (-.1,  -.4){\small $\omega_1'$};
        \node at (-2.8,  -.9){\small $\omega$};
    \node at (-1.1,  -1){\small $\omega_1$};
    \node at (.7,  -1){\small $\omega_2$};
    \node at (2.5,  -1){\small $\omega_3$};
     \node at (2.75,-.3){\small $q_3$};
     \node at (2.1, .4){\small $\tau_3'$};
    \node at (3.1, .3){\small $\tau_4$};
    \node at (1,    -.4){\small $q_2$};
    \node at (1.65,.1){\small $p_3$};
    \node at (1.75,  -.4){\small $\omega_2'$};
    \node at (-1.0,  -2.0){\small $\sigma_2$};
    \node at (-1.0,   2.0){\small $\sigma_1$};
   \end{tikzpicture}
  \caption{}\label{triangle-square-oneedge-com7}
 \end{center}
\end{figure}

Let $\tau_3'$ be a face such that $\tau_3'$ is lower-adjacent to $\tau_3$ and $\tau_3\cap\tau_3'=\{p_3,x_3\}$.

\textbf{Case 6.}  $|\omega'|=|\tau_1'|=|\omega_1'|=|\tau_2'|=4$, $|\omega_2'|=4$ and $|\tau_3'|=3$.

As in Claim 1 of Case 2, for $|\omega_2'|=4$, we obtain that there exists a triangle $\omega_3$ such that $\omega_3$ is $\sigma_2$-adjacent to $\omega_2$, $q_3=\partial\omega_3\setminus\partial\sigma_2$, $\{y_3,y_4\}=\omega_3\cap\sigma_2$, $y_3\prec\omega_3$, $q_3\nprec\sigma_1$ and  $p_3\sim q_3$, as shown in Figure \ref{triangle-square-oneedge-com7}. Then
\begin{align}\begin{split}\label{omega_3}
|\omega_3|=3.
\end{split}\end{align}

For $|\tau_3'|=3$, $q_3\sim x_3$,  with the similar argument as in Claim 2 of Case 2, we obtain that there exists a square $\tau_4$ such that $\tau_4$ is $\sigma_1$-adjacent to $\tau_3$, $\{q_3,x_3\}=\tau_4\cap\tau_3'$, $\{q_3,y_4\}=\tau_4\cap\omega_3$ and $\{x_3,x_3'\}=\tau_4\cap\sigma_1$.
Hence,
\begin{align}\begin{split}\label{tau_3'-tau_4}
|\tau_4|=4.
\end{split}\end{align}

Choose
\[
B=\partial\tau\cup\partial\omega\cup\left( \bigcup_{i=1}^3(\partial\tau_i\cup\partial
\omega_i)\right)\cup\partial\tau_4.
\]
Using (\ref{tau-tau'-tau_1}), (\ref{omega_1}),  (\ref{tau_2}), (\ref{omega_2}), (\ref{tau_3}), (\ref{omega_3}) and (\ref{tau_3'-tau_4}), together with
\[
|\omega'|=|\tau_1'|=|\omega_1'|=|\tau_2'|=|\omega_2'|=4, |\tau_3'|=3,
\]
we obtain
\begin{align}\begin{split}\label{p-p_1-p_3-q_3-q_2}
&\Pttn(p)=\Pttn(p_1)=\Pttn(p_3)=\Pttn(q_3)=(3,3,4,4),\\
&\Pttn(p_2)=\Pttn(q_1)=\Pttn(q_2)=(3,4,4,4).
\end{split}\end{align}
It is easy to check that $N=11$ and
\[
\sum\limits_{z'\in B\setminus(\partial\sigma_1\cup\partial\sigma_2)}\Phi(z')=\Phi(p)+\sum_{i=1}^3(\Phi(q_i)+\Phi(p_i))=\frac{11}{12}=\frac{N}{12} \]
by (\ref{p-p_1-p_3-q_3-q_2}).

\begin{figure}[htbp]
 \begin{center}
   \begin{tikzpicture}
    \node at (0,0){\includegraphics[width=0.85\linewidth]{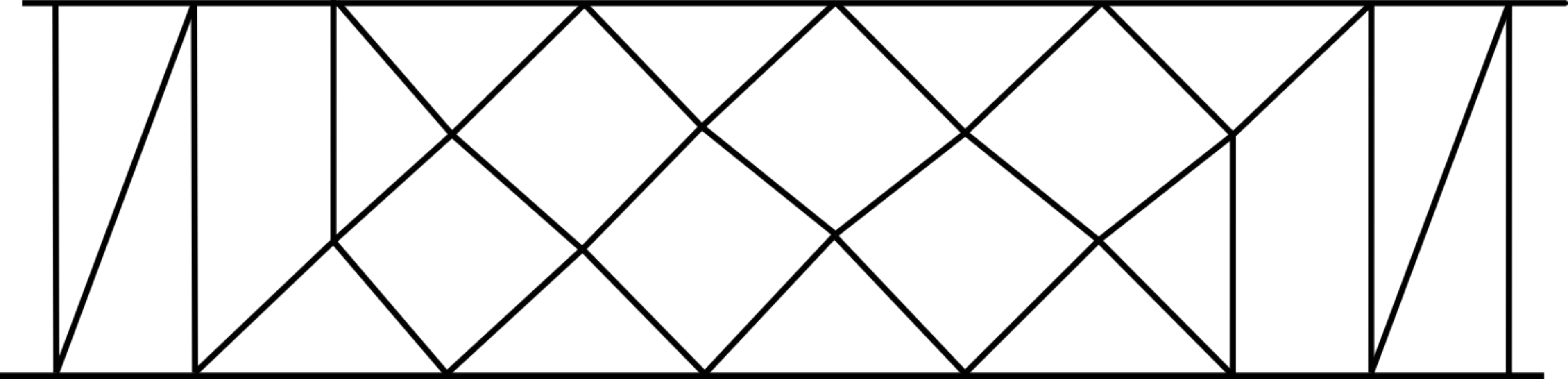}};
    \node at (-3,   1.5){\small $x$};
    \node at (-4,   1.6){\small $x'$};
    \node at (-1.3, 1.5){\small $x_1$};
    \node at (.5, 1.5){\small $x_2$};
    \node at (2.2, 1.5){\small $x_3$};
    \node at (4, 1.5){\small $x_4$};
    \node at (-4, -1.5){\small $y$};
    \node at (-2.3, -1.5){\small $y_1$};
     \node at (-.4, -1.5){\small $y_2$};
    \node at (1.3, -1.5){\small $y_3$};
    \node at (3.2, -1.5){\small $y_4$};
    \node at (4, -1.6){\small $y_4'$};
    \node at (-2.3, 1){\small $\tau_1$};
    \node at (-.5,1 ){\small $\tau_2$};
    \node at (1.3,1 ){\small $\tau_3$};
    \node at (3.1,1 ){\small $\tau_4$};
    \node at (-3.6,   .4){\small $\tau$};
    \node at (-2.8,   .4){\small $\tau'$};
    \node at (-1.9, .4){\small $p_1$};
    \node at (-1.4,.4){\small $\tau'_1$};
    \node at (.4,.5){\small $\tau'_2$};
    \node at (-.2,  .45){\small $p_2$};
    \node at (1.6, .45){\small $p_3$};
    \node at (-3.2,  -.9){\small $\omega$};
    \node at (-1.4,  -1){\small $\omega_1$};
    \node at (.4,  -1){\small $\omega_2$};
    \node at (2.2,  -1){\small $\omega_3$};
    \node at (-2.8,  -.4){\small $p$};
    \node at (-2.2,  -.4){\small $\omega'$};
    \node at (-1.1, -.4){\small $q_1$};
    \node at (-.5,  -.4){\small $\omega_1'$};
     \node at (.7,    -.4){\small $q_2$};
     \node at (1.45,  -.4){\small $\omega_2'$};
    \node at (2.1, .5){\small $\tau_3'$};
    \node at (2.2,-.1){\small $q_3$};
    \node at (2.65,  -.4){\small $\omega_3'$};
    \node at (3.6,  -.4){\small $\omega_4$};
    \node at (3.3, .2){\small $p_4$};

    \node at (-1.0,  -2.0){\small $\sigma_2$};
    \node at (-1.0,   2.0){\small $\sigma_1$};
   \end{tikzpicture}
  \caption{}\label{triangle-square-oneedge-com8}
 \end{center}
\end{figure}

Let $\omega_3'$ be a face such that $\omega_3'$ is lower-adjacent to $\omega_3$ and $\{q_3,y_4\}=\omega_3'\cap\omega_3$.

\textbf{Case 7.}  $|\omega'|=|\tau_1'|=|\omega_1'|=|\tau_2'|=|\omega_2'|=4$, $|\tau_3'|=4$, $|\omega_3'|=3$.

For $|\tau_3'|=4$, with the same argument as in Claim 1 of Case 2,  we obtain that there exists a triangle $\tau_4$ such that $\tau_4$ is $\sigma_1$-adjacent to $\tau_3$, $p_4=\partial\tau_4\setminus\partial\sigma_1$, $\{x_3,x_4\}=\tau_4\cap\sigma_1$, $x_3\prec\tau_4$, $p_4\nprec\sigma_2$ and $p_4\sim q_3$, as shown in Figure \ref{triangle-square-oneedge-com8}. Then
\begin{align}\begin{split}\label{tau_4}
|\tau_4|=3.
\end{split}\end{align}

For $|\omega_3'|=3$, with the similar argument as in Claim 2 of Case 2, we obtain that there exists a square $\omega_4$ such that $\omega_4$ is $\sigma_2$-adjacent to $\omega_3$, $\{p_4,y_4\}=\omega_4\cap\omega_3'$, $\{p_4,x_4\}=\tau_4\cap\omega_4$ and $\{y_4,y_4'\}=\omega_4\cap\sigma_2$, as shown in Figure \ref{triangle-square-oneedge-com8}.
Hence,
\begin{align}\begin{split}\label{omega_3'-omega_4}
|\omega_4|=4.
\end{split}\end{align}

Choose
\[
B=\partial\tau\cup\partial\omega \cup\left(\bigcup_{i=1}^4(\partial\tau_i\cup\partial
\omega_i)\right).
\]
Using (\ref{tau-tau'-tau_1}), (\ref{omega_1}),  (\ref{tau_2}), (\ref{omega_2}), (\ref{tau_3}), (\ref{omega_3}), (\ref{tau_4}) and (\ref{omega_3'-omega_4}), together with
\[
|\omega'|=|\tau_1'|=|\omega_1'|=|\tau_2'|=|\omega_2'|=|\tau_3'|=4, |\omega_3'|=3,
\]
we obtain
\begin{align}\begin{split}\label{p-q-p_4-q_3}
&\Pttn(p)=\Pttn(p_1)=\Pttn(q_3)=\Pttn(p_4)=(3,3,4,4),\\
&\Pttn(p_2)=\Pttn(p_3)=\Pttn(q_1)=\Pttn(q_2)=(3,4,4,4).
\end{split}\end{align}
It is easy to check that $N=12$ and
\[
\sum\limits_{z'\in B\setminus(\partial\sigma_1\cup\partial\sigma_2)}\Phi(z')=\Phi(p)+\sum_{i=1}^3(\Phi(q_i)+\Phi(p_i))+\Phi(p_4)=1=\frac{N}{12} \]
by (\ref{p-q-p_4-q_3}).

\begin{figure}[htbp]
 \begin{center}
   \begin{tikzpicture}
    \node at (0,0){\includegraphics[width=0.92\linewidth]{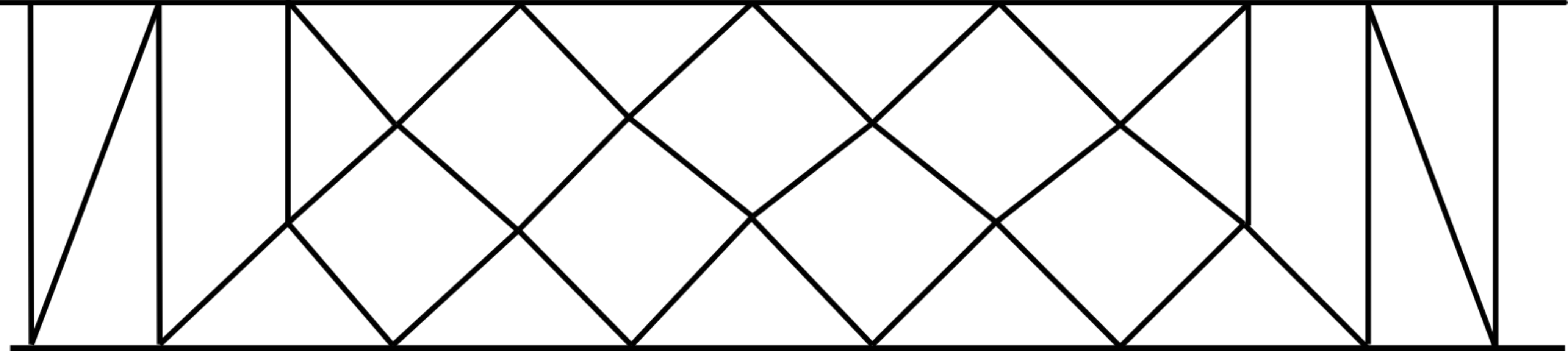}};
    \node at (-3.6,   1.5){\small $x$};
    \node at (-4.6,   1.5){\small $x'$};
    \node at (-2, 1.5){\small $x_1$};
    \node at (-.3, 1.5){\small $x_2$};
    \node at (1.7, 1.5){\small $x_3$};
    \node at (3.5, 1.5){\small $x_4$};
    \node at (4.4, 1.5){\small $x_4'$};

    \node at (-4.6, -1.5){\small $y$};
    \node at (-2.8, -1.5){\small $y_1$};
     \node at (-1.1, -1.5){\small $y_2$};
    \node at (.7, -1.5){\small $y_3$};
    \node at (2.5, -1.5){\small $y_4$};
    \node at (4.4, -1.5){\small $y_5$};
    \node at (-2.8, 1){\small $\tau_1$};
    \node at (-1,1 ){\small $\tau_2$};
    \node at ( .8,1 ){\small $\tau_3$};
    \node at (2.6,1 ){\small $\tau_4$};
    \node at (-4.1,   .4){\small $\tau$};
    \node at (-3.3,   .4){\small $\tau'$};
    \node at (-2.5, .4){\small $p_1$};
    \node at (-2,.4){\small $\tau'_1$};
    \node at (-.2,.5){\small $\tau'_2$};
    \node at (1.5, .5){\small $\tau_3'$};
    \node at (3.1, .4){\small $\tau_4'$};
    \node at (-.8,  .45){\small $p_2$};
    \node at (1.0, .45){\small $p_3$};
    \node at (2.5,  .1){\small $p_4$};
    \node at (-3.8,  -.9){\small $\omega$};
    \node at (-2.1,  -1){\small $\omega_1$};
    \node at (-.3,  -1){\small $\omega_2$};
    \node at (1.5,  -1){\small $\omega_3$};
    \node at (3.4,  -1){\small $\omega_4$};
    \node at (-3.4,  -.4){\small $p$};
    \node at (-2.8,  -.4){\small $\omega'$};
    \node at (-1.7, -.4){\small $q_1$};
    \node at (-1.1,  -.4){\small $\omega_1'$};
    \node at (.1,    -.4){\small $q_2$};
    \node at (.8,  -.4){\small $\omega_2'$};
    \node at (1.9,-.4){\small $q_3$};
    \node at (2.5,  -.4){\small $\omega_3'$};
    \node at (2.5,  -.4){\small $\omega_3'$};
    \node at (3.7,  -.4){\small $q_4$};
    \node at (3.9, .5){\small $\tau_5$};

    \node at (-1.0,  -2.0){\small $\sigma_2$};
    \node at (-1.0,   2.0){\small $\sigma_1$};
   \end{tikzpicture}
  \caption{}\label{triangle-square-oneedge-com9}
 \end{center}
\end{figure}

Let $\tau_4'$ be a face such that $\tau_4'$ is lower-adjacent to $\tau_4$ and $\{p_4,x_4\}=\tau_4'\cap\tau_4$.

\textbf{Case 8.}  $|\omega'|=|\tau_1'|=|\omega_1'|=|\tau_2'|=|\omega_2'|=|\tau_3'|=4$, $|\omega_3'|=4$.

For $|\omega_3'|=4$, with the same argument as in Claim 1 of Case 2,  we obtain that there exists a triangle $\omega_4$ such that $\omega_4$ is $\sigma_2$-adjacent to $\omega_3$, $q_4=\partial\omega_4\setminus\partial\sigma_2$, $\{y_4,y_5\}=\omega_4\cap\sigma_2$, $y_4\prec\omega_4$, $q_4\nprec\sigma_1$ and $p_4\sim q_4$, as shown in Figure \ref{triangle-square-oneedge-com9}. Then
\begin{align}\begin{split}\label{omega_4}
|\omega_4|=3.
\end{split}\end{align}
Using (\ref{tau-tau'-tau_1}), (\ref{omega_1}),  (\ref{tau_2}), (\ref{omega_2}), (\ref{tau_3}), (\ref{omega_3}), (\ref{tau_4}) and (\ref{omega_4}), we obtain
\begin{align*}
&\Pttn(p)=\Pttn(p_1)=(3,3,4,4),\\
&\Pttn(p_2)=\Pttn(p_3)=\Pttn(q_1)=\Pttn(q_2)=\Pttn(q_3)=(3,4,4,4).
\end{align*}
This implies that
\begin{align}\begin{split}\label{part-sum-curv}
\Phi(p)+\sum_{i=1}^3\left(\Phi(p_i)+\Phi(q_i)\right)=\frac{3}{4}.
\end{split}\end{align}

By $x_4\prec\sigma_1$ and (\ref{pattern}), we obtain $\Pttn(x_4)=(3,3,3,k_1)$ or $(3,3,4,k_1)$. Since $x_4\prec \tau_4'$, we obtain $|\tau_4'|=3$ or $4$. We divide it into subcases.

\textbf{Case 8.1.} $|\tau_4'|=3$. With the same argument as in Case 7, there exists a square $\tau_5$ lower-adjacent to $\tau_4'$ and $\omega_4$, as shown in Figure \ref{triangle-square-oneedge-com9}. Using (\ref{tau_4}), (\ref{omega_4}) and $|\tau_5|=4$, we obtain
\begin{align*}
\Pttn(p_4)=\Pttn(q_4)=(3,3,4,4), \Phi(p_4)=\Phi(q_4)=\frac{1}{6}.
\end{align*}
Combining this with (\ref{part-sum-curv}), we obtain
\[
\Phi(p)+\sum_{i=1}^4\left(\Phi(p_i)+\Phi(q_i)\right)=\frac{13}{12}>1,
\]
this contradicts Theorem \ref{thm:DeVosMohar}.

\begin{figure}[htbp]
 \begin{center}
   \begin{tikzpicture}
    \node at (-0.4,0){\includegraphics[width=0.95\linewidth]{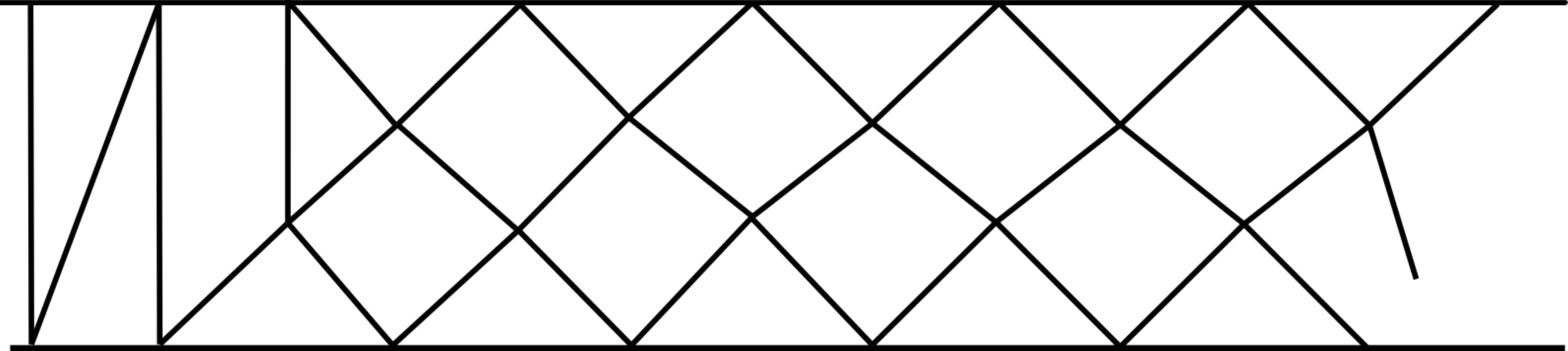}};
    \node at (-4.1,   1.5){\small $x$};
    \node at (-5.1,   1.5){\small $x'$};
    \node at (-2.4, 1.5){\small $x_1$};
    \node at (-.7, 1.5){\small $x_2$};
    \node at (1.3, 1.5){\small $x_3$};
    \node at (3.2, 1.5){\small $x_4$};
    \node at (5, 1.5){\small $x_5$};
    \node at (-5.2, -1.5){\small $y$};
    \node at (-3.4, -1.5){\small $y_1$};
     \node at (-1.7, -1.5){\small $y_2$};
    \node at (.2, -1.5){\small $y_3$};
    \node at (2.0, -1.5){\small $y_4$};
    \node at (3.9, -1.5){\small $y_5$};
    \node at (-3.4, 1){\small $\tau_1$};
    \node at (-1.6,1 ){\small $\tau_2$};
    \node at ( .3,1 ){\small $\tau_3$};
    \node at (2.2,1 ){\small $\tau_4$};
    \node at (4.1, 1){\small $\tau_5$};

    \node at (-4.6,   .4){\small $\tau$};
    \node at (-3.9,   .4){\small $\tau'$};
    \node at (-3.1, .4){\small $p_1$};
    \node at (-2.5,.4){\small $\tau'_1$};
    \node at (-.7,.5){\small $\tau'_2$};
    \node at (1.2, .5){\small $\tau_3'$};
    \node at (3.2, .5){\small $\tau_4'$};
    \node at (-1.2,  .45){\small $p_2$};
    \node at (.6, .45){\small $p_3$};
    \node at (2.6,  .4 ){\small $p_4$};
    \node at (4.4,  .3){\small $p_5$};
    \node at (3.1,  -0.1){\small $q_4$};

    \node at (-4.3,  -.9){\small $\omega$};
    \node at (-2.5,  -1){\small $\omega_1$};
    \node at (-.7,  -1){\small $\omega_2$};
    \node at (1.2,  -1){\small $\omega_3$};
    \node at (3,  -1){\small $\omega_4$};
    \node at (-3.9,  -.4){\small $p$};
    \node at (-3.3,  -.4){\small $\omega'$};
    \node at (-2.1, -.4){\small $q_1$};
    \node at (-1.6,  -.4){\small $\omega_1'$};
    \node at (-.3,    -.4){\small $q_2$};
    \node at (.3,  -.4){\small $\omega_2'$};
    \node at (1.5,-.4){\small $q_3$};
    \node at (2.2,  -.4){\small $\omega_3'$};
    \node at (3.7,  -.4){\small $\omega_4'$};
    \node at (5,  +.6){\small $\tau_5'$};

    \node at (-1.0,  -2.0){\small $\sigma_2$};
    \node at (-1.0,   2.0){\small $\sigma_1$};
   \end{tikzpicture}
  \caption{}\label{triangle-square-oneedge-com10}
 \end{center}
\end{figure}
\textbf{Case 8.2.} $|\tau_4'|=4$.  Then $\Pttn(p_4)=(3,4,4,4)$ and $\Phi(p_4)=\frac{1}{12}$. With the same argument as in Claim 1 of Case 2,  we obtain that there exists a triangle $\tau_5$ such that $\tau_5$ is $\sigma_1$-adjacent to $\tau_4$, $p_5=\partial\tau_5\setminus\partial\sigma_1$, $\{x_4,x_5\}=\tau_5\cap\sigma_1$, $x_4\prec\tau_4$, $p_5\nprec\sigma_2$ and $p_5\sim q_4$, as shown in Figure \ref{triangle-square-oneedge-com10}. Let $\omega_4'$ be a face such that $\omega_4'$ is lower-adjacent to $\omega_4$ and $\{q_4,y_5\}=\omega_4'\cap\omega_4$. Let $\tau_5'$ be a face such that $\tau_5'$ is lower-adjacent to $\tau_5$ and $\{p_5,x_5\}=\tau_5'\cap\tau_5$. By (\ref{pattern}) and $x_5\prec \omega_4'$ and $y_5\prec\tau_5'$, we have $|\omega_4'|$ and $|\tau_5'|$ are $3$ or $4$. This implies that $\Pttn(q_4)$ and $\Pttn(p_5)$ are $(3,3,4,4)$ or $(3,4,4,4)$, that is, $\Phi(q_4)\geq \frac{1}{12}$ and $\Phi(p_5)\geq \frac{1}{12}$. We divide it into subcases.

\textbf{Subcase 8.2.1.} $\Phi(q_4)>\frac{1}{12}$ or $\Phi(p_5)>\frac{1}{12}$. That implies that $\Phi(q_4)+\Phi(p_5)>\frac{1}{6}$. By (\ref{part-sum-curv}), we obtain
\[
\Phi(p)+\sum_{i=1}^3\left(\Phi(p_i)+\Phi(q_i)\right)+\Phi(p_4)+\left(\Phi(q_4)+\Phi(p_5)\right)>1.
\]
This contradicts Theorem \ref{thm:DeVosMohar}.

\begin{figure}[htbp]
 \begin{center}
   \begin{tikzpicture}
    \node at (0,0){\includegraphics[width=0.98\linewidth]{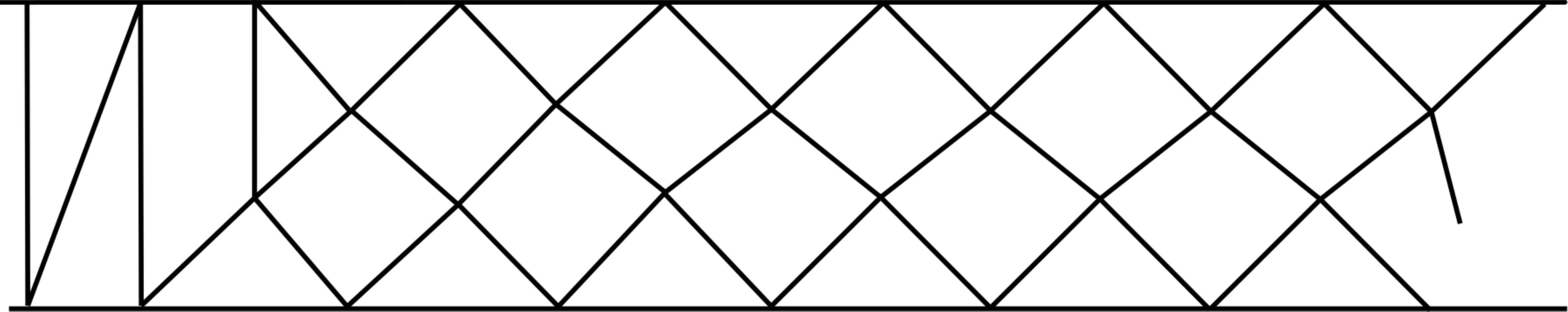}};
    \node at (-4.1,   1.4){\small $x$};
    \node at (-5.1,   1.4){\small $x'$};
    \node at (-2.6, 1.4){\small $x_1$};
    \node at (-.8, 1.4){\small $x_2$};
    \node at (.9, 1.4){\small $x_3$};
    \node at (2.6, 1.4){\small $x_4$};
    \node at (4.3, 1.4){\small $x_5$};
    \node at (5.9, 1.4){\small $x_6$};
    \node at (-5.1, -1.4){\small $y$};
    \node at (-3.5, -1.4){\small $y_1$};
     \node at (-1.8, -1.4){\small $y_2$};
    \node at (-.0, -1.4){\small $y_3$};
    \node at (1.7, -1.4){\small $y_4$};
    \node at (3.4, -1.4){\small $y_5$};
    \node at (5.1, -1.4){\small $y_6$};
    \node at (-3.5, .9){\small $\tau_1$};
    \node at (-1.8,.9 ){\small $\tau_2$};
    \node at ( -.1,.9){\small $\tau_3$};
    \node at (1.7,.9 ){\small $\tau_4$};
    \node at (3.5,.9 ){\small $\tau_5$};
    \node at (-4.7,   .4){\small $\tau$};
    \node at (-4,   .4){\small $\tau'$};
    \node at (-3.1, .4){\small $p_1$};
    \node at (-2.6,.4){\small $\tau'_1$};
    \node at (-1.5,  .45){\small $p_2$};
    \node at (4.2, .45){\small $\tau_5'$};
    \node at (-4.0,  -.35){\small $p$};
    \node at (-1,.4){\small $\tau'_2$};
    \node at (.7, .4){\small $\tau_3'$};
    \node at (2.5, .4){\small $\tau_4'$};
    \node at (-3.4,  -.4){\small $\omega'$};
    \node at (-2.3, -.4){\small $q_1$};
    \node at (-1.8,  -.4){\small $\omega_1'$};
    \node at (-.7,    -.3){\small $q_2$};
    \node at (.0,  -.4){\small $\omega_2'$};
     \node at (1.1,-.4){\small $q_3$};
     \node at (1.7,  -.4){\small $\omega_3'$};
     \node at (2.8,  -.4){\small $q_4$};
     \node at (3.4,  -.4){\small $\omega_4'$};
     \node at (3.7,  .4){\small $p_5$};
     \node at (5.4,  .3){\small $p_6$};
    \node at (.2, .4){\small $p_3$};
    \node at (2.0,  .38){\small $p_4$};
    \node at (-4.3,  -.9){\small $\omega$};
    \node at (-2.6,  -.9){\small $\omega_1$};
    \node at (-.9,  -.9){\small $\omega_2$};
    \node at (0.8,  -.9){\small $\omega_3$};
    \node at (2.4,  -.9){\small $\omega_4$};
    \node at (4.55,  -.4){\small $q_5$};
    \node at (4.3,  -.9){\small $\omega_5$};
    \node at (5,  -.7){\small $\omega_5'$};
    \node at (-1.0,  -2.0){\small $\sigma_2$};
    \node at (-1.0,   2.0){\small $\sigma_1$};
    \node at (5.1, .9){\small $\tau_6$};
   \end{tikzpicture}
  \caption{}\label{triangle-square-oneedge-com11}
 \end{center}
\end{figure}

\textbf{Subcase 8.2.2.} $\Phi(q_4)=\Phi(p_5)=\frac{1}{12}$. That is, $\Pttn(q_4)=\Pttn(p_5)=(3,4,4,4)$. This implies that $|\omega_4'|=|\tau_5'|=4$.  With the same argument as in Claim 1 of Case 2, we obtain there exist two triangles $\omega_5$ and $\tau_6$ such that $\omega_5$ is $\sigma_2$-adjacent to $\omega_4$, $q_5=\partial\omega_5\setminus\partial\sigma_2$, $\{y_5,y_6\}=\omega_5\cap\sigma_2$, $y_5\prec\omega_4$, $q_5\nprec\sigma_1$, $q_5\sim p_5$ and $\tau_6$ is $\sigma_1$-adjacent to $\tau_5$, $p_6=\partial\tau_6\setminus\partial\sigma_1$, $\{x_5,x_6\}=\tau_6\cap\sigma_1$, $x_5\prec\tau_5$, $p_6\nprec\sigma_2$, $p_6\sim q_5$, as shown in Figure \ref{triangle-square-oneedge-com11}. $|p_5|=4$ implies that $\{p_5,x_5\}$ and $\{p_5,q_5\}$ are contained in a face, so that $\{p_5,q_5\}\prec\tau_5'$. Let $\omega_5'$ be a face lower-adjacent to $\omega_5$ such that $\omega_5'\cap \omega_5=\{q_5,y_6\}$. Then $\Pttn(q_5)=(3,4,|\tau_5'|,|\omega_5'|)$. By (\ref{pattern}), $|\omega_5'|=3$ or $4$, together with $|\tau_5'|=4$, we obtain $\Phi(q_5)\geq\frac{1}{12}$.
Hence,
\[
\Phi(p)+\sum_{i=1}^3\left(\Phi(p_i)+\Phi(q_i)\right)+\Phi(p_4)+\left(\Phi(q_4)+\Phi(p_5)+\Phi(q_5)\right)\geq\frac{13}{12}>1.
\]
This contradicts Theorem \ref{thm:DeVosMohar}.

\begin{figure}[htbp]
 \begin{center}
   \begin{tikzpicture}
    \node at (0,0){\includegraphics[width=0.5\linewidth]{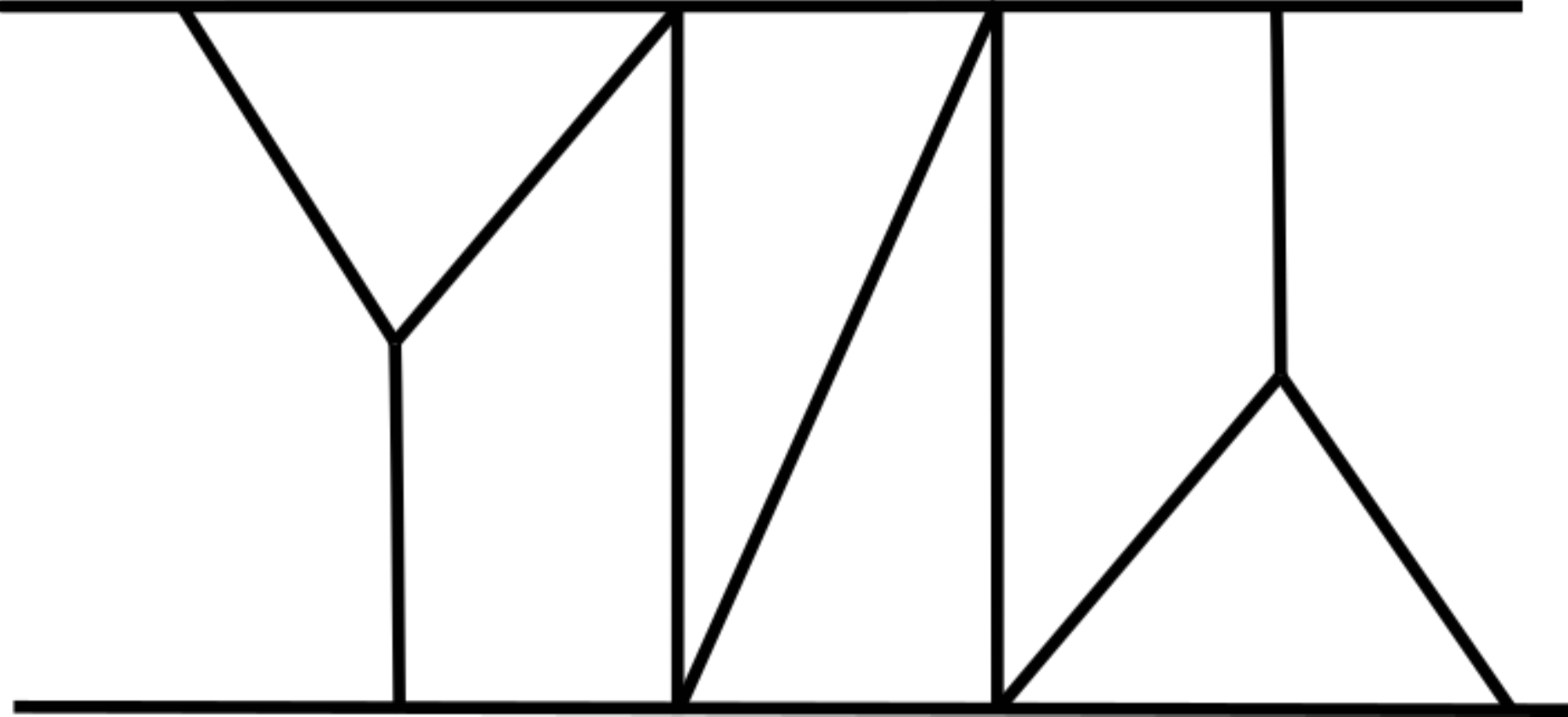}};
    \node at (.9,   1.7){\small $x'$};
    \node at (2.0, 1.6){\small $x$};
    \node at (2.0,  -1){\small $\omega$};
    \node at (1.4, .45){\small $\tau$};
    \node at (1.0, -1.7){\small $y$};
    \node at (3.0, -1.7){\small $y_1$};
    \node at (2.3,-.0){\small $p$};
    \node at (1.5,  -.5){\small $e$};
    \node at (-.0,  .45){\small $\eta_1$};
    \node at (.5,    -.4){\small $\eta_1'$};
    \node at (-1.1, -.4){\small $\gamma$};
    \node at (-1.5, .8){\small $\eta_2$};
    \node at (-1.8, -.1){\small $p'$};
    \node at (-1.6, -1.7){\small $z'$};
    \node at (-.4, -1.7){\small $z$};
    \node at (-.4,  1.7){\small $v_1$};
    \node at (-1.0,  -2.0){\small $\sigma_2$};
    \node at (-1.0,   2.0){\small $\sigma_1$};
   \end{tikzpicture}
  \caption{}\label{left-1}
 \end{center}
\end{figure}

\textbf{Step 2.} For the other side of $\tau$, we will prove that the modified curvature of these vertices bounded below by $\frac{1}{12}$.

Since $\tau$ is lower-adjacent to $\omega$, there exists an edge, denoted by $e$, satisfying $e=\tau\cap\omega$. Since $p=\partial\omega\setminus\partial\sigma_2$, $p\nprec\sigma_1$, and
$\{y,y_1\}=\omega\cap\sigma_2$,  $e=\{p,y\}$ or $\{p,y_1\}$. So $e=\{p,y\}$ if $x'\sim y$, as shown in Figure \ref{left-1}.

 By (\ref{pattern}) and $|\tau|=4$, $\Pttn(x')=(3,3,4,k_1)$. Then there exists a triangle $\eta_1$ such that $\eta_1$ is $\sigma_1$-adjacent to $\tau$, $\{v_1,x'\}=\eta_1\cap\sigma_1$ and $z=\partial\eta_1\setminus\partial\sigma_1$. $|x'|=4$ implies that $\{x',z\}$ and $\{x',y\}$ are contained in a face, denoted by $\eta_1'$. Since $\Pttn(x')=(3,3,4,k_1)$, $|\tau|=4$, and $|\eta_1'|=3$, $z\sim y$. Clearly, $z\prec\sigma_2$. Otherwise, $\{z,y\}\nprec\sigma_2$, which yields $|y|=5$. This contradicts $|y|=4$. Since $|z|=4$, there exists $z'\neq y$ such that $\{z',z\}\prec \sigma_2$ and both $\{z,z'\}$ and $\{z,v_1\}$ are contained in a face, denoted by $\gamma$. We divided it into cases.

\textbf{Case 1.} $|\gamma|=4$, as shown in Figure \ref{left-1}. Then $\Pttn(v_1)=(3,3,4,k_1)$ by (\ref{pattern}). Since $\{v_1,z\}=\gamma\cap\eta_1$, there exists a triangle, denoted by $\eta_2$, such that $\eta_2$ is $\sigma_1$-adjacent to $\eta_1$ and $\eta_2$ is lower-adjacent to $\gamma$. Let $p'=\partial\eta_2\setminus\partial\sigma_1$. Then $\{p',v_1\}\prec\gamma$.
Since $\{p',v_1\}\nprec\eta_1$, $\{p',v_1\}$, $\{v_1,z\}$ and $\{z,z'\}$ are contained in a face $\gamma$. Using $|\gamma|=4$, we obtain $p'\sim z'$ and $p'\nprec \sigma_2$, as shown in Figure \ref{left-1}. For $\gamma$ and $\eta_2$, by the same arguments about $\tau$ and $\omega$ in Case 1 to Case 7,  we obtain that there exists some $B$, same as the choice of $B$ in Case 1 to Case 7, such that $v_1,z\prec B\cap(\partial\sigma_1\cup\partial\sigma_2)$, and the modified curvature satisfies
$\Phi'(z)\geq\frac{1}{12}$ and $\Phi'(v_1)\geq\frac{1}{12}$.

\begin{figure}[htbp]
 \begin{center}
   \begin{tikzpicture}
    \node at (0,0){\includegraphics[width=0.46\linewidth]{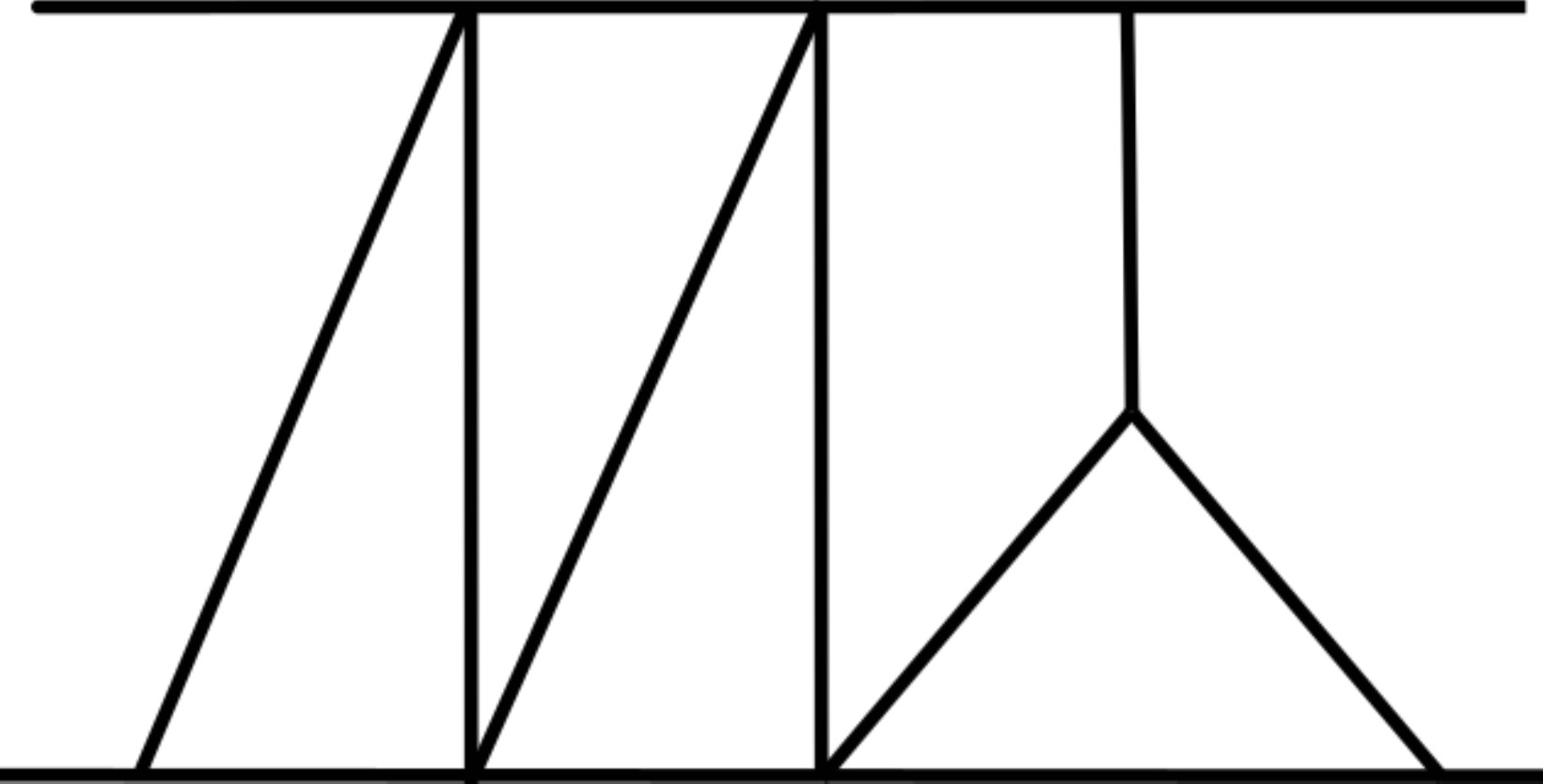}};
    \node at (.3,   1.8){\small $x'$};
    \node at (1.5, 1.7){\small $x$};
    \node at (-1.1, 1.7){\small $v_1$};
    \node at (1.3,  -1){\small $\omega$};

    \node at ( .6, .45){\small $\tau$};
    \node at (.2, -1.8){\small $y$};
    \node at (2.5, -1.8){\small $y_1$};
    \node at (1.6,-.0){\small $p$};
    \node at (-.8,  .45){\small $\eta_1$};
    \node at (-.3,    -.4){\small $\eta_1'$};
    \node at (-1.6, -.4){\small $\gamma$};
    \node at (-2.1, .8){\small $\eta_2$};
    \node at (-2.5, -1.7){\small $z'$};
     \node at (-1.1, -1.7){\small $z$};
    \node at (-1.0,  -2.2){\small $\sigma_2$};
    \node at (-1.0,   2.1){\small $\sigma_1$};
   \end{tikzpicture}
  \caption{}\label{left-2}
 \end{center}
\end{figure}
\textbf{Case 2.} $|\gamma|=3$, as shown in Figure \ref{left-2}. Then $v_1\sim z'$. This yields that $\Pttn(z)=(3,3,3,k_2)$ and $\Phi(z)=\frac{1}{k_2}\geq\frac{1}{12}$.  Let $\eta_2$ be a face lower-adjacent to $\{v_1,z'\}$ such that $\{v_1,z'\}=\eta_2\cap\gamma$.
Since $|v_1|=4$ and $\Pttn(v_1)=(3,3,4,k_1)$ or $(3,3,3,k_1)$, we obtain
$|\eta_2|=3$ or $4$.

If $|\eta_2|=3$, then $\Pttn(v_1)=(3,3,3,k_1)$ and $\Phi'(v_1)\geq\Phi(v_1)=\frac{1}{k_1}\geq \frac{1}{12}$.
If $|\eta_2|=4$, then by the same argument as in Case 1, we obtain that there exists some $B$ such that $v_1\prec B\cap(\partial\sigma_1\cap\partial\sigma_2)$, which yields that $\Phi'(v_1)\geq \frac{1}{12}$.
For $\eta_2$ and $\gamma$, repeating the process as those for $\eta_1$ and $\gamma$, we obtain that the vertices in $\partial\eta_2$ and $\partial\gamma$ has curvature bounded below by $\frac{1}{12}$.

Hence, there are at least $2(k_1-1)$ vertices with modified curvature bounded below by $\frac{1}{12}$.

\textbf{Step 3:} We will show that the sum of the modified curvature is greater than $1$. By $k_1\geq 8$,
 \[
2(k_1-1)\times \frac{1}{12}>1.
\]
This contradicts Theorem \ref{thm:DeVosMohar}.
This proves the theorem.
\end{proof}

Next we will prove the third case.
\begin{prop}\label{vert-in-sigma-conn-anosigma}
Let $\Gamma$ be a $4$-regular infinite planar graph satisfying $\Gamma\in \NNG$. For $i=1,2$, let $\sigma_i$ be a face with facial degree $k_i\geq 8$. Then
\[\d U_1(\sigma_i)\cap \d\sigma_j=\varnothing\] for $1\leq i\neq j\leq 2$.
\end{prop}

\begin{proof}
We will prove it by contradiction. Let $x,y$ be two vertices such that $x\prec\sigma_1$, $y\prec\sigma_2$ and $x\sim y$. Let $\tau$ be a face such that $\{x,y\}\prec \tau$. We divide it into cases.

\begin{figure}[htbp]
 \begin{center}
   \begin{tikzpicture}
    \node at (0,0){\includegraphics[width=0.25\linewidth]{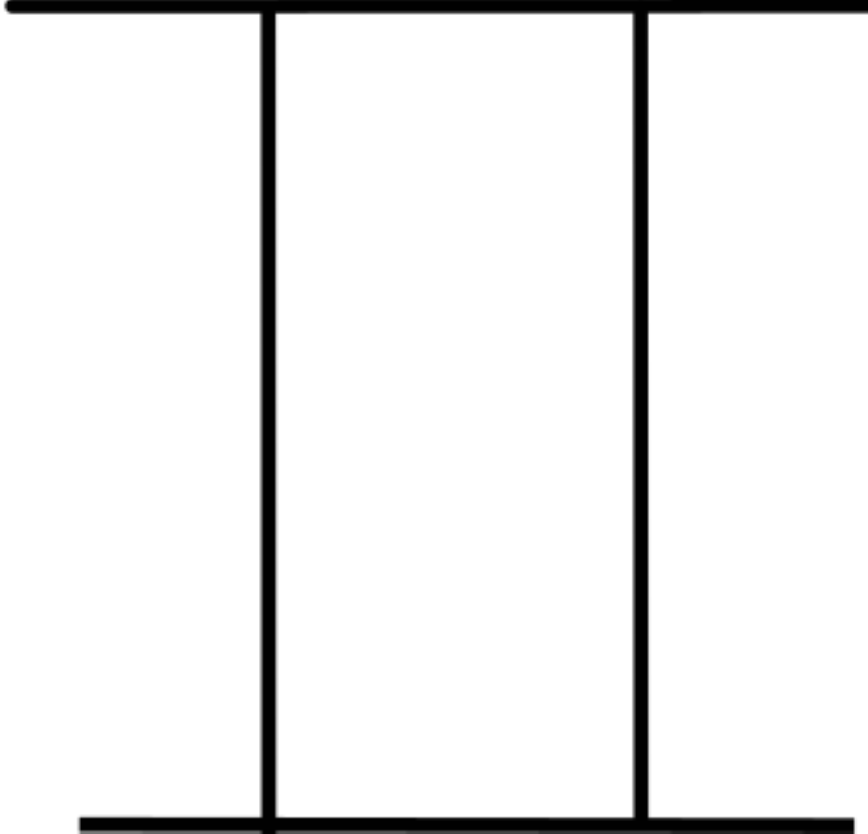}};
     \node at (-1.3,  -2.1){\Large $\sigma_2$};
    \node at (-1.3,   2.1){\Large $\sigma_1$};
     \node at ( .7,  -1.8){\Large $y$};
    \node at ( .8,   1.7){\Large $x$};
     \node at ( .1,   .2){\Large $\tau$};
   \end{tikzpicture}
  \caption{}\label{vertice-of-bothsigma-conn}
 \end{center}
\end{figure}

\textbf{Case 1.} $|\tau|=4$. We divide it into subcases.

\textbf{Case 1.1.} $\tau$ is a square lower adjacent to both $\sigma_1$ and $\sigma_2$, as shown in Figure \ref{vertice-of-bothsigma-conn}.
By Lemma \ref{square-to-trianglespair}, there exist two triangles, denoted by $\tau_1$ and $\tau_1'$, such that $\tau_1$ is $\sigma_1$-adjacent to $\tau$, $\tau_1'$ is $\sigma_2$-adjacent to $\tau$ and $\d\tau_1\setminus\d\sigma_1=\d\tau_1'\setminus\d\sigma_2$. Thus $\tau_1$ and $\tau_1'$ are same as $\tau$ and $\omega$ in Case 3.1 of Theorem \ref{nexist-both-adj}, which is impossible.

\begin{figure}[htbp]
 \begin{center}
   \begin{tikzpicture}
    \node at (0,0){\includegraphics[width=0.3\linewidth]{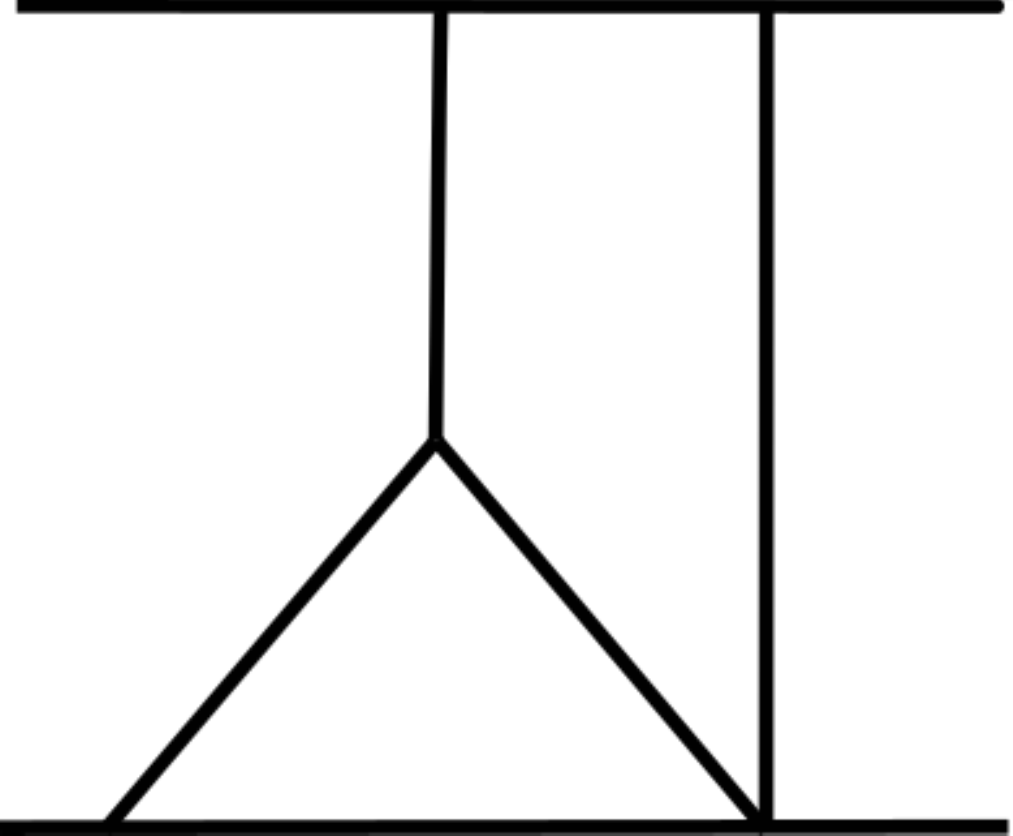}};
    \node at (-1.0,  -2.2){\Large $\sigma_2$};
    \node at (-1.0,   2.2){\Large $\sigma_1$};
    \node at (1.0, 1.8){\Large $x$};
    \node at ( -.3, 1.8){\Large $x'$};
    \node at (1.0,  -1.8){\Large $y$};
    \node at (-1.5,  -1.8){\Large $y_1$};
    \node at (-.1,     0){\Large $p$};
    \node at (-.3,    -1.0){\Large $\omega$};
    \node at (.4,    .2){\Large $\tau$};
   \end{tikzpicture}
  \caption{}\label{vertice-of-bothsigma-conn1}
 \end{center}
\end{figure}
\textbf{Case 1.2.} $\tau$ is a square lower-adjacent to one of $\sigma_1$ and $\sigma_2$, say $\sigma_1$, as shown in Figure \ref{vertice-of-bothsigma-conn1}.
 Then there is a vertex, denoted by $p$, such that $p\nprec(\sigma_1\cup\sigma_2)$ and $\{p,y\}\prec\tau$. Let $\{x,x'\}=\tau\cap\sigma_1$. Then it follows from $|\tau|=4$ that $x'\sim p$. By $|\tau|=4$ and $y\prec\tau$, $\Pttn(y)=(3,3,4,k_2)$. For $|y|=4$, there exists a vertex $y_1\prec \sigma_2$ such that $\{p,y\}$, $\{y,y_1\}$ are contained in a triangle, denoted by $\omega$. $|\omega|=3$ yields that $p\sim y_1$. Hence, $\tau\smile\sigma_1$, $\omega\smile\sigma_2$, $\tau\smile\omega$. Since $|\tau|=4$, $|\omega|=3$, $p\prec\omega$, and $p\nprec (\sigma_1\cup\sigma_2)$, this is impossible by Proposition \ref{non-square-lower-triangle}.

\begin{figure}[htbp]
 \begin{center}
   \begin{tikzpicture}
    \node at (0,0){\includegraphics[width=0.3\linewidth]{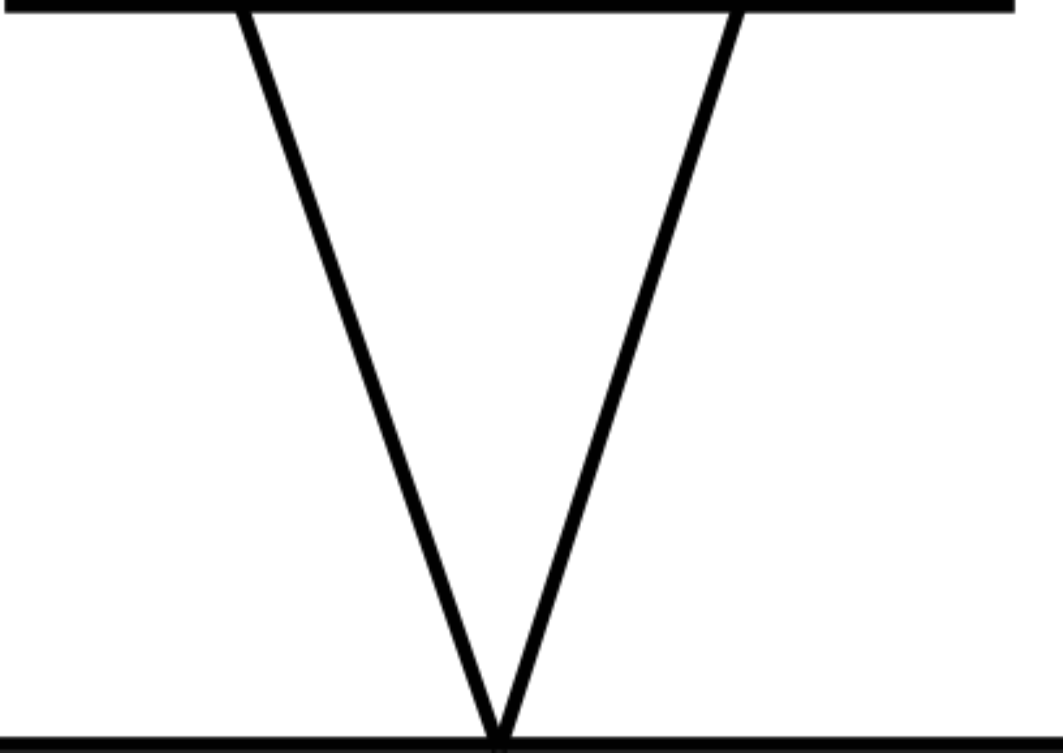}};
    \node at (.8, 1.6){\Large $x$};
    \node at (-.1, -1.6 ){\Large $y$};
    \node at (-1.0,  -2.0){\Large $\sigma_2$};
    \node at (-1.0,   2.1){\Large $\sigma_1$};
    \node at ( 1, -.2){\Large $\omega$};
    \node at (-.2,    .1){\Large $\tau$};
    \node at (-1.1, 1.6){\Large $z$};
    \end{tikzpicture}
  \caption{}\label{vertice-of-bothsigma-conn2}
 \end{center}
\end{figure}

\textbf{Case 2.}
$|\tau|=3$.
Without loss of generality, we assume that $\tau$ lower-adjacent to $\sigma_1$, i.e. there exists $z\prec\sigma_1$ such that $\{z,x\}=\tau\cap\sigma_1$ and $y\sim z$, as shown in Figure \ref{vertice-of-bothsigma-conn2}. Let $\omega$ be a face such that $\{x,y\}=\tau\cap\omega$. For $\Phi(x)\geq 0$, $\Pttn(x)=(3,3,3,k_1)$ or $(3,3,4,k_1)$, which implies that $|\omega|=3$ or $4$.

\begin{figure}[htbp]
 \begin{center}
   \begin{tikzpicture}
    \node at (0,0){\includegraphics[width=0.4\linewidth]{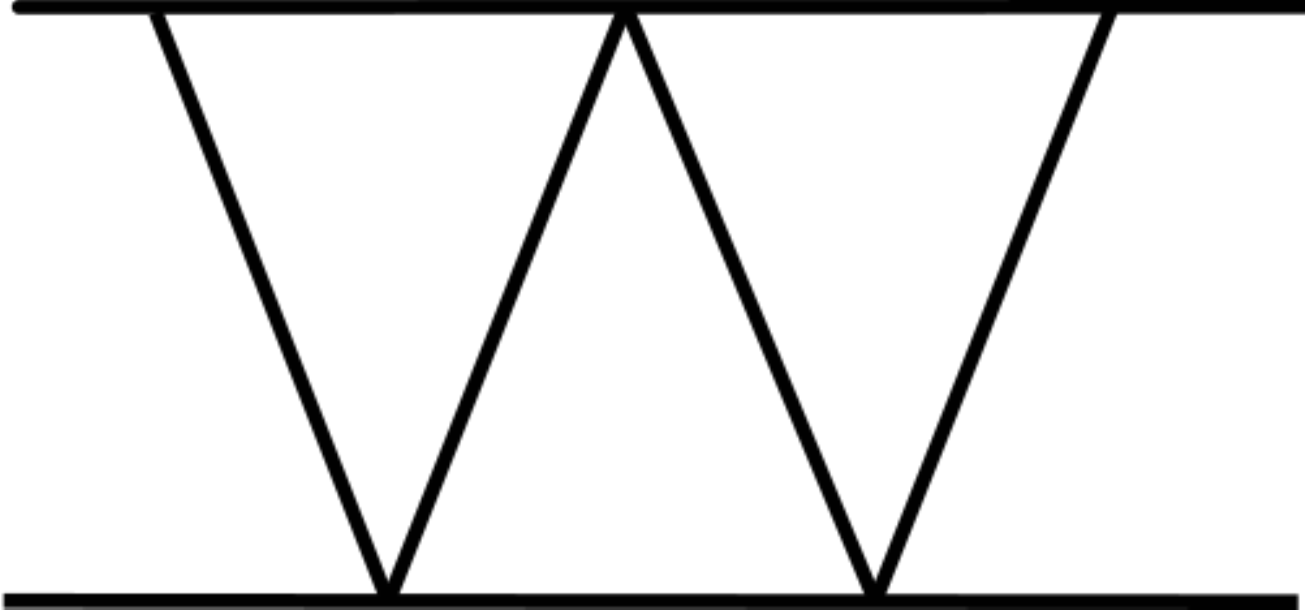}};
    \node at (-1.8, 1.5){\Large $x'$};
    \node at (0, 1.4){\Large $x$};
    \node at (1.7, 1.4){\Large $x_1$};
    \node at (-1.0, -2){\Large $\sigma_2$};
    \node at (-1.0,   1.8){\Large $\sigma_1$};
    \node at ( 1, .3){\Large $\omega'$};
    \node at ( 1.7, -.3){\Large $\omega_1$};
    \node at (0,   -.3){\Large $\omega$};
    \node at ( -1,   .3){\Large $\tau$};
    \node at (-1, -1.4){\Large $y$};
    \node at (.8, -1.5){\Large $y'$};
    \end{tikzpicture}
  \caption{}\label{two-triange-two-comvert-triangle}
 \end{center}
\end{figure}

Since $\{x,y\}\prec\omega$, $x\prec\sigma_1$ and $y\prec\sigma_2$, by Case 1, we obtain $|\omega|\neq 4$. So that $|\omega|=3$. $|y|=4$ implies that there exists a vertex $y'\prec\sigma_2$ such that $\{y,y'\}\prec\sigma_2$. $\{x,y\}$ and $\{y,y'\}$ are contained in $\tau'$ as shown in Figure \ref{two-triange-two-comvert-triangle}. Let $\omega'$ be a face such that $\{x,y'\}=\omega\cap\omega'$. By $|x|=4$, there exists a vertex $x_1\neq x'$ such that $x_1\prec \sigma_1$ and $\omega'$ contains $\{x,x_1\}$ and $\{x,y'\}$. By $\Phi(x)\geq 0$, $\Pttn(x)=(3,3,4,k_1)$ or $(3,3,3,k_1)$. Then $|\omega'|=3$ or $4$. From Case 1, we know that $|\omega'|\neq 4$. Hence, $|\omega'|=3$. This yields that $x_1\sim y'$  and $\Phi(x)=\frac{1}{k_1}\geq \frac{1}{12}$. For $|\omega'|=3$, let $\omega_1$ be a face such that $\{x_1,y'\}=\omega'\cap\omega_1$. By the same arguments as above for $|\omega|=3$, we obtain $|\omega_1|=3$ and $\Phi(y')=\frac{1}{k_2}\geq \frac{1}{12}$.
Repeating the process, we obtain that there are at least $2(\min\{k_1,k_2\}-1)$ vertices in $\d\sigma_1\cup \d\sigma_2$ such that each vertex has modified curvature bounded below by $\frac{1}{12}$. Hence, the sum of curvature of these vertices is bounded below by
\[\frac{2(\min\{k_1,k_2\}-1)}{12}>1
\]
since $k_1,k_2\geq 8$. This contradicts Theorem \ref{thm:DeVosMohar}.

\end{proof}

Now we can prove Theorem \ref{nexist-both-adj}.
\begin{proof}[Proof of Theorem \ref{nexist-both-adj}]
Theorem \ref{nexist-both-adj} follows from Remark \ref{three-cases}, Proposition \ref{prop:disj1}, Theorem \ref{prop:disj2} and Proposition \ref{vert-in-sigma-conn-anosigma}.
\end{proof}

\section{The discharging method and the proof of Theorem~\ref{thm:main2}}\label{section 5}
In this section, we use the discharging method to estimate the number of $k$-gons with $k\geq 8$ in an infinite $4$-regular planar graph with non-negative combinatorial curvature, and give the proof of main theorem, Theorem~\ref{thm:main2}.


First, we introduce the discharging process as follows. We denote by  \[W=\bigcup_{\sigma\in F,8\leq |\sigma|\leq 12}\partial\sigma \] the set of vertices on faces with facial degree $k$.
Recall that for any face $\sigma$ of $\Gamma$ with facial degree $k$, we denote by $\d U_1(\sigma)$ the boundary vertices of $1$-neighborhood of $\sigma$. We set
\[
W_1=\bigcup_{\sigma\in F, 8\leq |\sigma|\leq 12}\d U_1(\sigma).
\]
For the discharging rule, we defined the modified curvature by the discharging as
\begin{align}\begin{split}\label{mod-cur}\wt{\Phi}(x)=
\begin{cases}
\Phi(x)+\frac12 \sum\limits_{y\in W_1, y\sim x}\Phi(y),\ \  & x\in W\\
\Phi(x)-\frac12\#\{y\in W_1|y\sim x\}\Phi(x),\ \  & x\in W_1\\
\Phi(x), &x\in V\setminus(W\cup W_1).
\end{cases}
\end{split}\end{align}
That is, we redistribute the curvature of vertices in $W_1$ to that of vertices in $W$.

One readily checks that the sum of the curvature of all vertices, the total curvature, doesn't change after redistributing the curvature on vertices by Theorem \ref{nexist-both-adj}.
\begin{theorem}\label{sum-keep}
Let $\Gamma$ be a $4$-regular infinite planar graph with non-negative combinatorial curvature. Then for $8\leq k\leq 12$
\begin{equation}\label{eq4:eq10}\wt{\Phi}(\Gamma)= \Phi(\Gamma).\end{equation}
\end{theorem}



For some $k$ satisfying $8\leq k\leq 12$, to estimate the number of $k$-gons, we transfer the charge of vertices near $k$-gons to vertices on $k$-gons such that the total modified curvature on each $k$-gon is uniformly bounded below by $\frac{1}{2}.$
\begin{theorem}\label{12-gon-modcur-lower-bd}
Let $\Gamma$ be a $4$-regular infinite planar graph with non-negative combinatorial curvature and $\sigma$ be a $k$-gon with $8\leq k\leq 12$. Then
\[\sum_{z\in\partial\sigma}\tilde{\Phi}(z)> \frac{1}{2} \]
for $8\leq k\leq 11$ and
\[\sum_{z\in\partial\sigma}\tilde{\Phi}(z)\geq\frac{1}{2} \]
for $k=12$.
\end{theorem}

Since the total curvature of the graph is at most $1$ and $1$-neighborhoods of $k$-gons are disjoint for $8\leq k\leq 12$, see Theorem \ref{nexist-both-adj}. We may estimate the number of $k$-gons with $8\leq k\leq 12$.


\begin{lemma}\label{Two-triang-adj}
Let $\Gamma$ be a $4$-regular infinite planar graph with non-negative combinatorial curvature and $\sigma$ be a $k$-gon with $8\leq k\leq 12$. Assume that $\tau_i$, $i=1,2$ are triangles and $\tau_1$ is $\sigma$-adjacent to $\tau_2$. If $z=\partial\tau_1\cap\partial\tau_2$, then
\[
\tilde{\Phi}(z)\geq\frac{1}{2k}.
\]
The equality holds if and only if $k=12$, one of $\Pttn(\partial\tau_1\setminus\partial\sigma)$ and $\Pttn(\partial\tau_2\setminus\partial\sigma)$ is $(3,4,4,4)$, the other is $(3,4,4,6)$.
\end{lemma}

\begin{proof}
By Proposition \ref{nonadj1} and $8\leq k\leq 12$, $z\prec \sigma$. Since $\Phi(z)\geq 0$, $\Pttn(z)$ is $(3,3,3,k)$ or $(3,3,4,k)$. If $\Pttn(z)=(3,3,3,k)$, then $\tilde{\Phi}(z)\geq \Phi(z)=\frac{1}{k}>\frac{1}{2k}$.
Next we consider the case $\Pttn(z)=(3,3,4,k)$, which yields
 \[
 \Phi(z)=\frac{1}{k}-\frac{1}{12}.
 \]
 Let $\omega_1$ and $\omega_2$ be two other faces, $\sigma$-adjacent to $\tau_1$ and $\tau_2$ respectively. Let $z_i=\partial \omega_i\cap \partial \tau_i,$ $i=1,2$. Proposition \ref{nonadj1} implies that $z_i\in \partial\sigma$.

\textbf{Case 1.} $|\omega_1|=|\omega_2|=3.$ Let $y$ be the vertex such that $y\sim x_1,$ $y\sim x_2$ and $y\not \in \partial \sigma.$
Let $\eta_1$ (resp. $\eta_2$) be the face containing $x_1$ and $y$ (resp. $x_2$ and $y$) and satisfying $\eta_1\cap\sigma=\varnothing$ (resp. $\eta_2\cap\sigma=\varnothing$).

By considering the vertex pattern at $z_i$, we have $\Pttn(z_i)=(3,3,3,k)$ or $(3,3,4,k)$. We divide it into subcases.

\textbf{Case 1.1.}  $\Pttn(z_1)=\Pttn(z_2)=(3,3,4,k)$. Then there are three squares containing $z,z_1,z_2$, as shown in Figure~\ref{Fig:m-1-2}. \begin{figure}[htbp]
 \begin{center}
   \begin{tikzpicture}
    \node at (0,0){\includegraphics[width=0.7\linewidth]{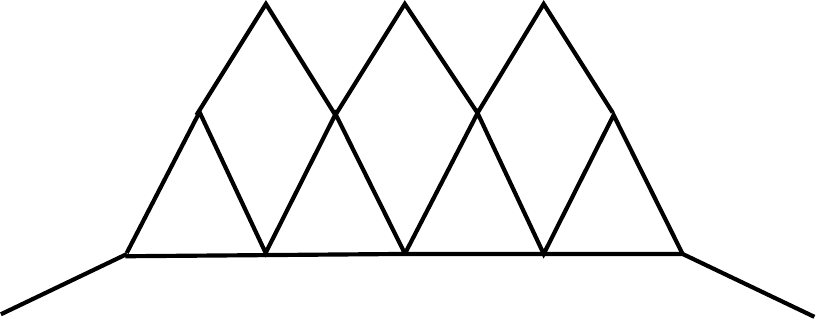}};
    \node at (-0.35, 0.4){\Large $x_1$};
        \node at (1.15, 0.4){\Large $x_2$};
       \node at (-.7,   1.5){\Large $\eta_1$};
        \node at (0,   1.2){\Large $y$};
    \node at (.8,   1.5){\Large $\eta_2$};
        \node at (-.8,  -.5){\Large $\tau_1$};
        \node at (.8,  -.5){\Large $\tau_2$};
      \node at (2.2,  -.5){\Large $\omega_2$};
    \node at (-2.2,  -.5){\Large $\omega_1$};
     \node at (0,  -1.35){\Large $z$};
     \node at (-1.5,  -1.35){\Large $z_1$};
     \node at (1.5,  -1.35){\Large $z_2$};
       \node at (-3.5,  -1.6){\Large $\sigma$};

   \end{tikzpicture}
  \caption{}\label{Fig:m-1-2}
 \end{center}
\end{figure}
 Then for any $i=1,2,$ $|\eta_i|=3,4,5,$ or $6.$ Since $\Phi(y)\geq 0,$ one of $|\eta_1|$ and $|\eta_2|$ is $3$ or $4.$ Without loss of generality, we assume that $|\eta_1|=3$ or $4.$ Then the vertex pattern of $x_1$ is either $(3,3,4,4)$ or $(3,4,4,4),$ which yields that $\Phi(x_1)= \frac{1}{6}$ or $\Phi(x_1)=\frac{1}{12}$, respectively. Hence,
\[
\Phi(x_1)+\Phi(x_2)\geq \frac{1}{12},
\]
the equality holds if and only if $\Phi(x_1)=\frac{1}{12}$ and $\Phi(x_2)=0$, that is, $\Pttn(x_1)=(3,4,4,4)$ and $\Pttn(x_2)=(3,4,4,6)$. By $\Phi(z)=\frac{1}{k}-\frac{1}{12},$ we obtain
\[
\tilde{\Phi}(z)=\Phi(z)+\frac{1}{2}\left(\Phi(x_1)+\Phi(x_2)\right)\geq\frac{1}{2k},
\]
where the equality holds if and only if $k=12$, $\Pttn(x_1)=(3,4,4,4)$ and $\Pttn(x_2)=(3,4,4,6)$ or $\Pttn(x_1)=(3,4,4,6)$ and $\Pttn(x_2)=(3,4,4,4)$.

\begin{figure}[htbp]
 \begin{center}
   \begin{tikzpicture}
    \node at (0,0){\includegraphics[width=0.7\linewidth]{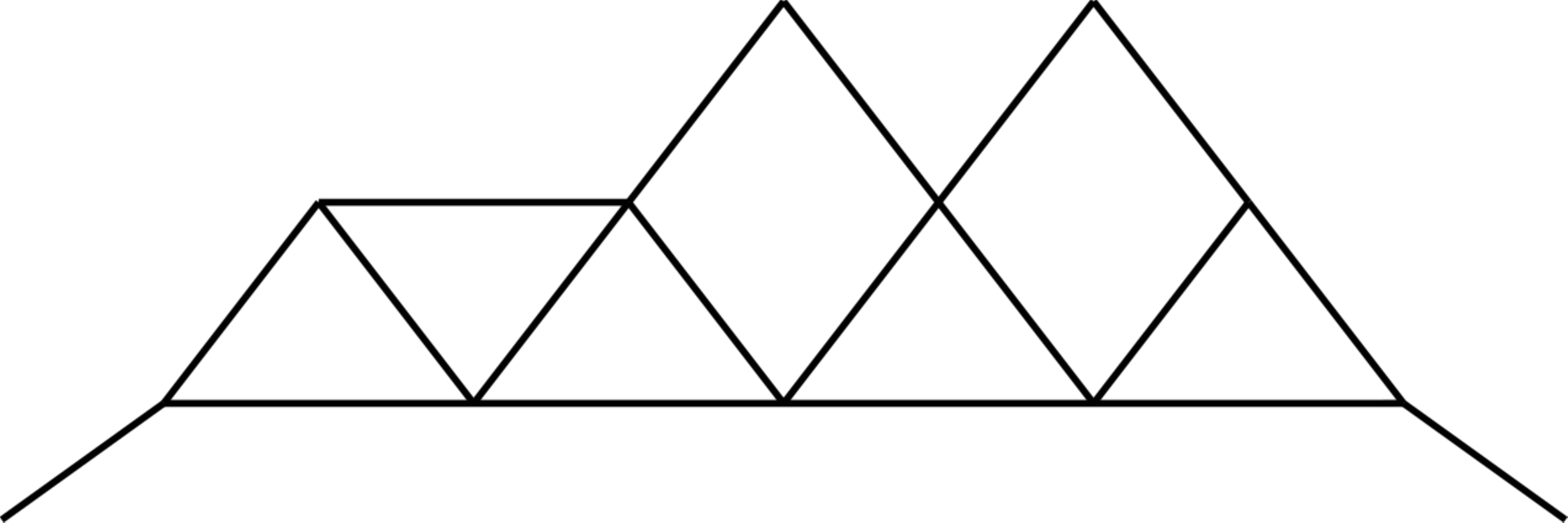}};
    \node at (-0.4, 0.3){\Large $x_1$};
        \node at (1.3, 0.3){\Large $x_2$};
       \node at (-1,   1.0){\Large $\eta_1$};
        \node at (0,   1.7){\Large $y$};
    \node at (0.8,   1.2){\Large $\eta_2$};
        \node at (-.8,  -.4){\Large $\tau_1$};
        \node at (.8,  -.4){\Large $\tau_2$};
      \node at (2.6,  -.4){\Large $\omega_2$};
    \node at (-2.6,  -.4){\Large $\omega_1$};
     \node at (0,  -1.2){\Large $z$};
     \node at (-1.7,  -1.2){\Large $z_1$};
     \node at (1.8,  -1.2){\Large $z_2$};
       \node at (-3.5,  -1.6){\Large $\sigma$};

   \end{tikzpicture}
  \caption{}\label{two-triangles-344}
 \end{center}
\end{figure}

\textbf{Case 1.2.} One of $\Pttn(z_1)$ and $\Pttn(z_2)$ is $(3,3,4,k)$, the other is $(3,3,3,k)$. Without loss of generality, we assume that $\Pttn(z_1)=(3,3,3,k)$ and $\Pttn(z_2)=(3,3,4,k)$.
As shown in Figure~\ref{two-triangles-344}, we claim that $|\eta_1|\leq 5$ or $|\eta_2|\leq 5$.
Otherwise, $|\eta_i|\geq 5$ for $i=1,2$, then $\Pttn(y)=(l,4,|\eta_1|,|\eta_2|)$ with $3\leq l\leq 12$, which yields $\Phi(y)<0$. This contradicts $\Phi(y)\geq 0$ and proves the claim.
Since $\Pttn(x_2)=(3,4,4,|\eta_2|)$ and $\Phi(x_2)\geq0$, $|\eta_2|\leq 6$.

If $5\leq|\eta_2|\leq 6$, then $|\eta_1|\leq 4$. We have $\Phi(x_1)\geq \frac{1}{6}$. Hence,
\[
\tilde{\Phi}(z)\geq\Phi(z)+\frac{1}{2}\Phi(x_1)=\frac{1}{k}>\frac{1}{2k}.
\]
If $|\eta_2|\leq 4$, then $\Pttn(x_2)=(3,4,4,|\eta_2|)$ and $\Phi(x_2)\geq\frac{1}{12}$. Since $x_1\prec \eta_1$, $z\prec\sigma$, $x_1\sim z$ and $8\leq |\sigma|\leq 12$, by Proposition \ref{vert-in-sigma-conn-anosigma}, we obtain that $|\eta_1|\leq 7$, so that $\Phi(x_1)>0$. Hence,
\[
\tilde{\Phi}(z)=\Phi(z)+\frac{1}{2}\left(\Phi(x_1)+\Phi(x_2)\right)>\frac{1}{k}-\frac{1}{24}\geq\frac{1}{2k}.
\]

\begin{figure}[htbp]
 \begin{center}
   \begin{tikzpicture}
    \node at (0,0){\includegraphics[width=0.7\linewidth]{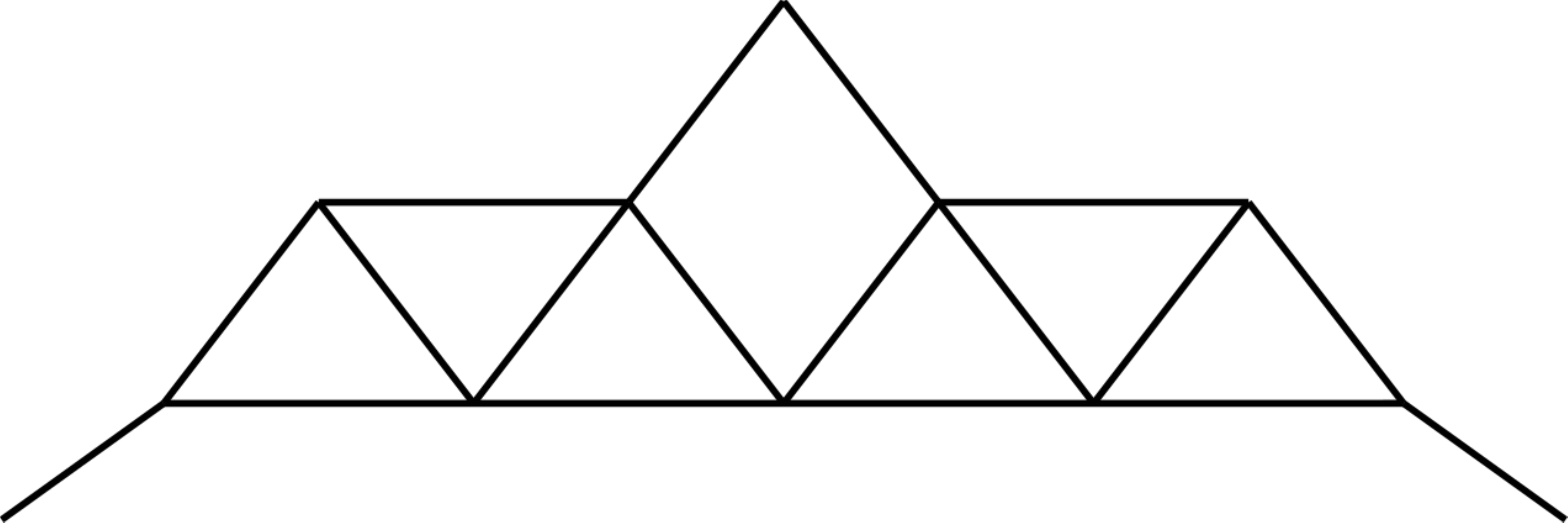}};
    \node at (-0.4, 0.3){\Large $x_1$};
        \node at (1.2, 0.6){\Large $x_2$};
       \node at (-1,   1.0){\Large $\eta_1$};
        \node at (0,   1.7){\Large $y$};
    \node at (0.8,   1.2){\Large $\eta_2$};
        \node at (-.8,  -.4){\Large $\tau_1$};
        \node at (.8,  -.4){\Large $\tau_2$};
      \node at (2.6,  -.4){\Large $\omega_2$};
    \node at (-2.6,  -.4){\Large $\omega_1$};
     \node at (0,  -1.2){\Large $z$};
     \node at (-1.7,  -1.2){\Large $z_1$};
     \node at (1.8,  -1.2){\Large $z_2$};
       \node at (-3.5,  -1.6){\Large $\sigma$};

   \end{tikzpicture}
  \caption{}\label{two-triangles-343}
 \end{center}
\end{figure}

\textbf{Case 1.3.} $\Pttn(z_1)=\Pttn(z_2)=(3,3,3,k)$. As shown in Figure \ref{two-triangles-343}, we claim that $|\eta_1|\leq 5$ or $|\eta_2|\leq 5$.
Otherwise, $|\eta_i|\geq 5$ for $i=1,2$, then $\Pttn(y)=(l,4,|\eta_1|,|\eta_2|)$ with $3\leq l\leq 12$, which yields $\Phi(y)<0$. This contradicts $\Phi(y)\geq 0$ and proves the claim.
Without loss of generality, we assume that $|\eta_1|\leq |\eta_2|$, then $|\eta_1|\leq 5$. Hence, $\Pttn(x_1)=(3,3,4,|\eta_1|)$ and $\Phi(x_1)>\frac{1}{12}$, which implies that
\[\tilde{\Phi}(z)\geq \Phi(z)+\frac{1}{2}\Phi(x_1)>\frac{1}{k}-\frac{1}{24}\geq \frac{1}{2k}.\]

\textbf{Case 2.} $|\omega_1|=4$ and $|\omega_2|=3$.
Let $y$ be the vertex satisfying $y\sim x_i(i=1,2)$ and $y\notin \d\sigma$. Let $\eta_i$ be the face containing the edge $\{y, x_i\}$ and $\eta_i\cap\sigma=\varnothing$, $i=1,2$. Clearly, $\Pttn(z_1)=(3,3,4,k)$, $\Pttn(z_2)=(3,3,3,k)$ or $(3,3,4,k)$. We divide it into subcases.

\begin{figure}[htbp]
 \begin{center}
   \begin{tikzpicture}
    \node at (0,0){\includegraphics[width=0.7\linewidth]{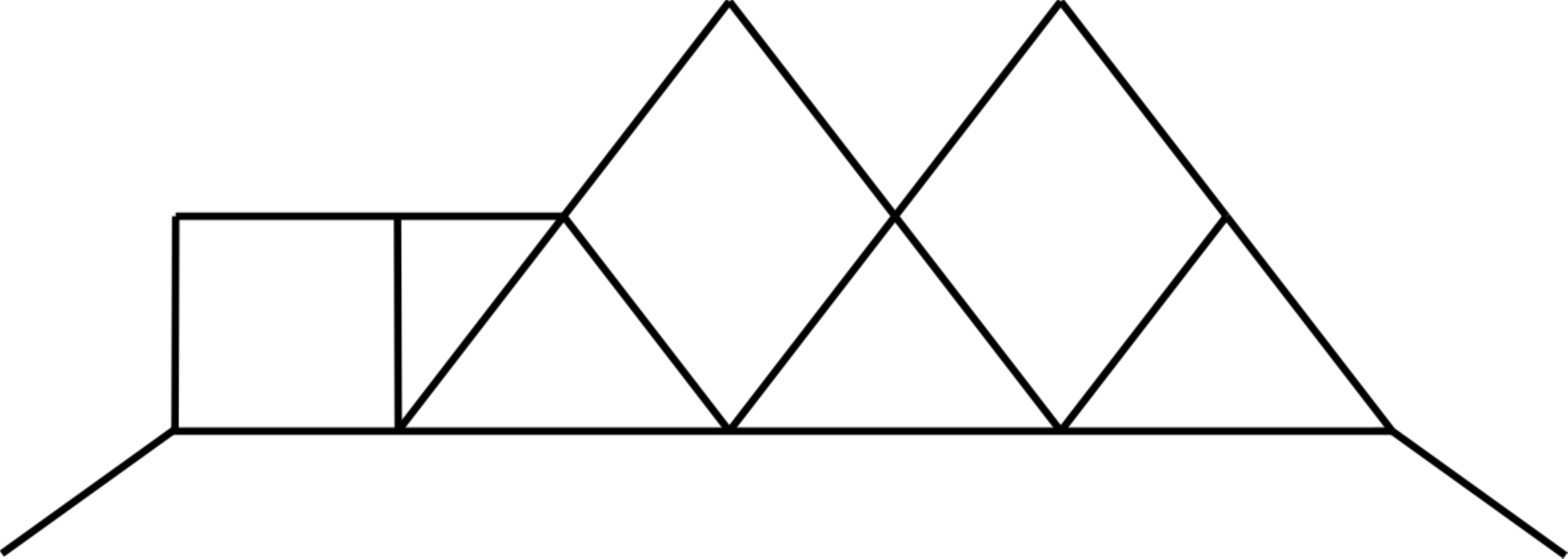}};
    \node at (-0.8, 0.25){\Large $x_1$};
        \node at (1, 0.25){\Large $x_2$};
    \node at (-1.4,   1){\Large $\eta_1$};
    \node at ( .7,   1){\Large $\eta_2$};
    \node at (-.3,   1.8){\Large $y$};
    \node at (-1.2,  -.4){\Large $\tau_1$};
        \node at (.7,  -.4){\Large $\tau_2$};
      \node at (2.5,  -.4){\Large $\omega_2$};
    \node at (-2.8,  -.4){\Large $\omega_1$};
           \node at (-3.5,  -1.4){\Large $\sigma$};
    \node at (-0.25,  -1.2){\Large $z$};
    \node at (-2,  -1.2){\Large $z_1$};
    \node at (1.6,  -1.2){\Large $z_2$};
   \end{tikzpicture}
  \caption{}\label{square-triangle-44}
 \end{center}
\end{figure}

\textbf{Case 2.1.} $\Pttn(z_2)=(3,3,4,k)$. As shown in Figure \ref{square-triangle-44}, with the same argument as in Case 1.2, we obtain
$|\eta_1|\leq 5$ or $|\eta_2|\leq 5$. Note that $|\eta_2|\leq 6$ by $\Pttn(x_2)=(3,4,4,|\eta_2|)$ and $\Phi(x_2)\geq 0$.
For any case that $|\eta_2|\leq 6$, it is easy to obtain
\[
\Phi(x_1)+\Phi(x_2)>\frac{1}{12}.
\]
Hence,
\[
\tilde{\Phi}(z)=\Phi(z)+\frac{1}{2}\left(\Phi(x_1)+\Phi(x_2)\right)
> \frac{1}{k}-\frac{1}{24}\geq \frac{1}{2k}.
\]

\textbf{Case 2.2.} $\Pttn(z_2)=(3,3,3,k)$. As shown in Figure \ref{square-triangle-43}, with the same argument as in Case 1.3, we obtain
\[
\tilde{\Phi}(z)=\Phi(z)+\frac{1}{2}\left(\Phi(x_1)+\Phi(x_2)\right)
>\frac{1}{2k}.
\]

\begin{figure}[htbp]
 \begin{center}
   \begin{tikzpicture}
    \node at (0,0){\includegraphics[width=0.7\linewidth]{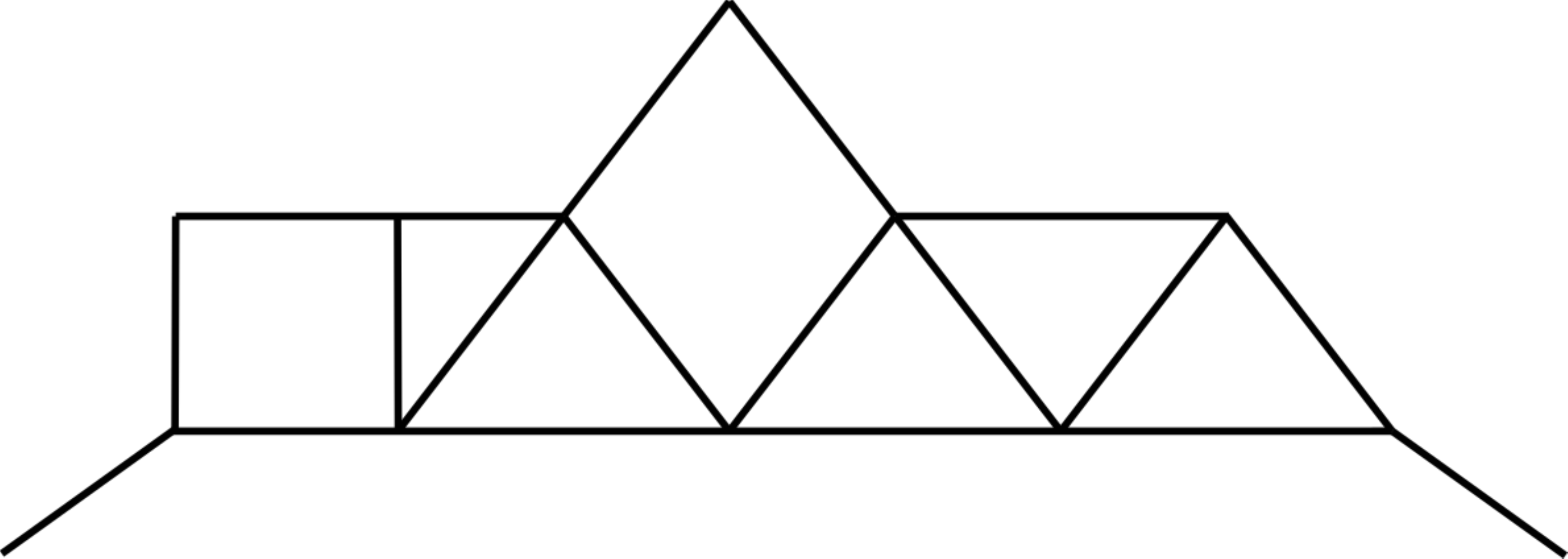}};
    \node at (-0.8, 0.35){\Large $x_1$};
        \node at (0.9, 0.55){\Large $x_2$};
    \node at (-1.4,   1){\Large $\eta_1$};
    \node at ( .9,   1.2){\Large $\eta_2$};
    \node at (-.3,   1.8){\Large $y$};
    \node at (-1.2,  -.4){\Large $\tau_1$};
        \node at (.7,  -.4){\Large $\tau_2$};
      \node at (2.5,  -.4){\Large $\omega_2$};
    \node at (-2.8,  -.4){\Large $\omega_1$};
           \node at (-3.5,  -1.4){\Large $\sigma$};
    \node at (-0.25,  -1.2){\Large $z$};
    \node at (-2,  -1.2){\Large $z_1$};
    \node at (1.6,  -1.2){\Large $z_2$};
   \end{tikzpicture}
  \caption{}\label{square-triangle-43}
 \end{center}
\end{figure}
\begin{figure}[htbp]
 \begin{center}
   \begin{tikzpicture}
    \node at (0,0){\includegraphics[width=0.7\linewidth]{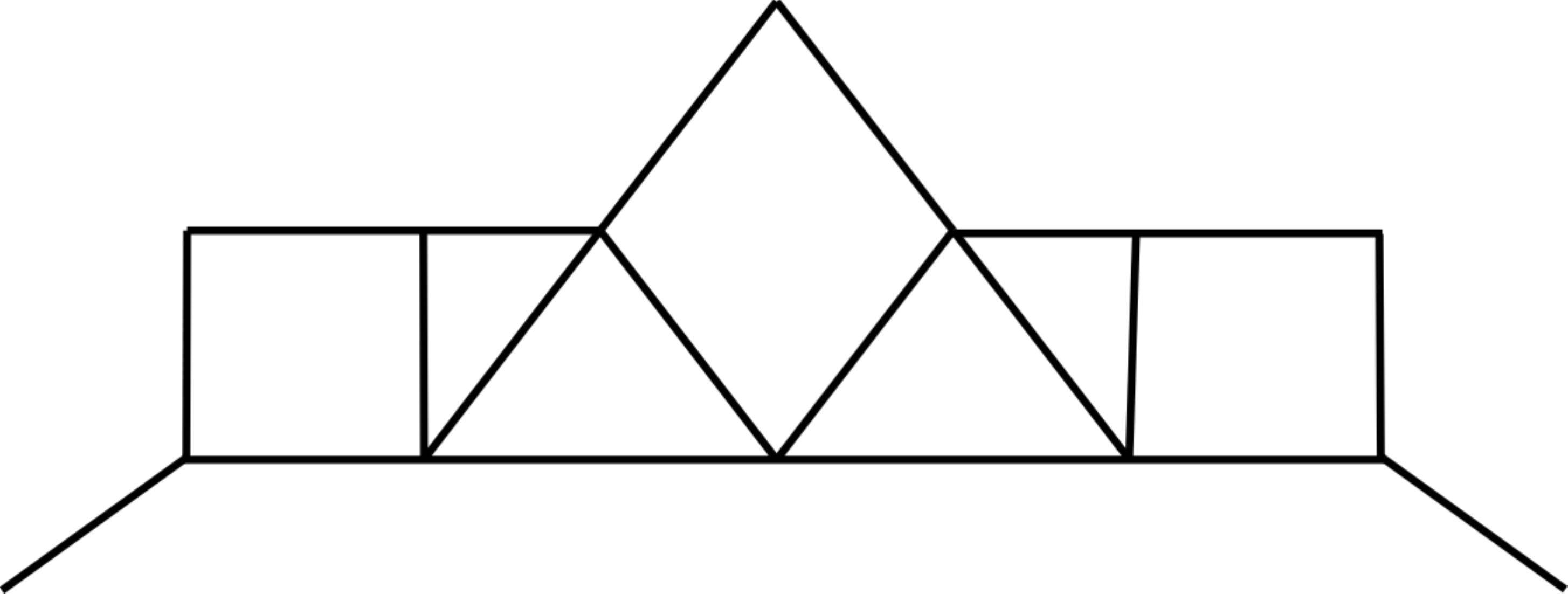}};
    \node at (-0.6, .3){\Large $x_1$};
        \node at (0.5, 0.3){\Large $x_2$};
    \node at (-1.2,   1){\Large $\eta_1$};
    \node at (0.2,   1.9){\Large $y$};
    \node at (-.9,  -.5){\Large $\tau_1$};
        \node at (.85,  -.5){\Large $\tau_2$};
        \node at (1.4,  0.9){\Large $\eta_2$};
      \node at (2.5,  -0.5){\Large $\omega_2$};
    \node at (-2.6,  -0.5){\Large $\omega_1$};
    \node at (-2.0,  -1.2){\Large $z_1$};
    \node at (-0,  -1.2){\Large $z$};
    \node at (2,  -1.2){\Large $z_2$};
           \node at (-3.6,  -1.5){\Large $\sigma$};
   \end{tikzpicture}
  \caption{}\label{two-squares}
 \end{center}
\end{figure}

\textbf{Case 3.} $|\omega_1|=|\omega_2|=4$. By $\Phi(z_i)\geq0$ for $i=1,2$, we obtain $\Phi(z_1)=\Phi(z_2)=(3,3,4,k)$, as in shown Figure~\ref{two-squares}. If $|\eta_1|\geq 7$, then by $\Phi(y)\geq 0$, we have $|\eta_2|=3$, which yields that $\Pttn(x_2)=(3,3,3,4)$ and $\Phi(x_2)=\frac{1}{4}$. Hence,
\[
\tilde{\Phi}(z)\geq \Phi(z)+\frac{1}{2}\Phi(x_2)>\frac{1}{2k}.\]
If $|\eta_1|=6$, then $\Pttn(x_1)=(3,3,4,6)$ and $\Phi(x_1)=\frac{1}{12}$. By $\Phi(y)\geq 0$, we obtain $|\eta_2|\leq 4$, which yields that $\Phi(x_2)>0$. Hence, \[\tilde{\Phi}(z)=\Phi(z)+\frac{1}{2}\left(\Phi(x_1)+\Phi(x_2)\right)>\frac{1}{k}-\frac{1}{24}\geq \frac{1}{2k}.\]
If $|\eta_1|\leq 5$, then $\Pttn(x_1)=(3,3,4,|\eta_1|)$ and $\Phi(x_1)=\frac{1}{|\eta_1|}-\frac{1}{12}>\frac{1}{12}$. Hence,
\[\tilde{\Phi}(z)\geq \Phi(z)+\frac{1}{2}\Phi(x_1)>\frac{1}{k}-\frac{1}{24}\geq \frac{1}{2k}.\]
\end{proof}

\begin{lemma}\label{one-triangle-one-square}
Let $\Gamma$ be a $4$-regular infinite planar graph with non-negative combinatorial curvature and $\sigma$ be a $k$-gon with $k\geq 8$. Assume $\omega$ is a square and $\tau_1$ is a triangle such that $\omega$ is $\sigma$-adjacent to $\tau_1$. If $z=\partial\omega\cap\partial\tau_1\cap\partial \sigma$, then
\[
\tilde{\Phi}(z)\geq\frac{1}{2k}.
\]
Moreover, if the equality holds, then there exists a vertex $z_1\in\partial\sigma$ such that $z\sim z_1$ and $\tilde{\Phi}(z_1)>\frac{1}{2k}$.
\end{lemma}

\begin{proof}
Clearly, $\Pttn(z)=(3,3,4,k)$ and $\Phi(z)=\frac{1}{k}-\frac{1}{12}$. Let $\tau_2$ be a face such that $\tau_2$ $\sigma$-adjacent to $\tau_1$. Let $z_1=\partial\tau_1\cap\partial\tau_2$. By Proposition \ref{nonadj1}, $z_1\in\partial\sigma$ and $|\tau_2|=3$ or $4$. We divide it into cases.

\textbf{Case 1.} $|\tau_2|=3$. Let $x_i=\partial\tau_i\setminus \partial\sigma$, $i=1,2$. Let $z_1=\partial\tau_1\cap\partial\tau_2$. By Proposition \ref{nonadj1}, $z_1\in\partial\sigma$ and $\Pttn(z_1)=(3,3,3,k)$ or $(3,3,4,k)$. We divide it into subcases.


\begin{figure}[htbp]
 \begin{center}
   \begin{tikzpicture}
    \node at (0,0){\includegraphics[width=0.75\linewidth]{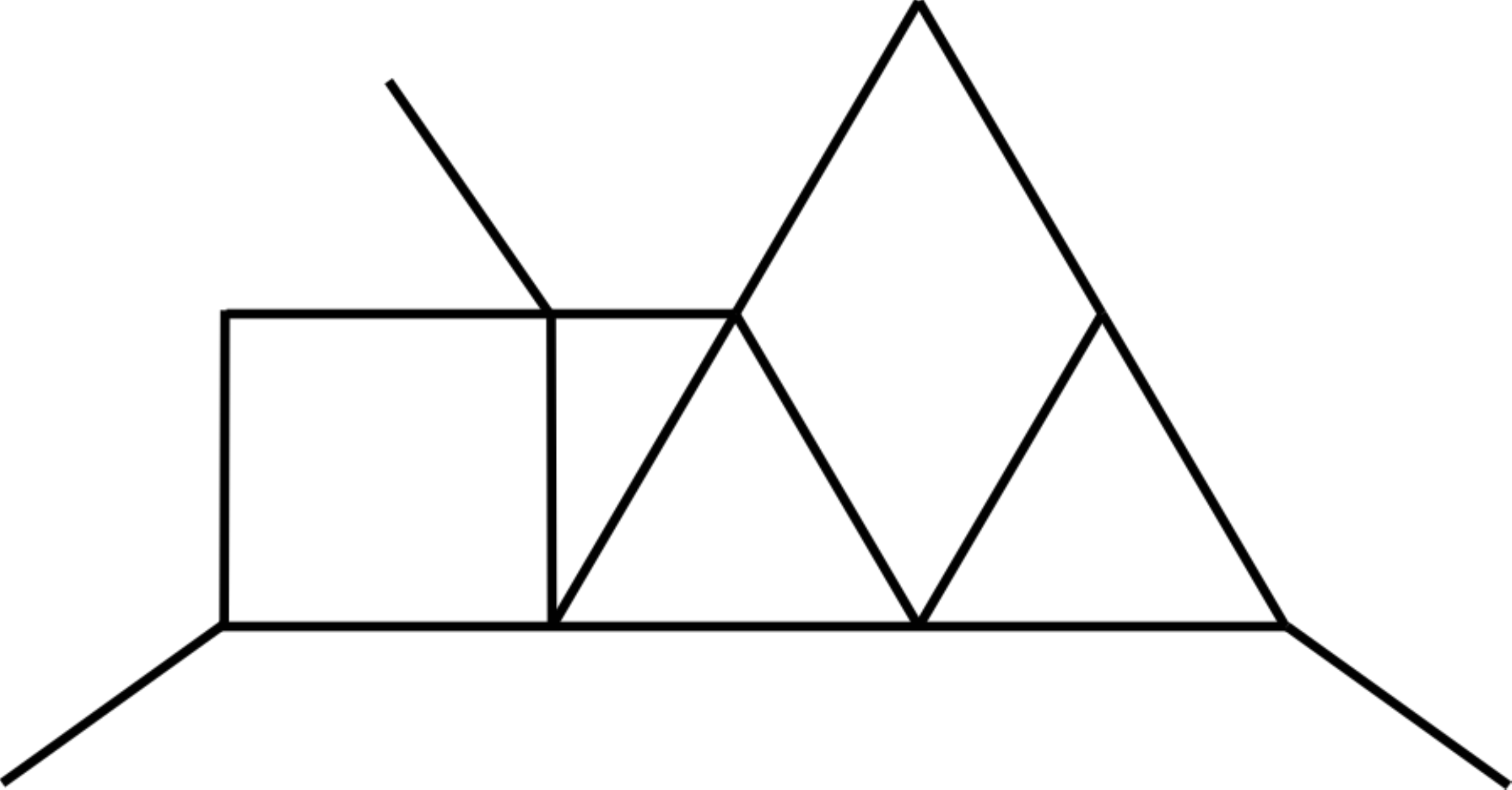}};
    \node at (0.25, 0.5){\Large $x_1$};
    \node at (2.5, 0.5){\Large $x_2$};
    \node at (-1.2,   1.6){\Large $\eta_1$};
    \node at (2, 1.8){\Large $\eta_2$};
    \node at (0.6,   2.4){\Large $y$};
    \node at (-.1,  -.5){\Large $\tau_1$};
    \node at (2.2,  -.5){\Large $\tau_2$};
    \node at (-2.4,  -0.5){\Large $\omega$};
    \node at (-1.3,  -1.8){\Large $z$};
    \node at (-1.1,  0.8){\Large $y_1$};
    \node at (1,  -1.8){\Large $z_1$};
   \node at (-3.6,  -2.0){\Large $\sigma$};
   \end{tikzpicture}
  \caption{}\label{square-triangles-4}
 \end{center}
\end{figure}
\textbf{Case 1.1.} $\Pttn(z_1)=(3,3,4,k)$.
Let $y$ be the vertex such that $y\sim x_1,$ $y\sim x_2$ and $y\not\prec \sigma.$ Let $\eta_1$ be the face containing $x_1$ and $y$ and satisfying $\eta_1\cap\sigma=\varnothing$, as shown in Figure \ref{square-triangles-4}. Let $y_1$ be a vertex such that $y_1\sim x_1$, $y_1\sim z$, $y_1\not\prec \sigma$ and $y_1\prec\omega$.

If $|\eta_1|\leq 5$, then $\Pttn(x_1)=(3,3,4,|\eta_1|)$ and $\Phi(x_1)=\frac{1}{|\eta_1|}-\frac{1}{12}>\frac{1}{12}$.
Hence,
\[\tilde{\Phi}(z)\geq \Phi(z)+\frac{1}{2}\Phi(x_1)>\frac{1}{k}-\frac{1}{24}\geq \frac{1}{2k}.
\]
If $|\eta_1|=6$, then $\Pttn(x_1)=(3,3,4,6)$ and $\Phi(x_1)=\frac{1}{12}$. Note that $\Pttn(y_1)=(3,3,4,6)$ or $(3,4,4,6)$. For $\Pttn(y_1)=(3,3,4,6)$, $\Phi(y_1)=\frac{1}{12}>0$, which implies that
\[
\tilde{\Phi}(z)=\Phi(z)+\frac{1}{2}\left(\Phi(y_1)+\Phi(x_1)\right)\geq \frac{1}{k}>\frac{1}{2k}.
\]
For $\Pttn(y_1)=(3,4,4,6)$, $\Phi(y_1)=0$, which implies that
\[
\tilde{\Phi}(z)=\Phi(z)+\frac{1}{2}\left(\Phi(y_1)+\Phi(x_1)\right)\geq \frac{1}{k}-\frac{1}{24}\geq \frac{1}{2k}.
\]
Let $\eta_2$ be the face containing $y$ and $x_2$ and satisfying $\eta_2\cap\sigma=\varnothing$. By $|\eta_1|=6$ and $\Phi(y)\geq 0$, $|\eta_2|\leq 4$. Using Proposition \ref{nonadj1}, we obtain $\Pttn(x_2)=(3,3,4,4)$ or $(3,4,4,4)$. Then $\Phi(x_2)=\frac{1}{12}$.
Hence,
\[
\tilde{\Phi}(z_1)=\Phi(z_1)+\frac{1}{2}\left(\Phi(x_1)+\Phi(x_2)\right)=\frac{1}{k}>\frac{1}{2k}.
\]

If $|\eta_1|=7$, then $\Pttn(y_1)=(3,3,4,|\eta_1|)$ and $\Pttn(x_1)=(3,3,4,7)$. That is $\Phi(y_1)=\Phi(x_1)=\frac{5}{84}$.
Hence,
\[
\tilde{\Phi}(z)=\Phi(z)+\frac{1}{2}\left(\Phi(y_1)+\Phi(x_1)\right)=\frac{1}{k}-\frac{1}{12}+\frac{5}{84}>\frac{1}{2k}.
\]
Since $x_1\prec\eta_1$, $z\prec\sigma$, $x_1\sim z$ and $|\sigma|=k$, by Proposition \ref{vert-in-sigma-conn-anosigma}, we obtain that $|\eta_1|\geq 8$ is impossible. Hence, $|\eta_1|\leq 7$.

\begin{figure}[htbp]
 \begin{center}
   \begin{tikzpicture}
    \node at (0,0){\includegraphics[width=0.75\linewidth]{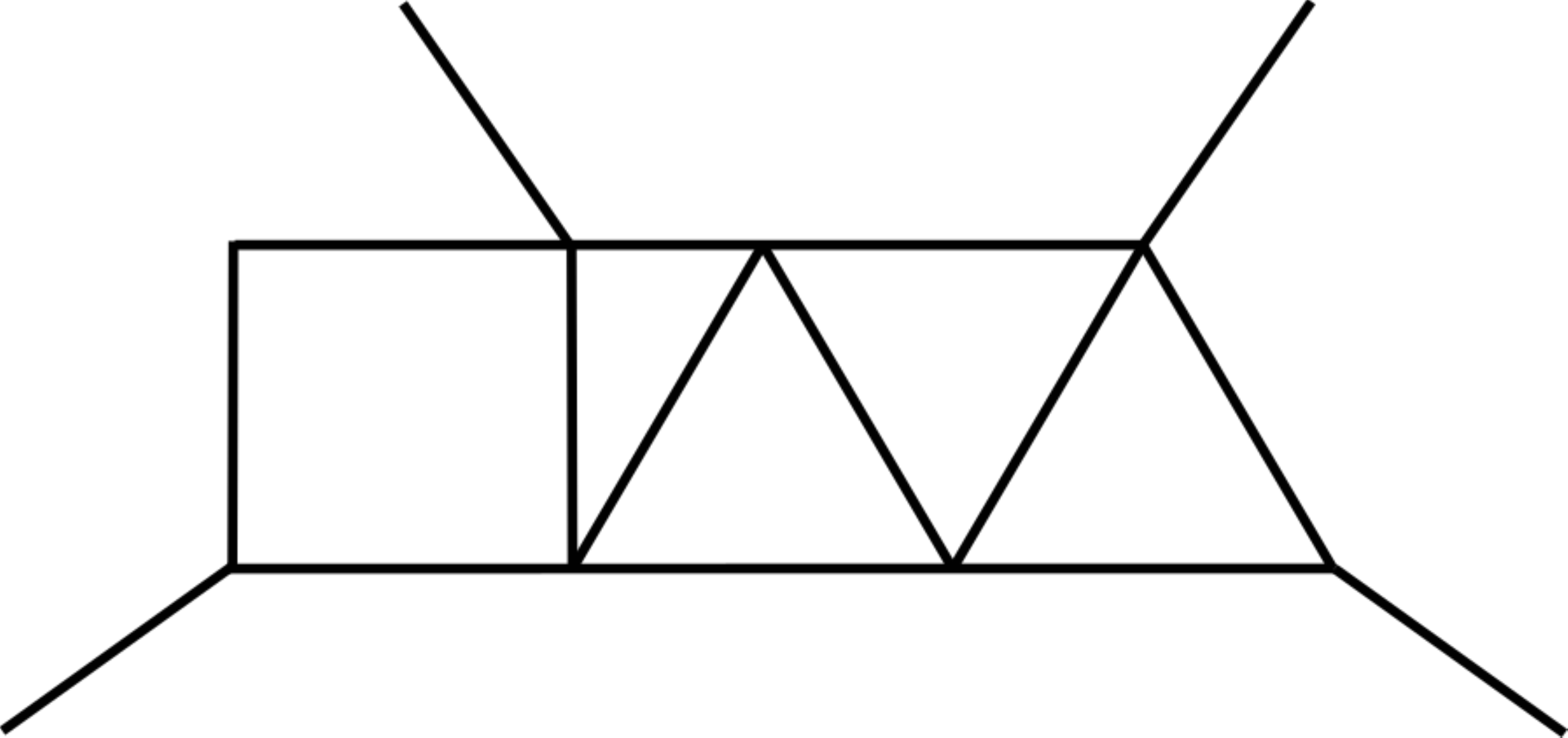}};
    \node at (-.1, 1){\Large $x_1$};
    \node at (2.0, 1){\Large $x_2$};
    \node at (-1.2,   2){\Large $\eta$};
    \node at (-.1,  -.5){\Large $\tau_1$};
    \node at (2.2,  -.5){\Large $\tau_2$};
    \node at (-2.4,  -0.5){\Large $\omega$};
    \node at (-1.3,  -1.6){\Large $z$};
    \node at (-1.1,  1){\Large $y_1$};
    \node at (1,  -1.6){\Large $z_1$};
   \node at (-3.6,  -2.0){\Large $\sigma$};
   \end{tikzpicture}
  \caption{}\label{square-triangles-3}
 \end{center}
\end{figure}

\textbf{Case 1.2.} $\Pttn(z_1)=(3,3,3,k)$. Let $y_1$ be a vertex such that $y_1\sim x_1$, $y_1\sim z$, $y_1\not\prec \sigma$ and $y_1\prec\omega$.
By $|x_1|=4$, there exists a face containing $y_1, x_1, x_2$, denoted by $\eta$, as shown in Figure \ref{square-triangles-3}. Since $x_1\prec\eta$, $z\prec \sigma$, $x\sim z$ and $|\sigma|=12$, by Proposition \ref{vert-in-sigma-conn-anosigma}, we obtain that $|\eta|\geq 8$ is impossible. Hence, $|\eta|\leq 7$, $\Pttn(x_1)=(3,3,3,|\eta|)$, which implies that $\Phi(x_1)>\frac{1}{12}$ and
\[
\tilde{\Phi}(z)\geq \Phi(z)+\frac{1}{2}\Phi(x_1)>\frac{1}{k}-\frac{1}{24}\geq\frac{1}{2k}.
\]

\textbf{Case 2.} $|\tau_2|=4$. Obviously, $\Pttn(z_1)=(3,3,4,k)$. Let $y_1$ (resp. $y_2$) be the vertex satisfying $y_1\sim z$ (resp. $y_2\sim z_1$) and $y_1\notin\partial\sigma$ (resp. $y_2\notin\partial\sigma$), as shown in Figure \ref{square-triangles-square}. By $|z|=4$ and $\Pttn(z)=(3,3,4,12)$, we obtain $\{y_1,z\}$ and $\{z,x_1\}$ are contained in a triangle. Then $y_1\sim x_1$. Similarly, $y_2\sim x_1$.
By $|x_1|=4$, $\{x_1,y_1\}$ and $\{x_1,y_2\}$ are contained in a face, denoted by $\eta$.

\begin{figure}[htbp]
 \begin{center}
   \begin{tikzpicture}
    \node at (0,0){\includegraphics[width=0.75\linewidth]{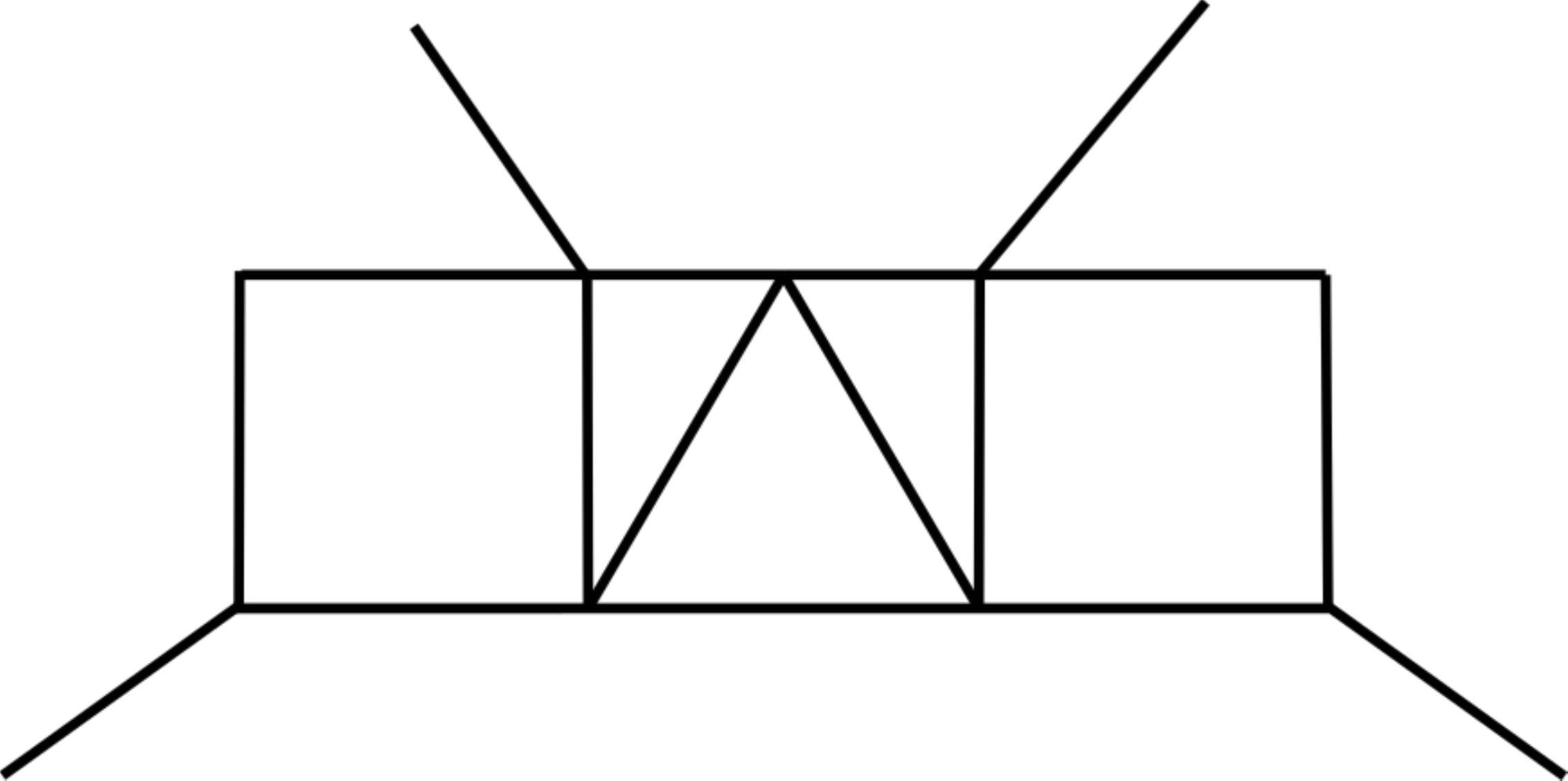}};
    \node at (0.15, 0.9){\Large $x_1$};
    \node at (-1.0, 0.95){\Large $y_1$};
    \node at (-1.2,  -1.6){\Large $z$};
    \node at (1.2,  -1.6){\Large $z_1$};
    \node at (1, 0.95){\Large $y_2$};
    \node at (0.6,   2.0){\Large $\eta$};
    \node at (-.1,  -.5){\Large $\tau_1$};
    \node at (2.2,  -.5){\Large $\tau_2$};
    \node at (-2.4,  -0.5){\Large $\omega$};
    \node at (-3.6,  -2.0){\Large $\sigma$};
   \end{tikzpicture}
  \caption{}\label{square-triangles-square}
 \end{center}
\end{figure}

Since $x_1\prec\eta$, $x_1\sim z$ and $z\prec \sigma$, $x_1\in U_1(\sigma)$ and $x_1\in\partial\eta$.  Proposition \ref{vert-in-sigma-conn-anosigma} implies that $|\eta|\leq 7$. Then $\Pttn(x_1)=(3,3,3,|\eta|)$ and $\Phi(x_1)=\frac{1}{|\eta|}\geq \frac{1}{7}$. Hence,
\[
\tilde{\Phi}(z)\geq \Phi(z)+\frac{1}{2}\Phi(x_1)\geq\frac{1}{k}-\frac{1}{12}+\frac{1}{14}>\frac{1}{2k}.
\]

Combining Case 1 and Case 2, we prove the lemma.

\end{proof}

\begin{proof}[Proof of Theorem \ref{12-gon-modcur-lower-bd}] For any $z\in \partial\sigma$, $\Pttn(z)=(3,3,3,k)$ or $(3,3,4,k)$. Applying Lemma \ref{Two-triang-adj} and Lemma \ref{one-triangle-one-square}, we obtain 
\[
\sum_{z\in\partial\sigma}\tilde{\Phi}(z)> k\times\frac{1}{2k}=\frac{1}{2}
\]
for $8\leq k\leq 11$ and
\[
\sum_{z\in\partial\sigma}\tilde{\Phi}(z)\geq k\times\frac{1}{2k}=\frac{1}{2}
\]
for $k=12$, the equality holds if and only if $\tilde{\Phi}(z)=\frac{1}{24}$ for each $z\in\partial\sigma$.
\end{proof}

Now we are ready to prove Theorem \ref{thm:main2}.
\begin{proof}[Proof of Theorem \ref{thm:main2}]
We argue by contradiction. Assume that $\sum_{k=8}^{12}F_k(\Gamma)>1$, that is there exist at least two faces with facial degrees between $8$ and $12$.
We denote by $\sigma_1$, $\sigma_2$ the two faces with $\sigma_i$ is $k_i$-gon and $8\leq k_i\leq 12$, $i=1,2$. We divide it into cases.

{\bf Case 1.} $8\leq k_1<12$ or $8\leq k_2<12$.
Applying Theorem \ref{thm:DeVosMohar}, Theorem \ref{sum-keep} and Theorem \ref{12-gon-modcur-lower-bd}, we obtain
\[
1\geq \Phi(\Gamma)=\tilde{\Phi}(\Gamma)=\sum_{z\in V}\tilde{\Phi}(z)\geq \sum_{z\in \partial\sigma_1}\tilde{\Phi}(z)+\sum_{z\in \partial\sigma_2}\tilde{\Phi}(z)>\frac{1}{2}+\frac{1}{2}=1.
\]
This yields a contradiction.

{\bf Case 2.} $k_1=k_2=12$.
Applying Theorem \ref{thm:DeVosMohar}, Theorem \ref{sum-keep} and Theorem \ref{12-gon-modcur-lower-bd}, we have
\[
1\geq \Phi(\Gamma)= \tilde{\Phi}(\Gamma)=\sum_{z\in V}\tilde{\Phi}(z)\geq \sum_{z\in \partial\sigma_1}\tilde{\Phi}(z)+\sum_{z\in \partial\sigma_2}\tilde{\Phi}(z)\geq\frac{1}{2}+\frac{1}{2}=1,
\]
then the equality holds, if and only if $\sum_{z\in\partial\sigma_i}\tilde{\Phi}(z)=\frac{1}{2}$ for $i=1, 2$.
According to Theorem \ref{12-gon-modcur-lower-bd}, that is
\begin{align}\begin{split}\label{12-gon-each-vertex-1/24}
\tilde{\Phi}(z)=\frac{1}{24},\ \forall z\in \partial\sigma_i, i=1,2.
\end{split}\end{align}
\begin{figure}[htbp]
 \begin{center}
   \begin{tikzpicture}
    \node at (0,0){\includegraphics[width=1\linewidth]{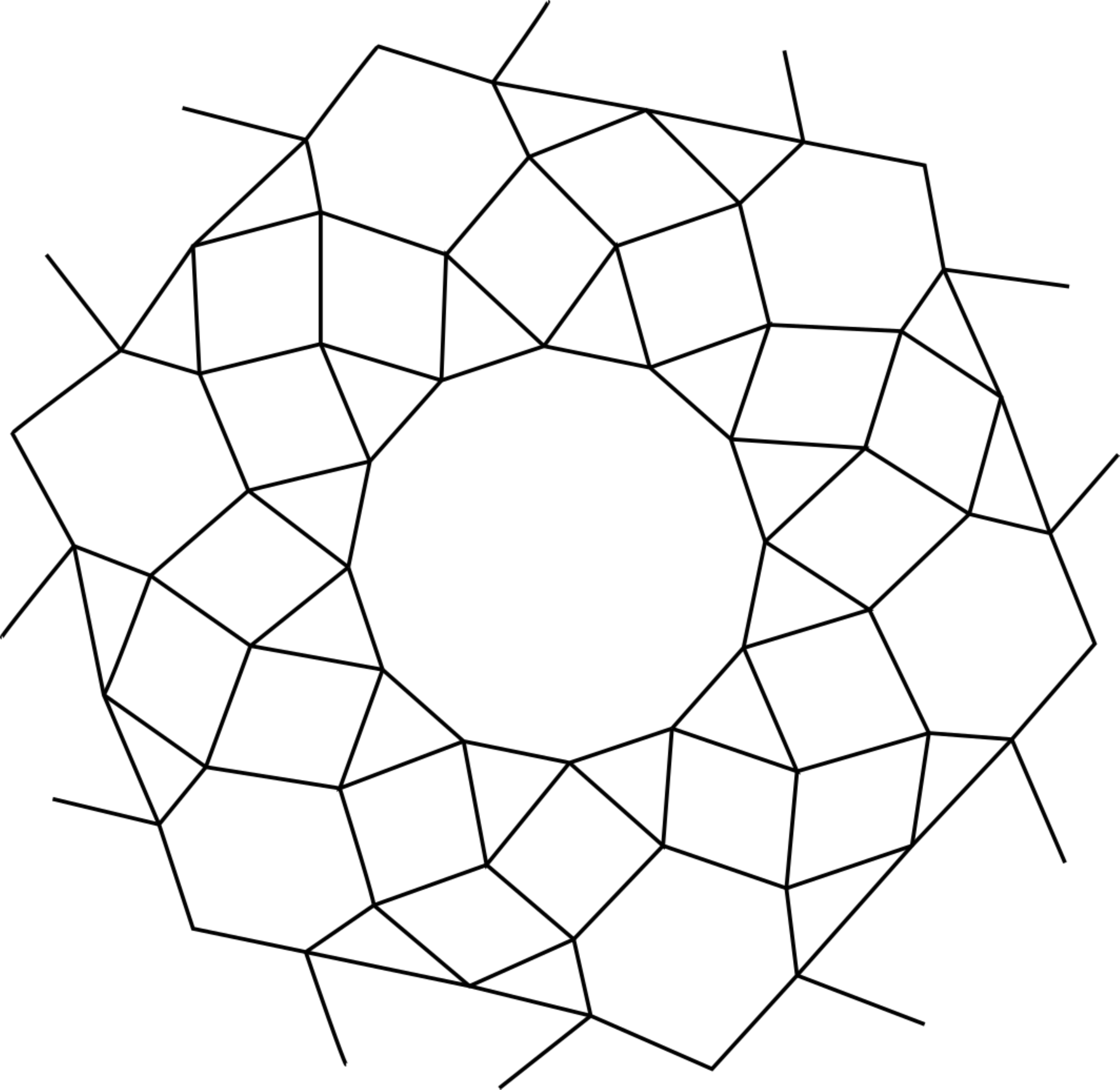}};
   \node at (-.1,   2.0){\small $z_1$};
   \node at (-.2,   3.2){\small $\omega_{12}$};
   \node at (1,   1.8){\small $z_2$};
   \node at (0.5,   2.5){\small $\tau_1$};
   \node at (0.7,   3.6){\small $x_1$};
   \node at (1.5,   2.9){\small $\omega_1$};
   \node at (1.8,   1.9){\small $\tau_2$};
   \node at (1.75,   1){\small $z_3$};
   \node at (2.1, -0.1){\small $z_4$};
   \node at (1.8, -1.1){\small $z_5$};
   \node at (1.0, -1.9){\small $z_6$};
   \node at (-0, -2.3){\small $z_7$};
   \node at (-1.0, -2){\small $z_8$};
   \node at (-1.8, -1.3){\small $z_9$};
   \node at (-2.1, -.2){\small $z_{10}$};
   \node at (-1.8,  .9){\small $z_{11}$};
   \node at (-1.1, 1.7){\small $z_{12}$};
   \node at (2.6,   0.8){\small $\tau_3$};
   \node at (2.6, -.6){\small $\tau_4$};
   \node at (2.,   -1.9){\small $\tau_5$};
   \node at (.8, -2.6){\small $\tau_6$};
   \node at (-0.6, -2.8){\small $\tau_7$};
   \node at (3.5,   0.8){\small $x_3$};
   \node at (-1.8, -2.1){\small $\tau_8$};
   \node at (-2.65, -.95){\small $\tau_9$};
   \node at (-2.65,  .4){\small $\tau_{10}$};
   \node at (-2, +1.7){\small $\tau_{11}$};
   \node at (-.9,  2.5){\small $\tau_{12}$};
    \node at (2.6 ,2.65){\small $x_2$};
    \node at (2.9,1.8){\small $\omega_2$};
    \node at (3.5, 0.1){\small $\omega_3$};
    \node at (3.4, -1){\small $x_4$};
    \node at (3.1, -1.7){\small $\omega_4$};
    \node at (2.9,   -2.3){\small $x_5$};
    \node at (2.,   -2.9){\small $\omega_5$};
    \node at (.1, -3.5){\small $\omega_6$};
    \node at ( 1.2, -3.7){\small $x_6$};
    \node at (-0.9, -3.85){\small $x_7$};
    \node at (-1.6, -3.2){\small $\omega_7$};
    \node at (-2.65, -2.9){\small $x_8$};
    \node at (-3.5, - .9){\small $x_9$};
    \node at (-3.5, -.3){\small $\omega_9$};
    \node at (-3, -1.9){\small $\omega_8$};
    \node at (-3.8,  .7){\small $x_{10}$};
    \node at (-3, 1.35){\small $\omega_{10}$};
    \node at (-2, 2.7){\small $\omega_{11}$};
     \node at (-2.9, +2.0){\small $x_{11}$};
    \node at (-1.4, 3.6){\small $x_{12}$};
    \node at (-.65,   4.5){\small $y_{12}$};
    \node at (0.7,   4.2){\small $\alpha_1$};
    \node at (3.2,3.5){\small $\alpha_2$};
    \node at (0.7,   5.6){\small $\eta_1$};
    \node at (1,   5.1){\small $p_1$};
    \node at (-.9,   5.5){\small $q_{12}$};
    \node at (2.3,   3.8){\small $y_1$};
    \node at (2.9 ,4.7){\small $q_1$};
    \node at (3.7 ,2.6){\small $y_2$};
    \node at (4.55,3.3){\small $q_2$};
    \node at (4.3,1.3){\small $\alpha_3$};
    \node at (5.3,1.6){\small $p_3$};
    \node at (5.7,1.8){\small $\eta_3$};
    \node at (4.7, .1){\small $y_3$};
    \node at (5.8, .1){\small $q_3$};
    \node at (4.8, -1){\small $\alpha_4$};
    \node at (4.4, -2){\small $y_4$};
    \node at (5.4, -2.2){\small $q_4$};
    \node at (4.2, -3.5){\small $p_5$};
    \node at (4.4, -4){\small $\eta_5$};
    \node at (2.8,   -3.6){\small $y_5$};
    \node at (3.3,   -3.1){\small $\alpha_5$};
    \node at (1.6,   -4.6){\small $\alpha_6$};
    \node at (2.8,   -5.1){\small $q_5$};
    \node at (.4, -4.5){\small $y_6$};
    \node at (.4, -5.5){\small $q_6$};
    \node at (-0.9, -5.2){\small $p_7$};
    \node at (-1.6, -5.4){\small $\eta_7$};
    \node at (-2.4, -4.1){\small $y_7$};
    \node at (-3, -4.8){\small $q_7$};
    \node at (-3.3, -3.6){\small $\alpha_8$};
    \node at (-0.9, -4.3){\small $\alpha_7$};
    \node at (-3.9,  -2.7){\small $y_8$};
    \node at (-4.7,  -3.3){\small $q_8$};
    \node at (-4.3, -1.5){\small $\alpha_9$};
    \node at (-4.6,  1 ){\small $\alpha_{10}$};
    \node at (-5.4, -1.7){\small $p_9$};
    \node at (-4.6, -.1){\small $y_9$};
    \node at (-5.3,  .1){\small $q_9$};
    \node at (-5.8, -2.0){\small $\eta_9$};
    \node at (-3.9,  -2.7){\small $y_8$};
    \node at (-1.65,   4.5){\small $\alpha_{12}$};
    \node at (-2.5,   4.5){\small $q_{11}$};
    \node at (-2.4,   3.5){\small $y_{11}$};
    \node at (-4.3,   3.5){\small $p_{11}$};
    \node at (-4.8,   4.0){\small $\eta_{11}$};
    \node at (-3.5, 2.8){\small $\alpha_{11}$};
     \node at (-4.3, +1.8){\small $y_{10}$};
     \node at (-4.9, +2){\small $q_{10}$};
    \node at (-0.0,   0.0){\small $\sigma_1$};
    \end{tikzpicture}
  \caption{}\label{figure14}
 \end{center}
\end{figure}

Now we consider $\sigma_1$.  Combining (\ref{12-gon-each-vertex-1/24}) with Lemma \ref{Two-triang-adj}, we obtain that there exist $12$ triangles, $\{\tau_i\}_{i=1}^{12}$, ordered clockwise along $\sigma_1$, such that $\tau_{12}$ is $\sigma$-adjacent to $\tau_1$ and $\tau_i$ is $\sigma_1$-adjacent to $\tau_{i+1}$ for $1\leq i\leq 11$, and
\begin{align}\begin{split}\label{pattern-x_i}
\Pttn(x_i)=
\begin{cases}
(3,4,4,4), & i \ \textrm{is odd}\\
(3,4,4,6), & i \ \textrm{is even},
\end{cases}
\end{split}\end{align}
where $x_i=\partial\tau_i\setminus \partial\sigma_1$, as shown in Figure \ref{figure14}. Let $z_i=\d\tau_i\cap\d\tau_{i+1}$ for $1\leq i\leq 11$ and $z_{12}=\tau_{12}\cap\tau_1$. Clearly, $z_i\prec\sigma$. For $1\leq i\leq 11$, by $|z_i|=4$, $\{z_i, x_i\}$ and $\{z_i,x_{i+1}\}$ are contained in a face, denoted by $\omega_i$.
Since $z_i\prec\omega_i$ and $z_i\prec\sigma_1$, $|\omega_i|=3$ or $4$. Using (\ref{pattern-x_i}), $x_i\prec\omega_i$ and $|\tau_i|=3$ for $1\leq i\leq 12$, we have $|\omega_i|=4$. By $|z_{12}|=4$, $\{z_{12}, x_{12}\}$ and $\{z_{12}, x_1\}$ are contained in a face, denoted by $\omega_{12}.$ Similarly, $|\omega_{12}|=4$.

Let $y_i$ be a vertex such that $y_i\sim x_i$, $y_i\sim x_{i+1}$ and $y_i\nprec\sigma_1$ for $i=1,\cdots, 11$ and $y_{12}$ be a vertex such that $y_{12}\sim x_{12}$ and $y_{12}\sim x_1$. Let $\alpha_i$ be a face containing $\{x_i, y_{i-1}\}$ and $\{x_i, y_i\}$ for $2\leq i\leq 12$ and $\alpha_1$ be a face containing $\{x_1,y_1\}$ and $\{x_1,y_{12}\}$. Hence, by $|x_i|=4$, we have
\[
x_1=\tau_1\cap\omega_{12}\cap\omega_1\cap\alpha_1,\ \ x_i=\tau_i\cap\omega_{i-1}\cap\omega_i\cap\alpha_i, \ \ i=2,\cdots, 12.
\]
By $|\tau_i|=3$ and $|\omega_i|=4$, together with (\ref{pattern-x_i}), we obtain
\begin{align}\begin{split}\label{alpha_i-degree}
|\alpha_i|=
\begin{cases}
4, & i \ \textrm{is odd},\\
6, & i \ \textrm{is even}.
\end{cases}
\end{split}\end{align}
For $|x_1|=4$, $\{x_1,y_1\}$ and $\{x_1,z_2\}$ are contained in a face. Combining it with $\{x_1,z_1\}\prec \omega_1$,  we have $\{x_1,y_1\}\prec\omega_1$ and $y_1\prec\omega_1$. With the same arguments as above, we obtain $y_i\prec \omega_i$ for $i=1,\cdots, 12$, as shown in Figure \ref{figure14}. Noting that $y_i\prec(\alpha_i\cap\alpha_{i+1})$ for $i=1,\cdots,11$ and $y_{12}\prec\alpha_{12}\cap\alpha_1$, we have
\[
y_{12}\prec\omega_{12}\cap \alpha_{12}\cap \alpha_1, \ \ y_i\prec\omega_i\cap \alpha_i\cap \alpha_{i+1},\ \ i=1,\cdots,11.
\]
By $|\omega_i|=4$, (\ref{alpha_i-degree}) and $\Phi(y_i)\geq 0$,
\begin{align}\begin{split}\label{pt-y_i}
\Pttn(y_i)=(3,4,4,6), \ i=1,\cdots,12.
\end{split}\end{align}
Let $p_i$ be a vertex such that $p_i\sim y_{i-1}$ and $p_i\sim y_i$ for $i=3,5,7,9,11$, and $p_1$ be a vertex such that $p_1\sim y_{12}$ and $p_1\sim y_1$. By $|y_{12}|=4$ and (\ref{pt-y_i}), there exists a
vertex $q_{12}$ satisfying $q_{12}\prec \alpha_{12}$ and $q_{12}\neq x_{12}$ such that $q_{12}\sim y_{12}$ and $q_{12}\sim p_1$. That is, there exists a triangle with boundary vertices $y_{12}, p_1, q_{12}$. With the same arguments as above, we obtain that there exists a vertex $q_i$ such that $q_i\prec \alpha_{i+1}$, $q_i\sim y_i$ and $q_i\sim p_i$ for $i$ odd and there exists a vertex $q_i$ such that $q_i\prec \alpha_i$, $q_i\sim y_i$ and $q_i\sim p_{i+1}$ for $i$ even, $1\leq i\leq 11$. Let $\eta_i$ be a face containing $\{p_i,q_{i-1}\}$ and $\{p_i, q_i\}$ for $i=3,5,7,9,11$ and $\eta_1$ be a face containing $\{p_1,q_{12}\}$ and $\{p_1,q_1\}$.




For $\eta_1$, since $q_1\prec\eta_1$, $q_1\prec\alpha_2$, $|\alpha_2|=6$ and $\Phi(q_1)\geq 0$, $|\eta_1|\leq 6$. For $p_1$, since $p_1\sim y_1\sim q_1\sim p_1$, $p_1\sim y_{12}\sim q_{12}\sim p_1$, and $p_1\prec \alpha_1$, we obtain that $p_1$ is contained in two triangles and a square $\alpha_2$, which yields that $\Pttn(p_1)=(3,3,4,|\eta_1|)$ and $\Phi(p_1)\geq \frac{1}{12}$.
Using the same arguments for $\eta_1$ and $p_1$ for $i=3,5,7,9,11$, we obtain that $|\eta_i|\leq 6$, $\Pttn(p_i)=(3,3,4,|\eta_i|)$ and $\Phi(p_i)\geq \frac{1}{12}$.

We claim that $|\eta_{2i-1}|=6$ for $1\leq i\leq 6$. Otherwise, there exists some $\eta_i$, say $|\eta_1|,$ such that $|\eta_1|<6$. So that $\Phi(p_1)>\frac{1}{12}$, which yields that
\[
\sum_{j=1}^6\Phi(p_{2j-1})>\frac{1}{2}.
\]
By (\ref{pattern-x_i}), we obtain $\sum^{12}_{i=1}\Phi(x_i)=\frac{1}{2}$, and
\[
\sum^{12}_{i=1}\Phi(x_i)+\sum_{i=1}^6\Phi(p_{2j-1})>1.
\]
This contradicts Theorem \ref{thm:DeVosMohar}. This proves the claim.

Hence, $|\eta_i|=6$ for $1\leq i\leq 6$ and
\begin{align}\begin{split}\label{sum-curvature}
\sum^{12}_{i=1}\Phi(x_i)+\sum_{i=1}^6\Phi(p_i)=1.
\end{split}\end{align}
Applying Theorem \ref{thm:DeVosMohar}, we obtain the remaining vertices, except $\{x_{2i-1}\}_{i=1}^6$ and $\{p_{2j-1}\}_{j=1}^6$, has vanishing curvature.

With the same argument as $\sigma_1$, we obtain that $\sigma_2$ has the same structure as $\sigma_1$. Theorem \ref{nexist-both-adj} implies that $x_i\nsim x'$ for any $x'\prec \sigma_2$ with $i=1,\cdots, 12$. Hence, there are vertices, $z'$ and $z$, such that $z'\sim z$, $z\prec \sigma_2$ and $\Phi(z')>0$. Combining it with (\ref{sum-curvature}), we obtain $\Phi(\Gamma)>1$, this contradicts  Theorem \ref{thm:DeVosMohar}.
This proves the theorem.
\end{proof}

\section*{Acknowledgements}
Y. A. is supported by JSPS KAKENHI Grant Number JP16K05247. B.H. is supported by NSFC, no.11831004 and no. 11826031. Y.S. is supported by NSFC, grant no. 11771083, NSF of Fujian Province through grants no. 2017J01556 and NSF of Fuzhou University through grant no. GXRC-18035. L. W. is supported by NSFC, no. 11671141 and China Postdoctoral Science Foundation, no. 2019M651332.

\bibliography{Mcurvature1}
\bibliographystyle{alpha}

\end{document}